\documentclass[a4paper,11pt]{article}
\usepackage{latexsym,amssymb,enumerate,amsmath,epsfig,amsthm,authblk,dsfont}
\usepackage[margin=1in]{geometry}
\usepackage{setspace,color}
\usepackage{tikz,widetext}
\usepackage{appendix}
\usepackage{floatrow}
\usepackage{subfig,wrapfig}
\usepackage{graphicx}
\usepackage[ruled]{algorithm2e}
\usepackage{epstopdf}
\usepackage{grffile}
\usepackage[shortlabels]{enumitem}
\usepackage{makecell}
\usepackage{caption,hyperref}
\allowdisplaybreaks

\usepackage{lineno}

\newcommand{\bx}{\mathbf{x}}
\newcommand{\ba}{\mathbf{a}}
\newcommand{\be}{\mathbf{e}}
\newcommand{\bw}{\mathbf{w}}
\newcommand{\bv}{\mathbf{v}}
\newcommand{\bz}{\mathbf{z}}
\newcommand{\br}{\mathbf{r}}

\newcommand{\bp}{\mathbf{p}}
\newcommand{\bu}{\mathbf{u}}

\newcommand{\commentout}[1]{}

\newcommand{\RR}{\mathbb{R}}
\newcommand{\NN}{\mathbb{N}}
\newcommand{\EE}{\mathbb{E}}
\newcommand{\PP}{\mathbb{P}}

\newcommand{\mone}{\mathds{1}}

\newcommand{\hf}{\widehat{f}}
\newcommand{\hg}{\widehat{g}}

\newcommand{\hPhi}{\widehat{\Phi}}
\newcommand{\tPhi}{\widetilde{\Phi}}
\newcommand{\hV}{\widehat{V}}
\newcommand{\hG}{\widehat{G}}

\newcommand{\tH}{\widetilde{H}}

\newcommand{\hbphi}{\widehat{\boldsymbol\phi}}
\newcommand{\bphi}{{\boldsymbol\phi}}

\newcommand{\tf}{\widetilde{f}}
\newcommand{\tg}{\widetilde{g}}

\newcommand{\tbx}{\widetilde\bx}

\newcommand{\calS}{\mathcal{S}}
\newcommand{\calA}{\mathcal{A}}
\newcommand{\calF}{\mathcal{F}}
\newcommand{\calB}{\mathcal{B}}

\newcommand{\tr}{\mathrm{tr}}
\newcommand{\var}{\mathrm{var}}
\newcommand{\cov}{\mathrm{Cov}}
\newcommand{\proj}{{\rm Proj}}

\newcommand{\supp}{{\rm supp}}

\newcommand{\Vfij}{V_f(\bx_i,\bx_j)}

\newcommand{\Vfijgamma}{V_f(\bx_i,\bx_j,r)}

\newcommand{\Vyijgamma}{V_y(\bx_i,\bx_j,r)}
\newcommand{\hVyijgamma}{\widehat{V}_y(\bx_i,\bx_j,r)}

\newcommand{\hn}{\widehat{n}}

\newcommand{\bnoise}{b^2_{\mathrm{bias}}}

\newcommand{\Vol}{{\rm Vol}}

\newtheorem{thm}{Theorem}
\newtheorem{lem}{Lemma}
\newtheorem{prop}{Proposition}
\newtheorem{coro}{Corollary}
\newtheorem{defi}{Definition}
\newtheorem{assum}{Assumption}
\newtheorem{condi}{Condition}

\newtheorem{example}{Example}
\newtheorem*{example*}{Example}

\DeclareMathOperator*{\argmin}{arg\,min}

\begin{document}
\title{Learning functions varying along a central subspace}
\date{}
\author{
Hao Liu\thanks{Department of Mathematics, Hong Kong Baptist University, Kowloon Tong, Hong Kong (Email: {haoliu@hkbu.edu.hk}).}\ \ \ \
Wenjing Liao\thanks{School of Mathematics, Georgia Institute of Technology, 686 Cherry Street, Atlanta, GA 30332 USA (Email: {wliao60@gatech.edu}). This research is supported by NSF DMS 1818751 and DMS 2012652.}

}
\maketitle

\begin{abstract}

Many functions of interest are in a high-dimensional space but exhibit low-dimensional structures. This paper studies regression of a $s$-H\"{o}lder function $f$ in $\mathds{R}^D$ which varies along a central subspace of dimension $d$ while $d\ll D$. A direct approximation of $f$ in $\mathds{R}^D$ with an $\varepsilon$ accuracy requires the number of samples $n$ in the order of $\varepsilon^{-(2s+D)/s}$. 
In this paper, we analyze the Generalized Contour Regression (GCR) algorithm for the estimation of the central subspace and use piecewise polynomials for function approximation.  GCR is among the best estimators for the central subspace, but its sample complexity is an open question. We prove that GCR leads to a mean squared estimation error of $O(n^{-1})$ for the central subspace, if a variance quantity is exactly known. The estimation error of this variance quantity is also given in this paper. 
The mean squared regression error of $f$ is proved to be in the order of
$\left(n/\log n\right)^{-\frac{2s}{2s+d}}$
where the exponent depends on the dimension of the central subspace $d$ instead of the ambient space $D$. This result demonstrates that GCR is effective in learning the low-dimensional central subspace. We also propose a modified GCR with improved efficiency.
The convergence rate is validated through several numerical experiments.


\end{abstract}


\section{Introduction}

A vast majority of statistical inference and machine learning problems can be modeled as regression, where  the goal is to estimate an unknown function from a finite number of training samples.
Nowadays, new challenges are introduced to regression and prediction due to the rise of high-dimensional data in many fields of contemporary science.
The well-known curse of dimensionality implies that, in order to achieve a fixed accuracy in prediction, the number of training data must grow exponentially with respect to the data dimension, which is beyond practical applications.

Fortunately, functions of interest in applications often exhibit low-dimensional structures. In many situations, the response may depend on few variables, or a low-dimensional subspace. For example, in Bioinformatics, the cDNA microarray data have thousands of dimension but an effective classification of tumor types only depends on a subspace of dimension one or two \cite{Bura2003}.
In engineering, the coefficients of certain elliptic partial differential equations are parameterized by many variables but has a small effective dimension \cite{constantine2014active}.
In photovoltaic industry, there are five input parameters in the single-diode solar cell model while the maximum power output only depends on a linear combination of these five parameters \cite{constantine2015discovering}. Similar low-dimensional models appear in optimization \cite{gill1981practical}, optimal control \cite{zhou1996robust}, uncertainty quantification \cite{Babuska2004}, text classification \cite{kim2005dimension} and biomedicine \cite{Pfeiffer2008,Pfeiffer2012}.

These applications motivate us to consider the regression model
\begin{equation}
y_i = f(\bx_i) +\xi_i  = g(\Phi^\top \bx_i) + \xi_i,
\label{eqyi}
\end{equation}
where $\bx \in \RR^D$, $\Phi \in \RR^{D \times d}$, $g: \RR^d \rightarrow \RR$ and $\xi_i \in \RR$.  The columns of $\Phi$ consist of an orthonormal basis of a $d$-dimensional subspace, i.e., $\Phi^\top \Phi = I_{d}$. The random variable $\xi_i$ models noise, which is independent of the $\bx_i$'s. In this model, the function $f$ is defined on $ \RR^D$ but only depends on the projection of $\bx$ to the subspace spanned by $\Phi$, denoted by $\calS_\Phi$. Let $\calS_{\Phi}$ be such a space of minimum dimension: $g(\bv^\top\bz)$ is not a constant function for any $\bv\in \calS_{\Phi}$. Then $\calS_{\Phi}$ is called the central subspace \cite{cook2002dimension}, or the Effective Dimension-Reduction subspace \cite{li1991sliced,dalalyan2008new,xia2009adaptive}. In other words, the function $f$ only varies along the central subspace, and remains the same along any orthogonal direction of the central subspace. This model is also called the single-index model for $d=1$ and the multi-index model for $d\ge 2$.




The goal of regression is to estimate the function $f$ from $n$ samples of data $\{(\bx_i,y_i)\}_{i=1}^n$ where the $\bx_i$'s are independently drawn from a probability measure $\rho$ in $\RR^D$. Given data and the a prior knowledge that $f$ varies along a $d$-dimensional central subspace, we aim at constructing an empirical estimator  $\hf: \RR^{D } \rightarrow \RR$ and studying the Mean Squared Error (MSE) $\EE_{\{(\bx_i,y_i)\}_{i=1}^n} \|\hf - f\|_{L^2(\rho)}$.

The central interest of this problem is to estimate the central subspace $\Phi$. A class of methods is based on the gradient $\nabla f(\bx)=\Phi\nabla g(\Phi^\top\bx)$.
When $d = 1$, $\nabla f$ is proportional to the $\Phi$ direction, which allows one to estimate $\Phi$ from average of the empirical gradients \cite{Powell1989,Hardle1989,hardle1993sensitive}. When $d \ge 2$, the covariance matrix of $\nabla f$ is given by $\Phi \int \nabla g(\Phi^\top\bx) [\nabla g(\Phi^\top\bx) ]^\top \rho(d\bx) \Phi^\top$, which gives an estimate of $\Phi$ as the eigenspace associated with the top $d$ eigenvalues of this covariance matrix \cite{Hristache2001,constantine2014active}. If the gradient can be accurately estimated, these gradient-based methods lead to a $\sqrt{n}$-consistent estimation of $\Phi$ \cite{hardle1993sensitive,constantine2014active}. In general, the gradient estimation in $\RR^D$ requires an exponentially large number of samples in dimension $D$.


The single-index model with $d=1$ has been extensively studied in literature. In this case, $g: \RR\rightarrow \RR$ is called the link function. The minimax mean squared regression error while the link function belongs to the $s$-H\"{o}lder class was proved to be $O(n^{-2s/(2s+1)})$ \cite{gaiffas2007optimal,lepski2014adaptive,chesneau2007regression}.
These results demonstrate that the optimal algorithm can automatically adapt to the central subspace of dimension $1$. For estimation,  a lot of methods based on non-convex optimization have been proposed \cite{lepski2014adaptive,gyorfi2006distribution,ichimura1993semiparametric,xia2002adaptive} while achieving the global minimum is not guaranteed. 
{It was proved in \cite[Chapter 22]{gyorfi2006distribution} that minimizing the empirical risk over the central subspace of dimension $1$ and approximating $g$ with piecewise polynomials simultaneously give rise to the MSE of $O((n/\log n)^{-2s/(2s+1)})$.} The main challenge of this approach is to obtain the global minimum due to non-convexity in the optimization.
In the context of regression with point queries, an adaptive query algorithm was proposed for the single-index model in \cite{Cohen2012} with performance guarantees.  This algorithm was generalized to the multi-index model in \cite{Fornasier2012}. 
In the standard regression setting, adaptive queries are not allowed, and the samples are usually given before learning starts.

Estimating the central subspace is related with sufficient dimension reduction in statistics. A class of methods related with inverse regression has been developed to estimate the central subspace $\Phi$ (see \cite{Ma2013,li2018sufficient} for a comprehensive review).
Sliced Inverse Regression (SIR) \cite{Li1991} is the first and most well known method. The term \lq\lq inverse regression\rq\rq\ refers to the conditional expectation $\EE(\bx|y)$. In SIR, the central subspace is estimated as the eigenspace associated with the top $d$  eigenvalues of $\cov(\EE(\bx|y)-\EE(\bx))$.
Similar techniques include kernel inverse regression \cite{Zhu1996}, parametric inverse regression \cite{Bura2001} and canonical correlation estimator \cite{fung2002dimension}. These methods are referred as first-order methods since the first-order statistical information is utilized, such as the conditional mean. Their performance crucially depends  on the function $g$ and the distribution of data.
If $g$ is symmetric about $\mathbf{0}$ along some directions, then those directions can not be recovered by the first-order methods.
This issue about symmetry is better addressed by second-order methods which utilize the variance and covariance information of data. Popular second-order methods include sliced
inverse variance estimation \cite{cook1991discussion}, sliced average variance estimation (SAVE)  \cite{cook1991sliced,dennis2000save},
average variance estimation (MAVE) \cite{xia2002adaptive}, contour regression \cite{li2005contour}, directional regression \cite{li2007directional}, principal Hessian directions \cite{li1992principal}, a hybrid  SIR and SAVE \cite{zhu2007hybrid}, etc. 
New methods are also developed in a multiscale framework \cite{lanteri2020conditional,klock2021estimating}. For the single-index model, Smallest Vector Regression (SVR) is proposed in \cite{lanteri2020conditional} in a multiscale framework, which combines the idea of SIR and SAVE. A performance analysis is provided for the index estimation and regression, while the assumptions are favorable to monotonic $g$.

This paper focuses on a second-order method called Generalized Contour Regression (GCR) introduced by Li, Zha and Chiaromonte \cite{li2005contour}. In comparison with SIR and SAVE mentioned above, GCR has advantages in the case where $g$ is not monotonic. Empirical experiments have demonstrated the success of GCR for the estimation of the central subspace, but its sample complexity is still not well understood yet. 
In this paper, we analyze the error behavior of the regression scheme which consists of GCR for the central subspace estimation and piecewise polynomial approximation of $g$. We also propose a modified GCR with improved efficiency.
Our contributions are:

\begin{enumerate}[(i)]
  \setlength\itemsep{-0.05cm}
	\item We prove that GCR estimates the central subspace with a mean squared estimation error of $O(n^{-1})$, if a variance quantity is exactly known. The estimation error of this variance quantity is also given in this paper.

	\item Our regression scheme gives rise to the MSE of $f$ in the order of $\left(n/\log n\right)^{-\frac{2s}{2s+d}}$. This demonstrates that the MSE decays exponentially with an exponent depending on the dimension of the central subspace $d$, instead of the ambient dimension $D$.
	\item Our modified GCR improves over the original GCR in efficiency. Numerical experiments demonstrate that the modified GCR has the same convergence rate as that  of the original GCR in our statistical theory.

\end{enumerate}
This paper is organized as follows: We first introduce SCR and GCR in Section \ref{sec.gcr}. Our regression scheme and main results are stated in Section \ref{sec.MainResults}. Numerical experiments are provided in Section \ref{sec.numerical} to validate our theory. We present proofs  in Section \ref{sec.mainproof} and  conclude in Section \ref{sec.conclusion}.

We use lowercase bold letters and capital letters  to denote vectors and matrices respectively. $\|\bx\|$ is the Euclidean norm of the vector $\bx$ and $\|A\|$ is the spectral norm of the matrix $A$. For two square matrices $A$ and $B$ of the same size, $A\preceq B$ means $B-A$ is positive semi-definite.
We use $\NN_0$ to denote nonnegative integers. For a function $g$, $\supp(g)$ denotes the support of $g$. For $x\in \RR$, $\lfloor x\rfloor$ denotes the largest integer that is less than or equal to $x$ and $\lceil x\rceil$ denotes the smallest integer that is greater than or equal to $x$. Throughout the paper, we use $\calS_{\Phi}$ and $\calS^\perp_{\Phi}$ to denote the central subspace and its orthogonal complement, respectively.


\section{Central subspace and contour regression}
\label{sec.gcr}
The model in (\ref{eqyi}) has an ambiguity since the solutions are not unique. The columns of $\Phi$ form an orthonormal basis of the central subspace, while the choice of orthonormal basis is not unique. This gives rise to an ambiguity between the basis $\Phi$ and the function $g$, which can be characterized by the following lemma:

\begin{lem}
  Consider the model in (\ref{eqyi}) where the columns of $\Phi$ form an orthonormal basis of the central subspace. For any orthogonal matrix $Q\in \RR^{d\times d}$, let $\widetilde{\Phi}=\Phi Q$ and $\widetilde{g}(\bz)=g(Q\bz)$ for any $\bz\in \mathds{R}^d$. Then the columns of $\widetilde{\Phi}$ form another orthonormal basis of the central subspace and $g(\Phi^\top\bx)=\widetilde{g}(\widetilde{\Phi}^\top\bx)$.
  \label{lemma.differentBasis}
\end{lem}

Lemma \ref{lemma.differentBasis} shows that the representation of $f$ is not unique. If another set of orthonormal basis is picked, the function $g$ changes accordingly. 
In this paper, we aim to recover one set of orthonormal basis $\hPhi$ and the corresponding function $\hg$.
The subspace estimation error and regression error are represented by $\|\proj_{\hPhi} -\proj_{\Phi} \|$ and
\begin{equation}
\|\hf - f\|_{L^2(\rho)} := 
\int |\hg(\hPhi^\top \bx ) - g(\Phi^\top \bx)|^2 \rho(d\bx)
\label{ferrorinf}
\end{equation}
respectively, which are invariant to the choice of orthonormal basis.
 After a change of basis, 
we can assume that the columns of $\hPhi$ are carefully chosen such that 
\begin{align}
    \|\hPhi-\Phi\|^2\leq \sqrt 2 \|\proj_{\hPhi}-\proj_{\Phi}\|
    \label{phiproj}
  \end{align}
  according to \cite[Theorem 2.2.1]{chen2020spectral}.



\subsection{Simple contour regression}
We start with Simple Contour Regression (SCR) \cite{li2005contour} which utilizes the fact that the contour directions of $f$ are orthogonal to the central subspace. 
Let $\alpha>0$ and define the following conditional covariance matrix:
$$
K(\alpha)=\EE\left[ (\tilde{\bx}-\bx)(\tilde{\bx}-\bx)^\top| |\tilde{y}-y|\leq  \alpha\right],
$$
where $(\tilde{\bx},\tilde{y})$ and $(\bx,y)$ are two independent samples.

When $d=1$ and $g$ is a monotonic function, we expect many of the $(\tilde\bx,\bx)$ pairs satisfying $|\tilde y - y| \le \alpha$ will have $|\Phi^\top \tilde\bx - \Phi^\top \bx|$ small while $|\bw^\top \tilde\bx - \bw^\top \bx|$ can be arbitrarily large for any $\bw \in \calS_\Phi^\perp$. 
In this case, most of the $\tilde \bx -\bx$ directions satisfying $|\tilde y - y| \le \alpha$ are aligned with $\calS_\Phi^\perp$.

For example, in Figure \ref{fig.demo.exp}(a), $f(\bx)=e^{x_1}$ with $\bx=(x_1,x_2)\in[-2,2]\times[-2,2]$. This function can be expressed as Model \eqref{eqyi} with $\Phi=[1,0]^\top$, ${\bf w}=[0,1]^\top$, and $g(z) = e^z$. Consider two inputs $\mathbf{p}=(p_1,p_2)$ and $\widetilde{\mathbf{p}}=(\widetilde{p}_1,\widetilde{p}_2)$ satisfying $|f(\mathbf{p})-f(\widetilde{\mathbf{p}})|\leq\alpha$. Since $e^{-2}|p_1-\widetilde{p}_1|\leq |e^{p_1}-e^{\widetilde{p}_1}|=|f(\mathbf{p})-f(\widetilde{\mathbf{p}})|$, the condition $|f(\mathbf{p})-f(\widetilde{\mathbf{p}})|\leq \alpha$ implies $|\Phi^\top(\widetilde{\mathbf{p}}-\mathbf{p})|=|p_1-\widetilde{p}_1|<e^2\alpha$. On the other hand, $|{\bf w}^\top(\widetilde{\mathbf{p}}-\mathbf{p})|=|p_2-\widetilde{p}_2|\in[0,4]$, which is independent of $\alpha$. In Figure \ref{fig.demo.exp}(b), 200 samples are uniformly drawn from $[-0.5,0.5]\times[-0.5,0.5]$. The $(\bx,\tbx)$ pair is connected if $|f(\bx)-f(\tbx)|\leq0.01$. We observe that, almost all connections by SCR in Figure \ref{fig.demo.exp}(b) are aligned with the $\bw$ direction.

SCR estimates the central subspace from the smallest $d$ eigenvectors of $K(\alpha)$. The success of SCR is guaranteed with the following condition:
\begin{condi}
\label{assum-scr}
There exists $\alpha_c>0$ such that for any $\alpha \in(0, \alpha_c)$ and unit vectors $\bv\in \calS_\Phi$, $\bw \in \calS_\Phi^\perp$, 
the following holds
  $$
  \var\left[ {\bf w}^\top(\tbx-\bx)| |\tilde{y}-y|\leq  \alpha\right]>\var\left[ {\bf v}^\top(\tbx-\bx)| |\tilde{y}-y|\leq \alpha\right].
  $$
\end{condi}
We define an elliptical distribution \cite{li2018sufficient} as
\begin{defi}[Elliptical distribution]
\label{def.elliptic}
Let 
$\rho$ be a probability measure in $\RR^D$ with density function $h$. We say $\rho$ has an elliptical distribution if there exists a positive-definite matrix $A \in \RR^{D\times D}$ such that 
$$
h(\bx)=\kappa(\bx^\top A \bx)
$$
for some function $\kappa:\RR\rightarrow \RR$. When $A=I$, we say $\rho$ has a spherical distribution. 
\end{defi}



It is known that if $\bx$ has a spherical distribution, then $\EE (\bx)=\mathbf{0}$ and $\cov (\bx)=aI$ for some $a>0$, see \cite[Theorem 2.5 and 2.7]{fang2018symmetric}. Condition \ref{assum-scr} together with an spherical distribution of $\bx$ guarantee that the eigenvectors associated with the smallest $d$ eigenvalues of $K(\alpha)$ span $\calS_\Phi$ \cite[Theorem 2.1]{li2005contour}.

\begin{figure}[t!]
  \subfloat[$f(x_1,x_2)=e^{x_1}$]{\includegraphics[width=0.33\textwidth]{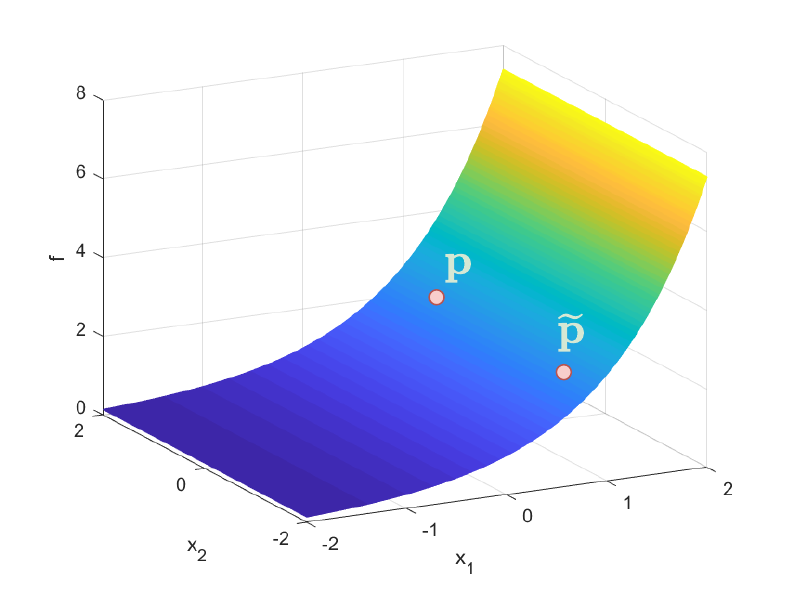}}
  \subfloat[SCR]{\includegraphics[width=0.33\textwidth]{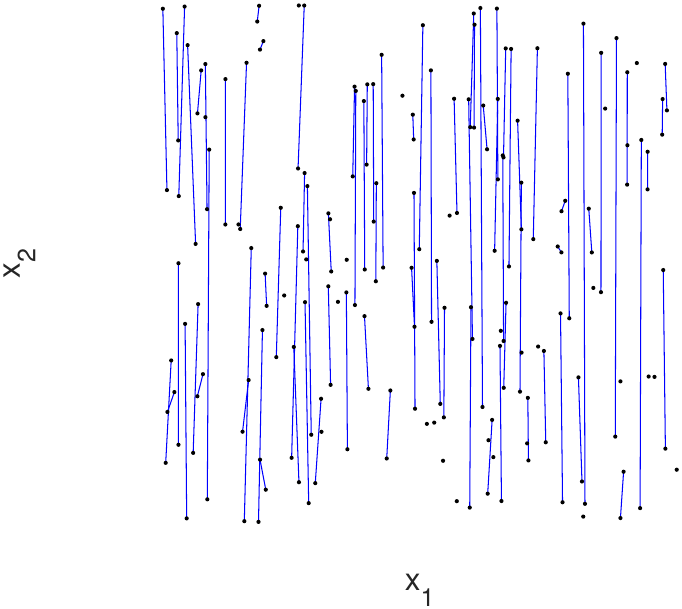}}
  \subfloat[GCR]{\includegraphics[width=0.33\textwidth]{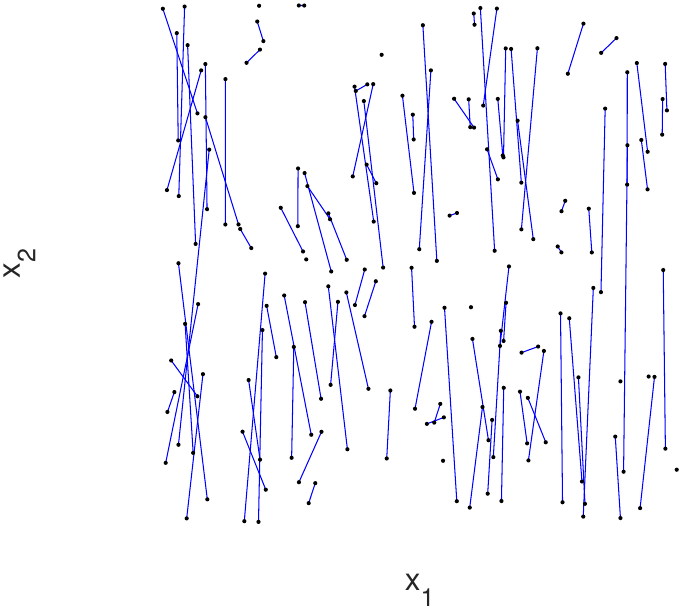}}
  \caption{(a) Function $f(x_1,x_2)=e^{x_1}$. When 200 samples are uniformly drawn in $[-0.5,0.5]\times[-0.5,0.5]$, (b) and (c) display the connected $(\bx,\tbx)$ pairs by SCR (b) and GCR (c), respectively. An $(\bx,\tbx)$ pair is connected by SCR in (b) if $|f(\bx)-f(\tbx)|\leq0.01$ and by GCR in (c) if $\hV_y(\bx,\tbx;r)\leq0.001$ with $r=0.01$. Here $\hV_y(\bx,\tbx;r)$ is an empirical estimator of $V_f(\bx,\tbx)$ to be defined in \eqref{eqhatvy}.}
  \label{fig.demo.exp}
\end{figure}

\begin{figure}[t!]
  \subfloat[$f(x_1,x_2)=x_1^2$]{\includegraphics[width=0.33\textwidth]{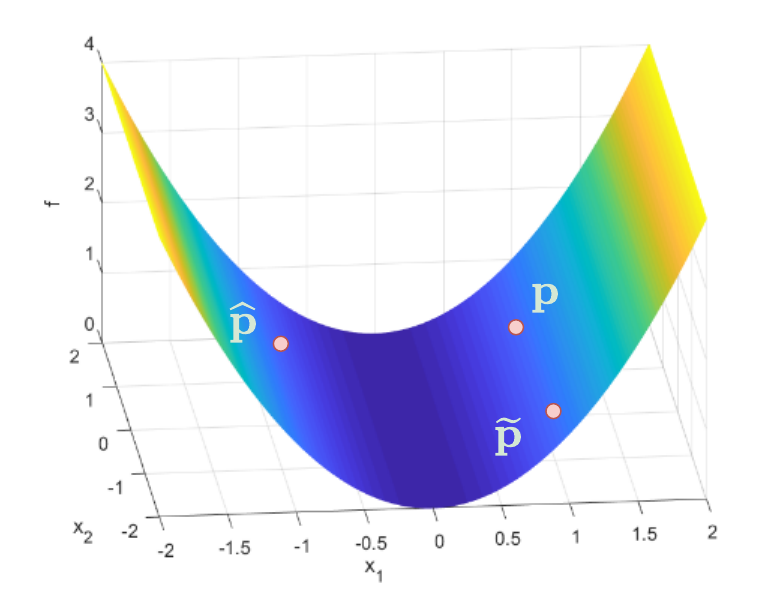}}
  \subfloat[SCR]{\includegraphics[width=0.33\textwidth]{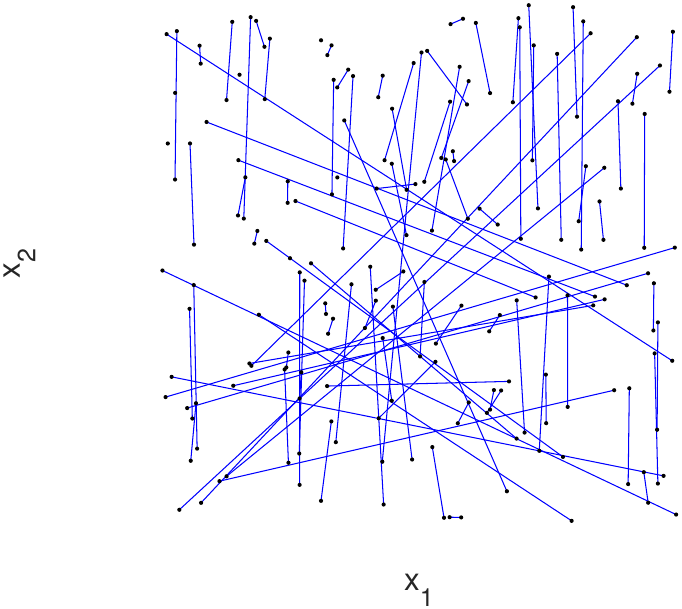}}
  \subfloat[GCR]{\includegraphics[width=0.33\textwidth]{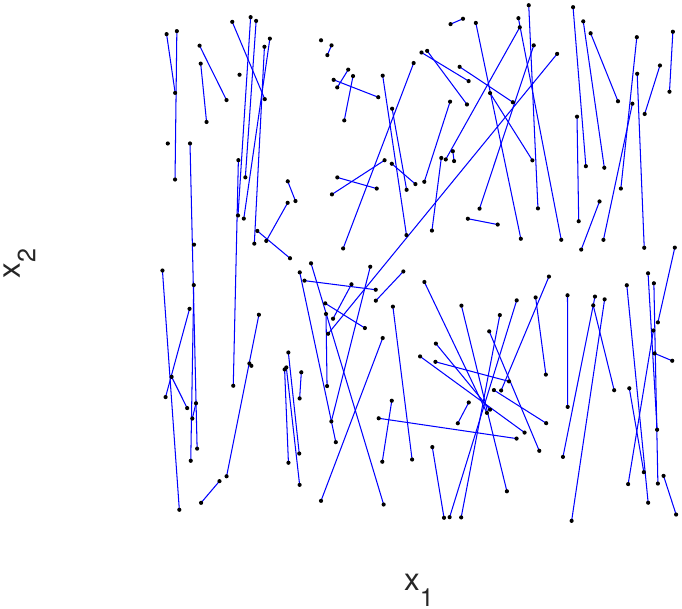}}
  \caption{(a) Function $f(x_1,x_2)=x_1^2$. When 200 samples are uniformly drawn in $[-0.5,0.5]\times[-0.5,0.5]$, (b) and (c) display the connected $(\bx,\tbx)$ pairs by SCR (b) and GCR (c), respectively. An $(\bx,\tbx)$ pair is connected by SCR in (b) if $|f(\bx)-f(\tbx)|\leq0.01$ and by GCR in (c) if $\hV_y(\bx,\tbx;r)\leq0.001$ with $r=0.01$.
  }
  \label{fig.demo.quad}
\end{figure}

%

\subsection{Generalized contour regression}

Condition \ref{assum-scr} does not necessarily hold when $d=1$ and the function $g$ is not monotonic, as well as when $d>1$. For example in Figure \ref{fig.demo.quad} (a), $f(\bx)=x_1^2$ with $\bx=(x_1,x_2)\in[-2,2]\times[-2,2]$. This function can be expressed in Model \eqref{eqyi} with $\Phi=[1,0]^\top$, ${\bf w}=[0,1]^\top$ and $g(z) = z^2$. Let $\mathbf{p}=(p_1,p_2)$ and $\widehat{\mathbf{p}}=(\widehat{p}_1,\widehat{p}_2)$ be two inputs satisfying $p_1>0,\widehat{p}_1<0$, and $|f(\mathbf{p})-f(\widehat{\mathbf{p}})| \le \alpha$ for some small $\alpha$. Due to   symmetry, $|\Phi^\top(\widehat{\mathbf{p}}-\mathbf{p})|=|\widehat{p}_1|+|p_1|$ can be very large, which violates Condition \ref{assum-scr}.
%
When 200 samples are uniformly drawn in $[-0.5,0.5]\times[-0.5,0.5]$,
an $(\bx,\tilde \bx)$ pair is connected by SCR in Figure \ref{fig.demo.quad} (b) if $|f(\bx)-f(\tilde \bx)| \le 0.01$. The ideal connections are along the $x_2$ direction, but SCR gives many misleading connections since $g$ is not monotonic.
When $d>1$, Condition \ref{assum-scr} can be easily violated as well. For any fixed $\bx\in\RR^D$, the contour $\{\tbx| f(\tbx)=f(\bx)\}$ is a curve or a  surface, so $\|\Phi^\top(\tbx-\bx)\|$ is not necessarily small.

Fortunately, most misleading connections can be identified from a large variance of $f$ along the segment between the two inputs. In Figure \ref{fig.demo.quad}, $|f(\mathbf{p})-f(\widehat{\mathbf{p}})|$ is small but the variance of $f$ along the segment between $\mathbf{p}$ and $\widehat{\mathbf{p}}$ is large. This criterion helps to rule out the misleading connection between $\mathbf{p}$ and $\widehat{\mathbf{p}}$.

Replacing the condition of $|y-\tilde y|\le \alpha$ in SCR by a variance condition gives rise to the Generalized Contour Regression (GCR) \cite{li2005contour}.
Let $\ell(\bx_i,\bx_j)=\{\bx = (1-t)\bx_i+t\bx_j, t\in [0,1]\}$ be the segment between $\bx_i$ and $\bx_j$. The variance of $f$ along this line is:
$$
V_f(\bx_i,\bx_j)=\var\left(f(\bx)|\bx \in \ell(\bx_i,\bx_j)\right).
$$
In GCR, the covariance matrix is taken to be
\begin{align}
  G(\alpha)=\EE\left[ (\tilde{\bx}-\bx)(\tilde{\bx}-\bx)^\top| V_f(\tbx,\bx)\leq \alpha\right]
  \label{eq-gcrcov}
\end{align}
where $\alpha>0$ is a parameter. In \cite{li2005contour}, the authors assume that for any unit vectors $\bv \in \calS_\Phi$, ${\bf w}\in  \calS_\Phi^\perp$, there is some constant $\alpha$ such that
\begin{align} 
\var\left[ {\bf v}^\top(\tbx-\bx)| V_f(\tilde{\bx},\bx)\leq \alpha \right]<\var\left[ {\bf w}^\top(\tbx-\bx)| V_f(\tilde{\bx},\bx)\leq \alpha\right].
\label{eq.gcr.li}
\end{align}
It was proved that, when $\rho$ has an elliptical distribution, (\ref{eq.gcr.li}) always holds for sufficiently small $\alpha$. In this work, we refine the assumption \eqref{eq.gcr.li} to the following one which requires a gap between the two sides of (\ref{eq.gcr.li}):

\begin{assum}
\label{assum-gcr}
There exist $\alpha_{\rm thresh}>0, c_0>0$ such that
for any $0<\alpha<\alpha_{\rm thresh}$ and unit vectors $\bv \in \calS_\Phi$, ${\bf w}\in  \calS_\Phi^\perp$, the following hold
 \begin{align*}
   \var\left[ {\bf w}^\top(\tbx-\bx)| V_f(\tilde{\bx},\bx)\leq \alpha\right]- \var\left[ {\bf v}^\top(\tbx-\bx)| V_f(\tilde{\bx},\bx)\leq \alpha \right] &\ge c_0.
    \end{align*}

\end{assum}

This assumption with a gap is used to prove an eigengap of the matrix $G(\alpha)$, which is necessary to control the central subspace estimation error.


An $(\bx,\tilde \bx)$ pair is said to be connected if $V_f(\bx, \tilde\bx) \le \alpha$. In Figure \ref{fig.demo.quad}, even though $|f(\bp)-f(\widehat{\bp})|$ is small, the variance $V_f(\bp,\widehat{\bp})$ is large. When $\alpha$ is small, the condition of $V_f(\bp,\widehat{\bp})\leq \alpha$ is violated so the $(\bp,\widehat{\bp})$ pair is not connected. We expect all connected pairs by GCR to be aligned with the $x_2$ direction, such as $(\bp, \widetilde\bp)$. For such connected pairs, it is very likely that $p_1\widetilde p_1 >0$, which implies $|\Phi^\top(\widetilde{\bp}-\bp)| = |\widetilde p_1 -p_1| \le \alpha/\min(p_1,\widetilde p_1)$  while $|{\bf w}^\top(\widetilde{\bp}-\bp)| = |\widetilde p_2-p_2| \in[0,4]$ is independent of $\alpha$.
When 200 samples are uniformly drawn in $[-0.5,0.5]\times[-0.5,0.5]$, an $(\bx,\tilde \bx)$ pair is connected by GCR in Figure \ref{fig.demo.quad}(c) if an empirical estimator of $V_f(\bx,\tilde \bx)$ (see \eqref{eqhatvy}) is no more than $0.001$. We observe that, most connections by GCR are along the $x_2$ direction, and many misleading connections by SCR are ruled out.


Assumption \ref{assum-gcr} together with a spherical distribution of $\bx$ imply that, when $\alpha<\alpha_{\rm thresh}$ the eigenspace of $G(\alpha)$ associated with the smallest $d$ eigenvalues is $\calS_\Phi$ and the rest of the eigenvectors span $\calS_\Phi^\perp$.

\begin{prop}
\label{propphi}
Assume $\rho$ has a spherical distribution. Let $\lambda_1 \ge \lambda_2 \ge \ldots \ge \lambda_D$ be the eigenvalues of $G(\alpha)$. Under Assumption \ref{assum-gcr},  the followings hold for any $\alpha \in (0,\alpha_{\rm thresh})$:
\begin{itemize}
\item For any integer $1\leq j\leq D-d$ and $D-d+1\leq k \leq D$, we have $\lambda_j-\lambda_k\geq c_0$.
\item The eigenvectors of $G(\alpha)$ associated with $\{\lambda_j\}_{j=1}^{D-d}$ span $\calS_\Phi^\perp$ and the eigenvectors associated with $\{\lambda_k\}_{k=D-d+1}^D$ span $\calS_\Phi$.
 \end{itemize}
\end{prop}

Proposition \ref{propphi} is proved in Supplementary materials \ref{sec.proof.gcr}.
The second part of Proposition \ref{propphi} was first proved in \cite{li2005contour} based on the assumption in \eqref{eq.gcr.li}. We further show that, under Assumption \ref{assum-gcr}, there is a eigengap between the first $d$ eigenvalues and the rest of the eigenvalues for $G(\alpha)$. 

\section{Main Results}
\label{sec.MainResults}


\subsection{Assumptions}


In order to guarantee the success of GCR, we introduce the following assumptions on $\rho$, requiring that i) $\rho$ is supported on a bounded domain; ii) $\rho$ has a spherical distribution.

\begin{assum}
\label{assumrho}
Suppose the probability measure $\rho$ satisfies the following assumptions:
\begin{enumerate}[i)]

\item {\emph{\textbf{[Boundedness]}}} $\supp(\rho) $ is bounded by $B >0$: for every $ \bx$ sampled from $\rho$, $\|\bx\| \leq B$ almost surely.

\item {\emph{\textbf{[Spherical distribution]}}} $\rho$ has a spherical distribution.
\end{enumerate}
\end{assum}

%
Assumption \ref{assumrho}(ii) is common in inverse regression, but GCR is still more robust against nonsphericity (or nonellipticity) than SIR and SCR (see Experiment 1 in Section \ref{sec.numerical}).

We next define Lipschitz and H\"older functions and make some regularity assumption on $g$.

\begin{defi}[Lipschitz functions]
A function $g$ is Lipschitz with Lipschitz constant $L_g$ if
$$
|g(\bz)-g(\widetilde{\bz})|\leq L_g \|\bz-\widetilde{\bz}\|,\ \forall \bz,\widetilde{\bz}\in \supp(g).
$$

\end{defi}

\begin{defi}[H\"older functions]
Let $s=k+\beta$ for some $k \in  \NN_0$ and $0<\beta \le1$, and $C_g >0$. A function $g:\RR^d\rightarrow\RR$ is called $(s,C_g)$-smooth if for every $\boldsymbol\alpha=(\alpha_1,...,\alpha_d), \alpha_i\in\NN_0,|\boldsymbol\alpha|\leq k$, the partial derivative $D^{\boldsymbol\alpha }g:=\frac{\partial^k g}{\partial x_1^{\alpha_1}\cdots \partial x_d^{\alpha_d}}$ exists and satisfies
\begin{align*}
&\sup_{|\boldsymbol\alpha|<k}\sup_{\bz} |D^{\boldsymbol\alpha} g(\bz)|\le C_g,\\
&\left| D^{\boldsymbol\alpha} g(\bz)- D^{\boldsymbol\alpha} g(\widetilde{\bz})\right|\leq C_g \|\bz-\widetilde{\bz}\|^{\beta},\ \forall \bz,\widetilde{\bz}\in \supp(g), |\boldsymbol\alpha|=k,
\end{align*}
where $|\boldsymbol\alpha|=\sum_{j=1}^d \alpha_j$.
\end{defi}

In this paper, the function $g$ is assumed to be $(s,C_g)$ smooth.

\begin{assum}
\label{assumg}
The function $g : \RR^d \rightarrow \RR$ is $(s,C_g)$-smooth with $s\geq1$.

\end{assum}

Assumption \ref{assumrho} and \ref{assumg} imply the followings:
\begin{enumerate}[i)]
\item $g$ is Lipschitz as $s\geq 1$, and $C_g$ is an upper bound of the Lipschitz constant of $g$.
\item $g$ is bounded by Assumption \ref{assumrho} (i) and the Lipschitz property above: $|g|\leq M$ where $M$ denotes an upper bound of $g$ which satisfies $M\leq |g(\mathbf{0})|+C_gB$.
\end{enumerate}


\begin{assum}
\label{assumxi}
The noise $\xi$ has zero mean and is bounded: $\EE \xi = 0$ and $\xi \in [-\sigma,\sigma]$ almost surely.

\end{assum}

\subsection{Central subspace error}
\label{sec.activeSubspaceError}
Our central interest is to estimate the central subspace. In this paper, we prove the central subspace estimation error by GCR based on the exact variance quantity $V_f(\tbx,\bx)$. The empirical estimation of $V_f(\tbx,\bx)$ is discussed in Section \ref{subsecalgorithm}.
If $V_f(\tbx,\bx)$ is exactly known, we define
\begin{equation}
	H(\alpha):=\EE\left[ (\tilde{\bx}-\bx)(\tilde{\bx}-\bx)^\top\mone{\left\{ V_f(\tilde{\bx},\bx)\leq \alpha\right\}}\right].
\end{equation}
such that
\begin{align}
	H(\alpha)=G(\alpha)\PP(V_f(\tbx,\bx)\leq \alpha),
	\label{eq.HGP}
\end{align}
where $\mone{\left\{ V_f(\tilde{\bx},\bx)\leq \alpha\right\}}$ is the indicator function which is equal to 1 if $V_f(\tilde{\bx},\bx)\leq \alpha$ and 0 otherwise. The matrices $H(\alpha)$ and $G(\alpha)$ have the same set of eigenvectors. According to Proposition \ref{propphi}, for any $0<\alpha<\alpha_{\rm thresh}$, the eigenvectors associated with the smallest $d$ eigenvalues of $H(\alpha)$ form an orthonormal basis of the central subspace.

The empirical counterpart of $H(\alpha)$ based on the i.i.d. samples $\{(\bx_i,y_i)\}_{i=1}^n$ is
\begin{equation}
\tH(\alpha) :=\frac{1}{\binom{n}{2}} \sum_{1\leq i< j\leq n} (\bx_i -\bx_j)(\bx_i -\bx_j)^\top\mone{\{V_f(\bx_i,\bx_j)\leq \alpha\}},
\label{eq-hK-U}
\end{equation}
and
$	\EE \tH(\alpha)=H(\alpha).
$
The central subspace estimator based on $\tH$ is denoted by $\tPhi \in \RR^{D \times d}$ whose columns are the eigenvectors associated with the smallest $d$  eigenvalues of $\tH$.
This estimation error is quantified by $\|\proj_{\tPhi}-\proj_{\Phi}\|$.
 To show that $\proj_{\tPhi}$ is close to $\proj_{\Phi}$, we first derive a concentration inequality for 
 $\tH(\alpha)$.
\begin{thm}
\label{thm.HDisVf}
  Let $\{\bx_i\}_{i=1}^n$ be i.i.d. samples from a probability measure $\rho$ satisfying Assumption \ref{assumrho}. Assume $V_f(\tbx,\bx)$ is known for any $(\tbx,\bx)$ pair. For any $\alpha\in (0,\alpha_{\rm thresh})$ and any $t >0$,
\begin{align}
\PP(\|\tH(\alpha)-H(\alpha)\|\geq t)\leq D\exp\left(-\frac{3nt^2}{32B^2(12B^2+t)}\right),
\label{eq.HDisVf}
\end{align}
where $B$ are given in Assumption \ref{assumrho}.
 \end{thm}
Theorem \ref{thm.HDisVf} is proved in Section \ref{sec.mainproof.matrixU} with a new concentration inequality for matrix-valued U-statistics. To further bound the distance between $\proj_{\tPhi}$ and $\proj_{\Phi}$, we will use the relation between $H(\alpha)$ and $G(\alpha)$ in (\ref{eq.HGP}), and the eigengap implied from Assumption \ref{assum-gcr}. 
Denote $$p_{\alpha}=\PP(V_f(\tbx,\bx)\leq \alpha).$$
Since $f$ is Lipschitz and $\supp(\rho)$ is compact,   we expect that, for a fixed $\alpha>0$, $p_{\alpha}$ is bounded away from 0. The following example shows that when $\rho$ has a spherical distribution and under mild conditions, $p_{\alpha}$ is bounded away from 0.
\begin{example}\label{lem.sphere}
	Suppose Assumption \ref{assumrho} and \ref{assumg} hold, and the density function of $\rho$ on $\supp(\rho)$ is bounded below by $h_{\min}>0$. If 
	$\alpha<B^2C_g^2/D$, then
	\begin{align}
	p_{\alpha}\geq C\alpha^{d/2}h_{\min}^2
	\label{eq.PSP}
	\end{align}
	where $C=\frac{1}{C_g^d}\left(\frac{2B}{\sqrt{D}}\right)^{2(D-d)}\left(\frac{2B}{\sqrt{D}}- \frac{2\sqrt{\alpha}}{C_g}\right)^d\frac{\pi^{d/2}}{\Gamma(d/2+1)}$.
\end{example}
Example \ref{lem.sphere} is proved in Supplementary materials \ref{appen.lem.sphere}. 
The following theorem shows that, under the same conditions as in Theorem \ref{thm.HDisVf}, $\proj_{\tPhi}$ is close to $\proj_{\Phi}$ with high probability (see a proof in Section \ref{sec.mainproof.HDisVf}).
\begin{thm}
	\label{thm.projDisVf}
	Let $\{\bx_i\}_{i=1}^{n}$ be i.i.d. samples from a probability measure $\rho$.
	Suppose Assumption \ref{assum-gcr} and \ref{assumrho} hold and $V_f(\tbx,\bx)$ is exactly known for any $(\tbx,\bx)$ pair. For any $0<\alpha<\alpha_{\rm thresh}$ and $t\in (0,2)$, we have
	\begin{align}
		\PP\left\{
		\|\proj_{\tPhi} -\proj_{\Phi}\| \ge t
		\right\}
		\le
		D\exp\left(-\frac{p_{\alpha}^2c_0^2nt^2}{32B^2(36B^2+p_{\alpha}c_0t)}\right)
		\label{eq.projDisVf}
	\end{align}
	where $c_0,B$ are given in Assumption \ref{assum-gcr} and \ref{assumrho}.
\end{thm}

Integrating the probability in (\ref{eq.projDisVf}) gives rise to the following mean squared estimation error for the central subspace:
\begin{coro}
\label{coroproj}
  Under the conditions in Theorem \ref{thm.projDisVf}, we have
  \begin{align}
    \EE\|\proj_{\tPhi} -\proj_{\Phi}\|^2\leq \frac{C_4}{n},
  \end{align}
  where\begin{equation}
  	\begin{aligned}
  		&C_1=p_{\alpha}^2c_0^2,\  C_2=1152B^4,\  C_3=32B^2p_{\alpha}c_0,\ C_4=\frac{C_2}{C_1}(\log D+D+1)+\frac{8DC_3^2}{C_1^2}.
  	\end{aligned}
  	\label{eq.C1-4}
  \end{equation}
\end{coro}
Corollary \ref{coroproj} is proved in Section \ref{sec.mainproof.G3}. It implies that, if $V_f(\tbx,\bx)$ is exactly known, the mean squared estimation error of GCR for the central subspace converges in the rate of $O(n^{-1})$.

\subsection{Regression error}
\label{sec.hf}
\label{sec.regressionscheme}

\begin{algorithm}[t]
	\SetKwInOut{KwIni}{Initialization}
    \KwIn{$\{(\bx_i,y_i)\}_{i=1}^{2n}$, $V_f$, parameters $\alpha,s$ and dimension $d$.}
	\textbf{Step 1}: Split data to two subsets $\mathcal{S}_1=\{(\bx_i,y_i)\}_{i=1}^{n}$ and $\mathcal{S}_2=\{(\bx_i,y_i)\}_{i=n+1}^{2n}$.\\
    \textbf{Step 2 (GCR)}: Compute the matrix $\tH(\alpha)$ using $\mathcal{S}_1$.\\
    \hspace{1cm} $\tPhi \leftarrow$ eigenvectors associated with the smallest $d$ eigenvalues of $\tH(\alpha)$.\\
    \textbf{Step 3 (Regression)}: Compute $\tg$ as the piecewise polynomial of degree $k$ to fit the data $\{(\tPhi^\top\bx_i,y_i)\}_{i=n+1}^{2n}$, according to (\ref{eq.gh0}), where $k=\lceil s \rceil-1$. \\
	\KwOut{Recovered subspace $\tPhi$ and function $\tg$ such that $\tf(\bx)=\tg(\tPhi^\top\bx)$.}
	\caption{Regression of $f$ based on exact $V_f$}
	\label{alg2}
\end{algorithm}

Given $2n$ samples denoted by $\calS = \{(\bx_i,y_i)\}_{i=1}^{2n}$, we evenly split the data into two subsets $\calS_1 = \{(\bx_i,y_i)\}_{i=1}^{n}$ and $\calS_2 = \{(\bx_i,y_i)\}_{i=n+1}^{2n}$ while $\calS_1$ is used to compute $\tH(\alpha)$ as described in Section \ref{sec.activeSubspaceError}, and $\calS_2$ is used to estimate the function $g$.
After the central subspace is estimated as $\tPhi \in \RR^{D \times d}$, the next step is to estimate the function $g$ from the data $\{(\bz_i,y_i)\}_{i=n+1}^{2n}$, where $\bz_i= \tPhi^\top \bx_i \in \RR^d$. Our regression scheme is summarized in Algorithm \ref{alg2}. A rich class of nonparametric regression techniques \cite{gyorfi2006distribution,tsybakov2009nonparametric,wasserman2006all} can be used to estimate $g$, such as kNN, kernel regression, polynomial partitioning estimates, etc.
In this paper, we will present the results of polynomial partitioning estimates.

The data $\{(\bz_i,y_i)\}_{i=n+1}^{2n}$ satisfies the following model
\begin{equation}
\label{eqmodelg}
y_i = g(\bz_i ) + \widetilde{\xi}_i, \ i =n+1,\ldots,2n
\end{equation}
with noise $ \widetilde{\xi}_i$ given by
\begin{align}
  \widetilde{\xi}_i=y_i-g(\tPhi^\top\bx_{i} )=g(\Phi^\top\bx_{i})-g(\tPhi^\top\bx_{i})+\xi_{i}.
  \label{eq.txi}
\end{align}
Due to the mismatch between $\tPhi$ and $\Phi$, the noise $\widetilde\xi_i$ has a bias and its conditional mean at $\bz = \tPhi^\top \bx$ is
\begin{equation}
\label{eqeta}
\eta(\bz) := \EE[\widetilde\xi | \tPhi^\top \bx =\bz] = \EE[g(\Phi^\top\bx)-g(\tPhi^\top\bx) | \tPhi^\top \bx =\bz],
\end{equation}
where the expectation {is taken over $\bx \sim \rho$} and $\xi$, conditioning on $ \tPhi^\top \bx=\bz$. This function is bounded such that $\|\eta\|_\infty \le C_g B \|\tPhi-\Phi\|$. We use
\begin{align}
  b^2_{\mathrm{bias}}:=C_g^2 B^2 \|\tPhi-\Phi\|^2
  \label{eq.bnoise}
\end{align}
to denote an upper bound of the squared bias in noise.

The support of the measure $\rho$ is bounded by Assumption \ref{assumrho}(i), which implies $\bz_i \in [-B,B]^d $. For a fixed positive integer $K$, let $\mathcal{F}_k$ be the space of piecewise polynomials of degree no more than $k$ on the partition of $[-B,B]^d$ into $K^d$ cubes with side length $2B/K$. If $g$ is $(s,C_g)$ smooth, the polynomial degree $k$ should be chosen as $k=\lceil s \rceil-1$. Consider the piecewise polynomial estimator of order $k$:
\begin{align}
\label{eqtildeg}
\bar{g}(\bz)=\argmin_{h \in \mathcal{F}_k} \frac{1}{n}\sum_{i=n+1}^{2n} |h(\bz_i)-y_i|^2.
\end{align}
Since $g$ is bounded by $M$, we truncate the final estimator to $\widetilde{g}$ such that
\begin{equation}
\widetilde{g}(\bz)=T_M \bar g(\bz)=\begin{cases}
  \bar g(\bz), & \mbox{ if } |\bar g(\bz)|\le M,\\
  M\cdot \mbox{sign}(\bar g(\bz)), &\mbox{ otherwise.}
\end{cases}
\label{eq.gh0}
\end{equation}
The parameter $K$ determines the size of the partition, which we set as
\begin{equation}
  K=\Big\lceil \frac{n}{\max(\sigma^2  +2C_4n^{-1},2M^2+ 4C_4n^{-1})\log n}\Big\rceil^{\frac{1}{2s+d}}
  \label{eq.K}
\end{equation}
with $C_4$ defined in (\ref{eq.C1-4}).

The goal of this paper is to give an error analysis of $\widetilde{f}$, which has the Mean Squared Error (MSE)
$$
\EE\|\tf(\bx)-f(\bx) \|_{L^2(\rho)}^2=\underset{\{(\bx_i,y_i)\}_{i=1}^{2n}}{\EE} \int |\tf(\bx)-f(\bx) |^2 \rho(d\bx),
$$
where the expectation is taken over the joint distribution of $\{(\bx_i,y_i)\}_{i=1}^{2n}$.
For any $\bx \in \RR^D$,
\begin{eqnarray}
  |\tf(\bx)-f(\bx) |^2 &=& | \tg(\tPhi^\top\bx)- g(\Phi^\top\bx)|^2\nonumber\\
  &\leq& 2|g(\tPhi^\top \bx)- g(\Phi^\top\bx)|^2\nonumber+  2|\tg(\tPhi^\top\bx)-g(\tPhi^\top \bx)|^2\\
  &\leq& 2C_g^2 \|{\tPhi} -{\Phi} \|^2\|\bx\|^2 + 2|\tg(\tPhi^\top\bx)-g(\tPhi^\top \bx)|^2.
  \label{eq.squarf}
\end{eqnarray}
The first term captures the estimation error of the central subspace by GCR, and the second terms captures the regression error of $g$.
The following theorem provides an upper bound for the MSE of $\tf$ (see its proof in Section \ref{sec.mainproof.f}).

\begin{thm}
\label{thm.fError}
Let $\{\bx_i\}_{i=1}^{2n}$ be i.i.d. samples of a probability measure $\rho$ and $\{y_i\}_{i=1}^{2n}$ be sampled according to the model in \eqref{eqyi}. Suppose Assumption \ref{assum-gcr}-\ref{assumxi} hold and $V_f(\tbx,\bx)$ is exactly known for any $(\tbx,\bx)$ pair. Set $\alpha\in (0,\alpha_{\rm thresh})$ and $K$ according to (\ref{eq.K}). The estimator $\tf$ in (\ref{eq.gh0}) satisfies 
  \begin{align}
    &\EE\|\tf(\bx)-f(\bx)\|_{L^2(\rho)}^2 \leq \frac{28C_5}{n} + 2C\left(\frac{\max(\sigma^2  + 2M^2+ 6C_5n^{-1})\log n}{n}\right)^{\frac{2s}{2s+d}}.
    \label{thmfeq}
  \end{align}
  where $C_5=C_4C_g^2 B^2$, $C_4$ is defined in (\ref{eq.C1-4}) and $C>0$ is a constant depending on $d,k,s,C_g,B,M$.
\end{thm}

Theorem \ref{thm.fError} demonstrates that if GCR is used to estimate the central subspace, the mean squared regression error decays exponentially in $n$ with an exponent depending on $d$, instead of $D$.
GCR effectively exploits the low-dimensional structure of the function and gives a fast rate of convergence in comparison with a direct regression in $\RR^D$.

\subsection{A practical algorithm to estimate the central subspace}
\label{subsecalgorithm}

Our estimation theory in Theorem \ref{thm.projDisVf} and \ref{thm.fError}   utilizes  U-statistics with the exact knowledge of $V_f(\tbx,\bx)$ for any $(\tbx,\bx)$ pair. In practice, $V_f(\tbx,\bx)$ is not given and we need to estimate it from the samples. In this paper, we use the empirical estimation of $V_f$ proposed in \cite{li2005contour} and prove an estimation error for $V_f$.
We also propose an efficient algorithm to compute the empirical covariance matrix for the estimation of the central subspace.


\subsubsection{Empirical estimation of $V_f$}

The variance quantity $V_f(\bx_i,\bx_j)$ is the variance of $f$ along the segment $\ell(\bx_i,\bx_j)$. Since $\ell(\bx_i,\bx_j)$ has $0$ measure, it is unlikely to have data exactly lying on $\ell(\bx_i,\bx_j)$. Following the idea in \cite{li2005contour}, we approximate $V_f(\bx_i,\bx_j)$ by the variance of $y$ in a narrow tube enclosing $\ell(\bx_i,\bx_j)$ with radius $r$ (see Figure \ref{fig.tube}). 
Let $d(\bx,\ell(\bx_i,\bx_j))$ be the distance from $\bx$ to $\ell(\bx_i,\bx_j)$:
$d(\bx,\ell(\bx_i,\bx_j)) = \inf_{\bz \in \ell(\bx_i,\bx_j)} \|\bx-\bz\|. $
We define $T_{ij}(r)$ as the tube enclosing $\ell(\bx_i,\bx_j)$ with radius $r$ as follows
$$
T_{ij}(r):=\left\{\bx|d(\bx,\ell(\bx_i,\bx_j))\leq r, (\bx-\bx_i)\cdot(\bx_j-\bx_i)\geq0, (\bx-\bx_j)\cdot(\bx_i-\bx_j)\geq0\right\}.
$$
The variance of $y$ in $T_{ij}(r)$ is denoted as
\begin{equation}
	\label{eqvy}
	V_y(\bx_i,\bx_j,r) =\var\left(y|\bx \in T_{ij}(r)\right).\end{equation}
The empirical counterpart of $V_y(\bx_i,\bx_j,r)$ based on the data $\{(\bx_i,y_i)\}_{i=1}^n$ is
\begin{equation}
	\label{eqhatvy}
	\hV_y(\bx_i,\bx_j,r) =\frac{1}{\hn_{ij}(r)-1}\sum_{\bx_k \in T_{ij}(r)} (y_k-\bar{y}_{ij})^2, \text{ where } \bar{y}_{ij} = \frac{1}{\hn_{ij}(r)}\sum_{\bx_k \in T_{ij}(r)} y_k,
\end{equation}
\begin{wrapfigure}{r}{0.4\textwidth}
	\centering
	\includegraphics[width=\textwidth]{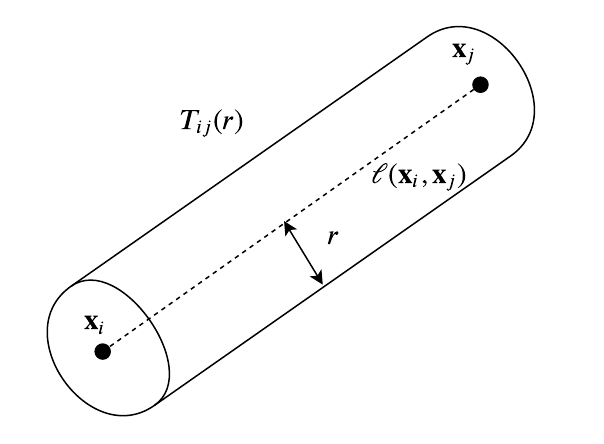}
	\caption{\label{fig.tube} $T_{ij}(r):$ the tube enclosing $\ell(\bx_i,\bx_j)$ with radius $r$.}
\end{wrapfigure}
and $\hn_{ij}(r)  = \#\{\bx_k \in T_{ij}(r)\}$. Here $\#S$ denotes the cardinality of the set $S$.

We make the following assumption on the measure of tube:
\begin{assum}
	\label{assumtube}
	Assume there exists $c_1  >0$ such that, for any $\bx_i , \bx_j \in \supp(\rho)$ and any $r \in (0,\|\bx_i-\bx_j\|/4)$,
		\begin{equation}\rho(T_{ij}(r)) \ge c_1 r^{D-1} \|\bx_i-\bx_j\|.
			\label{eqmeasuretube}
		\end{equation}
\end{assum}
Assumption \ref{assumtube} guarantees a sufficient amount of points in every tube so that $V_y(\bx_i,\bx_j,r)$ in \eqref{eqvy} can be well approximated by its empirical quantity $\hV_y(\bx_i,\bx_j,r)$ in \eqref{eqhatvy}. Assumption \ref{assumtube} is satisfied if $\rho$ has a density bounded below on $\supp(\rho)$.

Let $V_f(\bx_i,\bx_j,r)$ be the variance of $f$ in the tube $T_{ij}(r)$:
\begin{align}
	V_f(\bx_i,\bx_j,r)=\var\left(f(\bx)|\bx \in T_{ij}(r)\right).
\end{align}
We further assume that the difference between $V_f(\bx_i,\bx_j,r)$ and $V_f(\bx_i,\bx_j)$ is linear in $r$:
\begin{assum}
	\label{assumVr}
	There exists $c_2  >0$ such that, for any $\bx_i , \bx_j \in \supp(\rho)$,
	\begin{equation}
		|V_f(\bx_i ,\bx_j)-V_f(\bx_i ,\bx_j,r)|\le c_2 r.
		\label{eq.Vr}
	\end{equation}
\end{assum}
Assumption \ref{assumVr} implies that when $r$ is small, $V_f(\bx_i,\bx_j,r)$ is close to $V_f(\bx_i,\bx_j)$ for any $\bx_i , \bx_j \in \supp(\rho)$. The following example shows that when $\rho$ is the volume measure on its support and $T_{ij}(r)\subset \supp(\rho)$, then Assumption \ref{assumVr} holds.
\begin{example}
	\label{lemmavfgamma}
	Let $f$ be defined as \eqref{eqyi}.
	Suppose Assumption \ref{assumrho}(i) and \ref{assumg} hold, and $\rho$ is the volume measure on a compact set in $\RR^D$. For every $\bx_i,\bx_j \in \supp(\rho)$ such that  $T_{ij}(r)\subset \supp(\rho)$, and any $r\leq \sqrt{5}B/2$, we have
	\begin{equation}
		\label{eqvfgamma}
		|V_f(\bx_i ,\bx_j)-V_f(\bx_i ,\bx_j,r)|\le 10  C_g^2 B r.
	\end{equation}
\end{example}
Example \ref{lemmavfgamma} is proved in Supplementary materials \ref{appen.lem.lemmavfgamma}.

\subsection{An efficient algorithm to form the empirical covariance matrix}

Forming the empirical covariance matrix as in \eqref{eq-hK-U} requires $\hV_y(\bx_i,\bx_j,r)$ for any $(\bx_i,\bx_j)$ pair. With $n$ samples, the computational complexity is $n \choose 2$, which is expensive. When the parameter $\alpha$ is small, we expect to have a  small number of $(\bx_i,\bx_j)$ pairs  connected such that  $\hV_y(\bx_i,\bx_j,r) \le \alpha$.
We next propose a new algorithm to efficiently identify the connected pairs.

In our new algorithm, we form the empirical version of $G(\alpha)$ based on $n$ i.i.d. samples $\{(\bx_i,y_i)\}_{i=1}^n$. Let $\calA_{\alpha}\in\{0,1\}^{n\times n}$ be a matrix with entries in $\{0,1\}$ such that all $(i,j)$ indices with $\calA_{\alpha}(i,j)=1$ satisfy $\hV_y(\bx_i,\bx_j,r)\leq \alpha$. Our estimation of $G(\alpha)$ is
\begin{equation}
	\hG(\alpha,r) =\frac{1}{\hn_{\alpha}} \sum_{(i,j):\ \calA_{\alpha}(\bx_i,\bx_j) =1} (\bx_i -\bx_j)(\bx_i -\bx_j)^\top,
	\label{eq-hG}
\end{equation}
where $\hn_{\alpha} =\# \{ (i,j): \calA_{\alpha}(\bx_i,\bx_j) = 1\}$. In other words, the connected pairs are indicated by the nonzero entries in the matrix $\calA_{\alpha}$. In this paper, we form the matrix $\calA_{\alpha}$ via Algorithm \ref{alg1}. In Algorithm \ref{alg1}, each data point is used at most once: when a point is connected in one pair, we remove it from the rest of the computation. Algorithm \ref{alg1} is more efficient than the original GCR in the sense that it outputs at most $n/2$ connected pairs while the original GCR uses $n \choose 2$ pairs (most of which are not connected) to estimate $G(\alpha)$.  Therefore the computational cost is greatly reduced. Moreover, our numerical experiments show that the mean squared error of the central subspace estimation by Algorithm \ref{alg1} converges in the order of $n^{-1}$, which is the same as that in Corollary \ref{coroproj}.

In $\hG(\alpha,r)$, $\alpha$ and $r$ are two important parameters, which need to be properly chosen. In this paper, we choose these parameters as follows:
Let $\nu >0$ be a fixed constant. We set
\begin{align}
	&\alpha_0 = \max\left \{ C_7 \left( \frac{\log n}{n}\right)^{\frac 1 D}, C_8   \left( \frac{\log n}{n}\right)^{\frac 1 {D+2}} \right\}, \quad r=C_6\alpha_0 \label{eqalpha00}\\
	&\alpha =  4d C_g^2 B^2 \left(\frac{\log n}{n}\right)^{\frac 1 D}+\alpha_0+3\sigma^2
	\label{eqalpha}
\end{align}
with $C_6=\left(\max(2c_2,8C_g^2)\right)^{-1}$, $C_7 =  \left(\frac{56\nu C_g^2}{3c_1C_6^{D-1}}\right)^{\frac 1 D}  ,C_8 = \left(\frac{256\nu(M+\sigma)^4C_g^2}{c_1C_6^{D-1}} \right)^{\frac{1}{D+2}}$, where the parameters $C_g,B,c_1,c_2,\sigma,M$ are defined in Assumption \ref{assumrho}-\ref{assumVr}.

In the following theorem, we prove that, if the parameters $\alpha,r$ are chosen according to \eqref{eqalpha00} and \eqref{eqalpha}, Algorithm \ref{alg1} guarantees at least $n/4$ connected pairs of $(\bx_{i},\bx_{j})$ such that $\hV_y(\bx_{i},\bx_{j},r)\le \alpha$, and all the connected pairs satisfy $V_f(\bx_{i},\bx_{j})\leq \alpha+ \alpha_0+3\sigma^2$ with high probability.


\begin{thm}
	\label{prop.alg1}
	Let $\{\bx_i\}_{i=1}^n$ be i.i.d. samples from the probability measure $\rho$, and $\{y_i\}_{i=1}^n$ be sampled according to the model in \eqref{eqyi}, under Assumption \ref{assumrho}(i), \ref{assumg}-\ref{assumVr}.
	Let $\nu >2$ and set $\alpha_0,r,\alpha$ according to (\ref{eqalpha00}) and (\ref{eqalpha}). 
	Index the output pairs $\{(\bx_i,\bx_j)| \calA(i,j) = 1\}$ by Algorithm \ref{alg1} as $\{(\bx_{i_k},\bx_{j_k})\}_{k=1}^{\hn_{\alpha}}$. If {$n$ is sufficiently large such that $2 (n/\log n)^{\frac{d}{2D}} \le n$ and $\alpha_0<2C_g^2$}, running Algorithm \ref{alg1} gives rise to
	\begin{align}
		&\PP\left(\bigcap_{k=1}^{\hn_{\alpha}}\left[V_f(\bx_{i_k},\bx_{j_k})\leq \alpha+ \alpha_0+3\sigma^2\right]\right)\geq 1-2n^{-(\nu-1)},  \label{eq.hV} \\
		& \mbox{with }\quad \hn_{\alpha}\leq n/2,\ \mbox{ and }\ \PP(\hn_{\alpha}\geq n/4)\geq 1-2n^{-(\nu-2)} \label{eq.hn}.	
	\end{align}
\end{thm}

Theorem \ref{prop.alg1} is proved in Supplementary materials \ref{appen.prop.alg2}. Theorem \ref{prop.alg1} shows that if $\alpha$ and $r$ are properly chosen, with high probability, Algorithm \ref{alg1} gives rise to $\hG(\alpha,r)$ with at least $n/4$ pairs of $(\bx_i,\bx_j)$. Moreover, all such pairs satisfy
$V_f(\bx_{i},\bx_{j})\leq \alpha+ \alpha_0+3\sigma^2$ with high probability.
If $n$ is large enough and $\sigma$ is small enough such that $\alpha+ \alpha_0+3\sigma^2<\alpha_{\rm thresh}$, we expect the estimated subspace $\hPhi$ by Algorithm \ref{alg1}  to be a good approximation of the central subspace $\Phi$.

Based on the estimated subspace $\hPhi$ by Algorithm \ref{alg1}, we perform regression to obtain $\hg$ as described in Section \ref{sec.hf} using the samples $\{(\hPhi^\top\bx_i,y_i)\}_{i=n+1}^{2n}$. Then $f$ is estimated as
$
	\hf(\bx)=\hg(\hPhi^\top\bx) .
$

\subsection{Data normalization}

Our theoretical analysis assumes that $\EE \bx = \mathbf{0}$, $\EE \bx\bx^\top = I$, and $\bx$ follows a spherical distribution. In practice these conditions may not be satisfied for the given data, and we always preprocess the data by normalization \cite{li2007directional}. Given the data set $\{\bx_i,y_i\}_{i=1}^{n}$, we first compute the empirical mean $\bar \bx =\frac{1}{n} \sum_{i=1}^{n} \bx_i $ and the empirical covariance matrix $\widehat{\Sigma} = \frac{1}{n-1} \sum_{i=1}^n (\bx_i -\bar\bx)(\bx_i -\bar\bx)^\top$, and then we normalize the data as
$$
\widetilde{\bx}_i=\widehat{\Sigma}^{-\frac 1 2}(\bx_i-{\bar{\bx}}).
$$
This normalization does not alter the low-dimensional property of the function since
\begin{align*}
f(\bx)=g(\Phi^{\top}\bx)=g\left(({\widehat\Sigma}^{\frac 1 2}\Phi)^{\top}{\widehat\Sigma}^{-\frac 1 2}(\bx-\bar{\bx})+\Phi^\top\bar{\bx}\right) =\widetilde{g}\left(\widetilde{\Phi}^{\top}{\widetilde\bx}\right)
\end{align*}
where $\widetilde{\Phi}={\widehat\Sigma}^{\frac 1 2}\Phi$ and $\widetilde{g}(\bv)=g(\bv+\Phi^\top\bar{\bx})$ .

\begin{algorithm}[t]
	\KwIn{$\mathcal{S}=\{(\bx_i,y_i)\}_{i=1}^n$ and $\alpha>0$, $r>0$.}
	\textbf{Initialization:} $\hG=0$, $\hn_{\alpha}=0$ and $\calA=\mathbf{0}_{n \times n}.$\\
	\While {$\#\mathcal{S}>1$}{Index data in $\mathcal{S}$ as $\{(\bx_{i_k},y_{i_k})\}, k=1,...,\#\mathcal{S}$.\\
		\For {$k=2$ to $\#\mathcal{S}$} {
			\If{$\hV_y(\bx_{i_1},\bx_{i_k},r)\le \alpha$}{$\hG=\hG+(\bx_{i_1}-\bx_{i_k})(\bx_{i_1}-\bx_{i_k})^\top$,\\
				$\calA(i_1,i_k)=1 $, \\ 
				$\hn_{\alpha}=\hn_{\alpha}+1$,\\
				$\mathcal{S}=\mathcal{S}\backslash\{\bx_{i_1},\bx_{i_k}\}$\\
				\textbf{break}}
		}
		$\mathcal{S}=\mathcal{S}\backslash\{\bx_{i_1}\}$
	}
	$\hG(\alpha,r)=\hG/\hn_{\alpha}$. \\
	\KwOut{$\hPhi$: eigenvectors associated with the smallest $d$ eigenvalues of $\hG(\alpha,r)$.}
	\caption{A practical algorithm for the central subspace estimation.}
	\label{alg1}
\end{algorithm}

\section{Numerical experiments}
\label{sec.numerical}
In this section, we provide numerical experiments to demonstrate the performance of the modified GCR in Algorithm \ref{alg1} and the regression scheme. The data $\{(\bx_i,y_i)\}_{i=1}^n$ are sampled according to the model in (\ref{eqyi}) with $\xi_i \sim N(0,\sigma^2)$. The noise is $p\%$ if $\sigma={p\%}\sqrt{1/ n\sum_{i=1}^n f^2(\bx_i)}$.
In all experiments, $90\%$ of the given data is used for training and $10\%$ is used for testing. Training data are used for central subspace estimation by GCR, SCR or SIR respectively and regression of $g$ is estimated by Gaussian kernel regression through the MATLAB built-in function \emph{fitrkernel}. The test data are used to compute the central subspace estimation error $\|\proj_{\widehat{\Phi}}-\proj_{\Phi}\|$ and the regression error 
$$\text{regression error} = \sqrt{\frac{1}{n_{\rm test}}\sum_{(\bx_i,y_i)\ \in\ {\rm test\  set}} (\hf(\bx_i)-y_i)^2},$$
where $n_{\rm test}$ is the number of the test data. In GCR, the parameter $r$ is chosen in the order of $n^{-1/D}$ according to \eqref{eqalpha00}. We set $r=2n^{-1/D}$ without specification.

We expect $\EE\|\proj_{\widehat{\Phi}}-\proj_{\Phi}\|\sim n^{-1/2}$, so $\log_{10}\|\proj_{\widehat{\Phi}}-\proj_{\Phi}\|$ scales linearly with respect to $\log_{10}n$, with a slope of $-0.5$, independently of $d$ and $D$.

\subsection{Experiment 1: Robustness to non-elliptical distributions}

In the first experiment, we investigate the sensitivity of GCR, SCR and SIR to the condition of elliptical distributions in Assumption \ref{assumrho}(ii),. Let $f(\bx)=x_1^2$ where $\bx=[x_1,x_2]^T\in \RR^2$. This function can be expressed as the model in \eqref{eqyi} with $\Phi=[1,0]^T$ and $g(z)=z^2$ where $z=\Phi^T\bx$. We sample $\bx$ uniformly in the domain $[-0.5,0.5]\times[-0.5,0.5]$ excluding the forth quarter, which violates the condition of elliptical distributions. We set $r=0.01, \alpha=0.001$ in GCR, $\alpha=0.01$ in SCR and $10$ slices in SIR. Figure \ref{fig.ellip_compair} shows 1500 samples of $\bx$ (black dots), the direction of $\Phi$ (red arrow) and the direction of $\widehat{\Phi}$ (blue arrow). We observe that GCR is more robust than SCR and SIR when $\bx$ is not elliptically distributed.

\begin{figure}[t!]
  \centering
  \subfloat[GCR]{\includegraphics[width=0.3\textwidth]{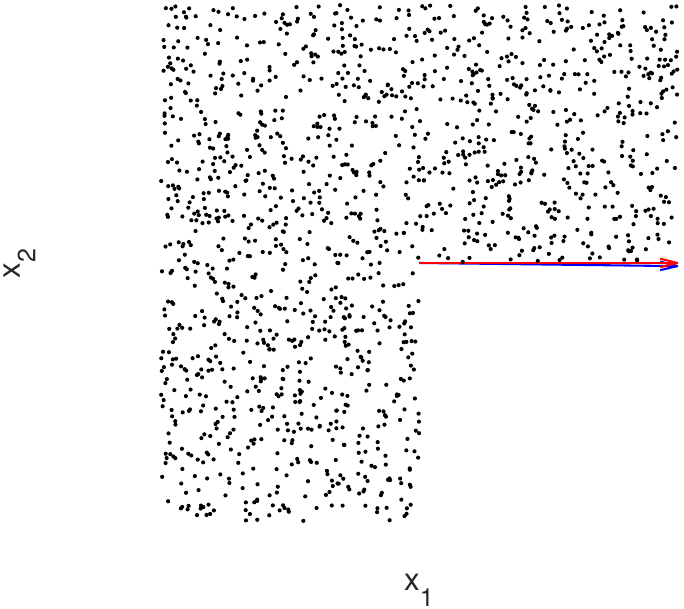}}\hspace{0.2cm}
  \subfloat[SCR]{\includegraphics[width=0.3\textwidth]{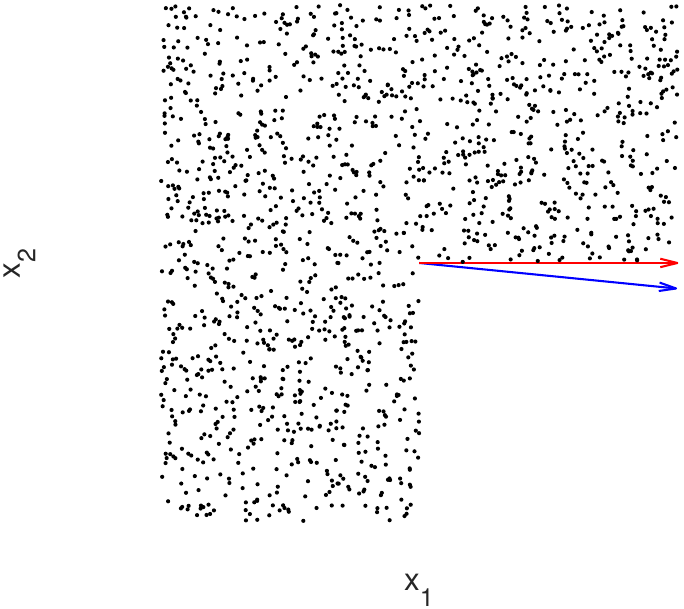}}\hspace{0.2cm}
  \subfloat[SIR]{\includegraphics[width=0.3\textwidth]{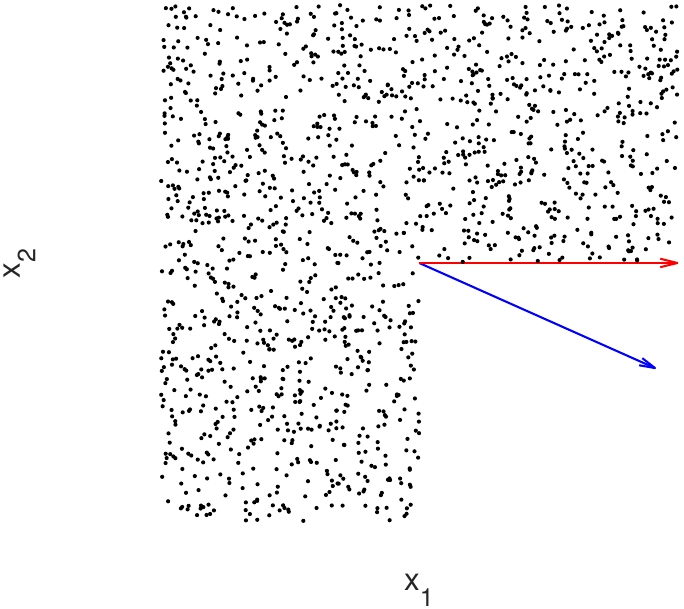}}
  \caption{(Experiment 1) central subspace estimation by GCR, SCR and SIR when $\bx$ is not elliptically distributed. Samples are displayed in black dots and the direction of $\Phi$ is represented by a red arrow. The direction of $\widehat{\Phi}$ is shown in a blue arrow in (a) by GCR, (b) by SCR and (c) by SIR. We observe that GCR is more robust than SCR and SIR when $\bx$ is not elliptically distributed.
  }\label{fig.ellip_compair}
\end{figure}

\subsection{Experiment 2: Monotonic functions}
\begin{figure}[t!]
  \centering
  \subfloat[Central subspace estimation error for $f_1$]{\includegraphics[width=0.5\textwidth]{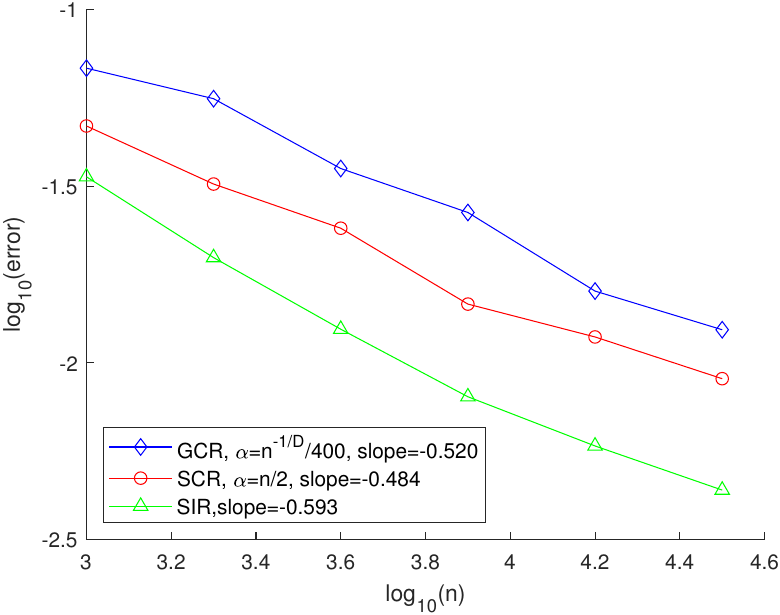}}
  \subfloat[Central subspace estimation error for $f_2$]{\includegraphics[width=0.5\textwidth]{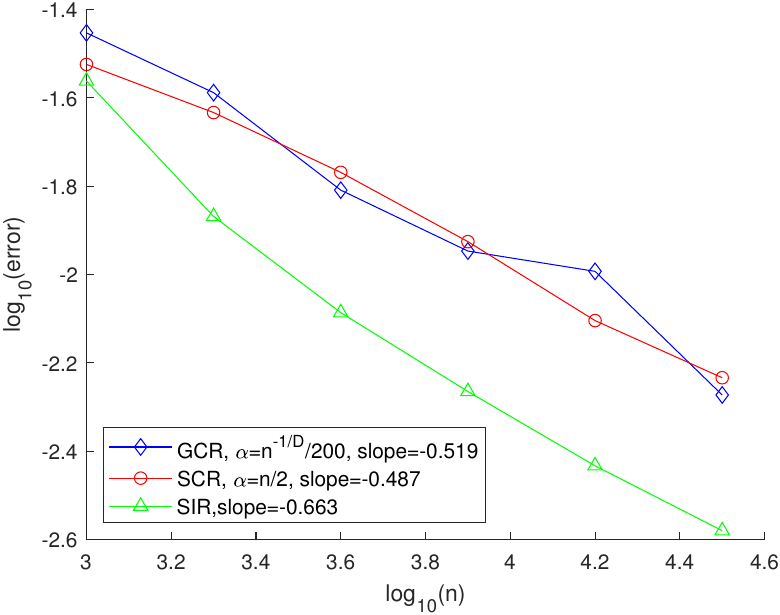}}
  \caption{(Experiment 2 -- comparison of GCR, SCR and SIR with $5\%$ noise.) Log-log plot of the central subspace estimation error versus $n$ for $f_1$ (a) and $f_2$ (b). SIR has the best performance when the function $g$ is monotonic.
}\label{fig.single}
\end{figure}
In the second experiment, we test and compare the performance of GCR, SCR and SIR on two monotonic single index models:
\begin{align*}
  f_1(\bx)=\left(\frac{x_1+x_2}{2}\right)^3,\ f_2(\bx)=e^{(x_1+x_2)/2},\ \bx\in\RR^{10}.
\end{align*}
Both functions can be expressed in model (\ref{eqyi}) with $D=10$ and $d=1$. They are $f_1(\bx)=g_1(z)=(z/\sqrt{2})^3$ and $f_2(\bx)=g_2(z)=e^{z/\sqrt{2}}$ with $z=\Phi^T\bx$ where $\Phi=\frac{1}{\sqrt{2}}(\be_1+\be_2)$. Here $\be_i \in \RR^D$ has $1$ in the $i$th entry and $0$ everywhere else. In this and the following experiments, the $\bx_i$'s are uniformly sampled from $[-1,1]^{D}$, and the sample size varies such that $n=10^3,10^{3.3}, 10^{3.6}, 10^{3.9}, 10^{4.2},10^{4.5}$.

We compare GCR,SCR and SIR with $5\%$ noise. Figure \ref{fig.single} shows the log-log plot of the central subspace estimation error versus $n$ for each function. In GCR, we use $\alpha=n^{-D}/200$ for $f_1$ and $\alpha=n^{-D}/400$ for $f_2$. For both functions, $\alpha=n/2$ is used in SCR and each slice in SIR is set to contain about 200 samples. The subspace error in SCR and GCR converges almost in the order of $n^{- 1 /2}$ as expected. When $g$ is monotonic, SIR yields the best performance. We will show in the next following experiments that SIR can easily fail when $g$ is not monotonic, while GCR can handle many more cases.

\subsection{Experiment 3: A non-monotonic function}

The third experiment is
\begin{equation}
  f(\bx)=\sin\left(-\frac{\pi}{2}+\frac{\pi}{6}\sum_{i=1}^9 x_i\right)+x_{10}, \ \bx \in \RR^{10}.
\end{equation}
This function can be expressed as the model in (\ref{eqyi}) with $D=10, d=2$, $g(z_1,z_2)= \sin\left(-\frac{\pi}{2}+\frac{\pi}{2}z_1\right)+z_2$, $\bz = \Phi^T \bx, \Phi=[\bv_1,\bv_2]$ and
$
\bv_1=\frac{1}{3}\sum_{i=1}^9 \be_i,\ \bv_2=\be_{10}.
$

\begin{figure}[h]
  \centering
  \subfloat[Central subspace estimation error]{\includegraphics[width=0.5\textwidth]{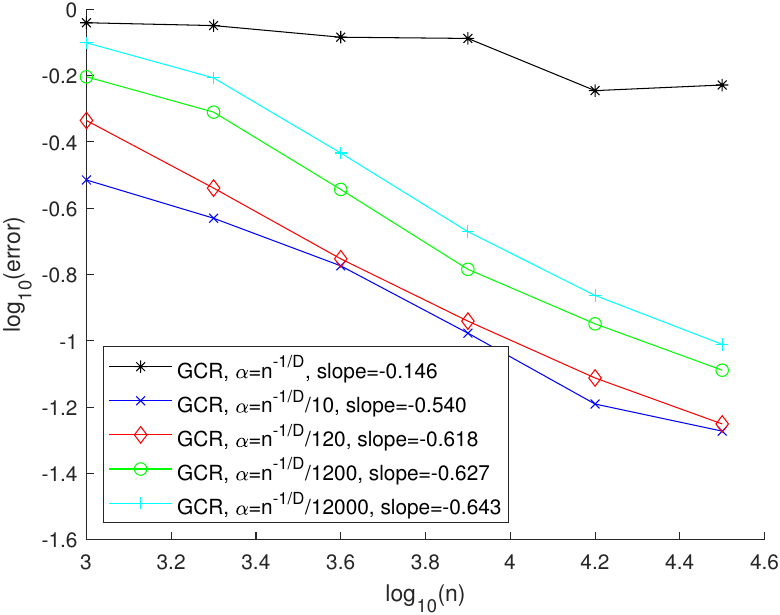}}
  \subfloat[Regression error]{\includegraphics[width=0.5\textwidth]{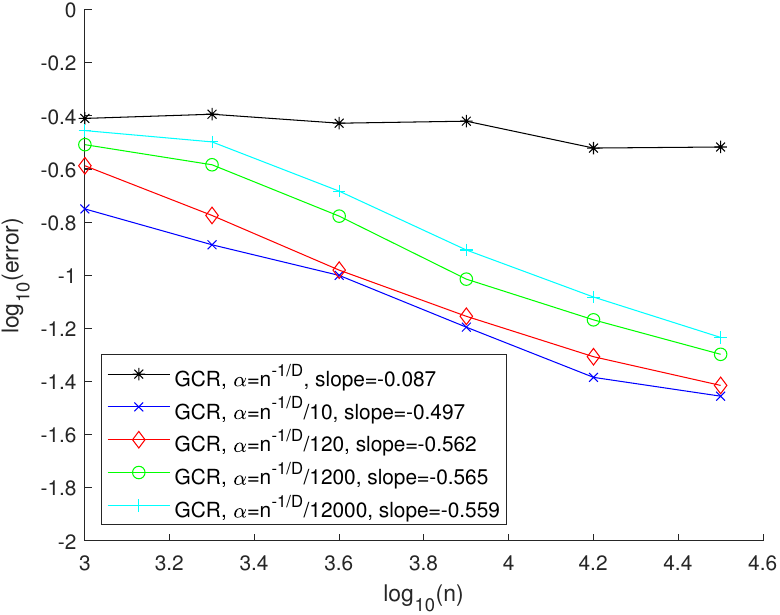}}
  \caption{(Experiment 3 -- performance of GCR without noise) Log-log plot of the central subspace estimation error (a) and the regression error (b) versus $n$, where $\alpha=Cn^{-1/D}$ in GCR with $C=1, 1/10,1/120,1/1200,1/12000$ respectively.
}\label{fig.sinCons}
\end{figure}

We first present the performance of GCR on noiseless data, i.e. $\sigma =0$, with different choices of the parameter $\alpha$. In GCR, the parameter $\alpha$ should be chosen as $\alpha=Cn^{-1/D}$ according to (\ref{eqalpha}). We set $C=1,1/10,1/120,1/1200,1/12000$ and show the central subspace estimation error versus $n$ in Figure \ref{fig.sinCons}(a) and the regression error versus $n$ in Figure \ref{fig.sinCons}(b). Each error is averaged over $10$ experiments. In log-log scale, the errors decay linearly as $n$ increases with the same rate as long as $C$ is not too large, as we expect. Our theory predicts the slopes in Figure \ref{fig.sinCons}(a) as $-0.5$, which are almost matched by the slopes in Figure \ref{fig.sinCons}(a). The success of GCR requires the condition $\alpha<\alpha_{\rm thresh}$, so GCR fails when $C$ is too large, i.e. $C =1$. The regression error is observed to converge in the order of $n^{-0.5}$ as long as $C$ is not too large.

\begin{figure}[t!]
  \centering
  \subfloat[Central subspace estimation error]{\includegraphics[width=0.5\textwidth]{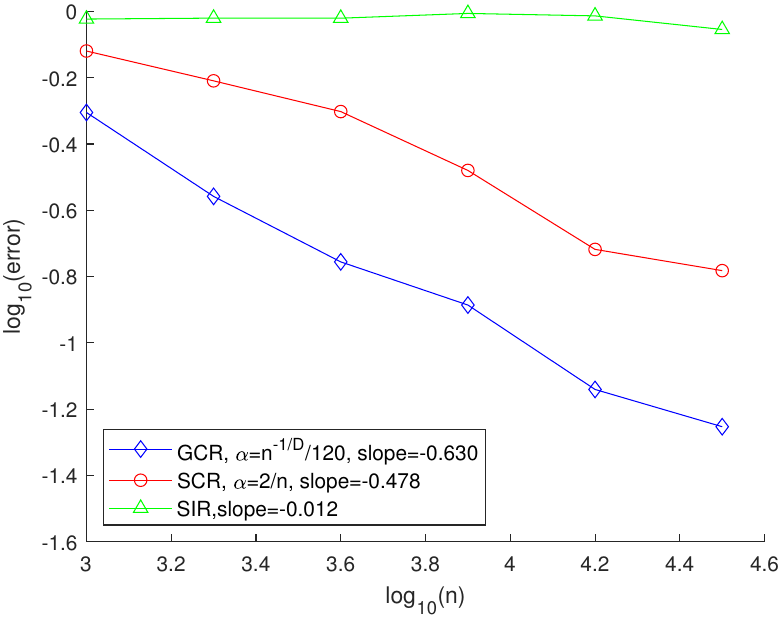}}
  \subfloat[Regression error]{\includegraphics[width=0.5\textwidth]{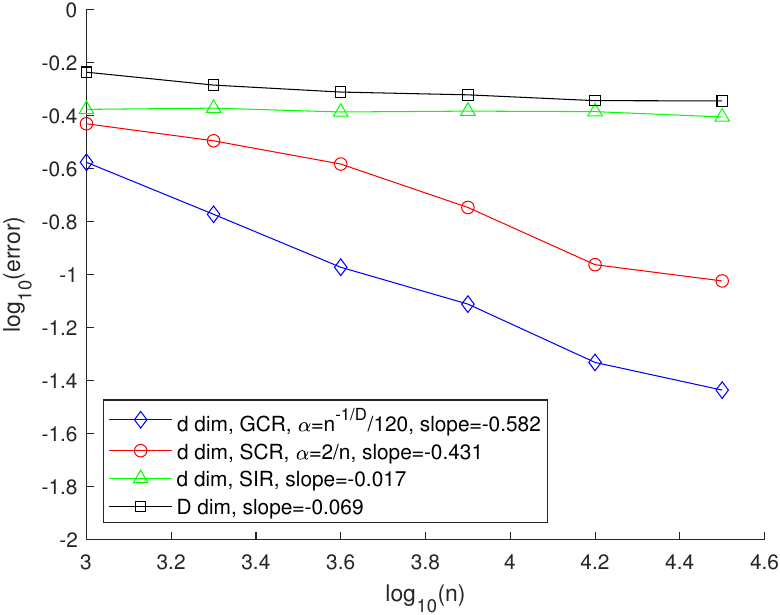}}
  \caption{(Experiment 3 -- Comparison of GCR, SCR and SIR without noise) Log-log plot of the central subspace estimation error (a) and the regression error (b) versus $n$ by GCR, SCR and SIR.
  The black curve in (b) represents the regression error of $f$ in $\RR^D$ without an estimation of the central subspace.
  GCR and SCR perform better than SIR, and GCR is the best. Estimating the central subspace greatly improves the rate of convergence in comparison with a direction regression in $\RR^D$.
  }\label{fig.sin}
\end{figure}

We next compare the performance of GCR, SCR and SIR with noiseless data. We set $\alpha=n^{-1/D}/120$ in GCR, and $\alpha=2/n$ in SCR since it provides the best results among many choices. In SIR, each slice contains about 200 samples. We display the central subspace estimation error versus $n$ in Figure \ref{fig.sin}(a) and the regression error versus $n$ in Figure \ref{fig.sin}(b).
Each error is the averaged error of 10 experiments. GCR and SCR perform better than SIR, and GCR is the best. Estimation of the central subspace greatly improves the rate of convergence in comparison with a direct regression in $\RR^D$.


\begin{figure}[t!]
  \centering
  \subfloat[Central subspace estimation error]{\includegraphics[width=0.5\textwidth]{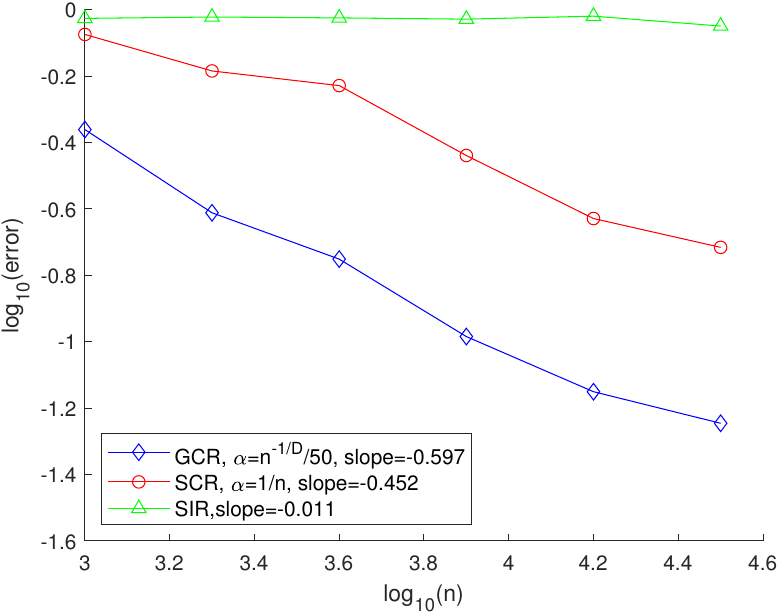}}
  \subfloat[Regression error]{\includegraphics[width=0.5\textwidth]{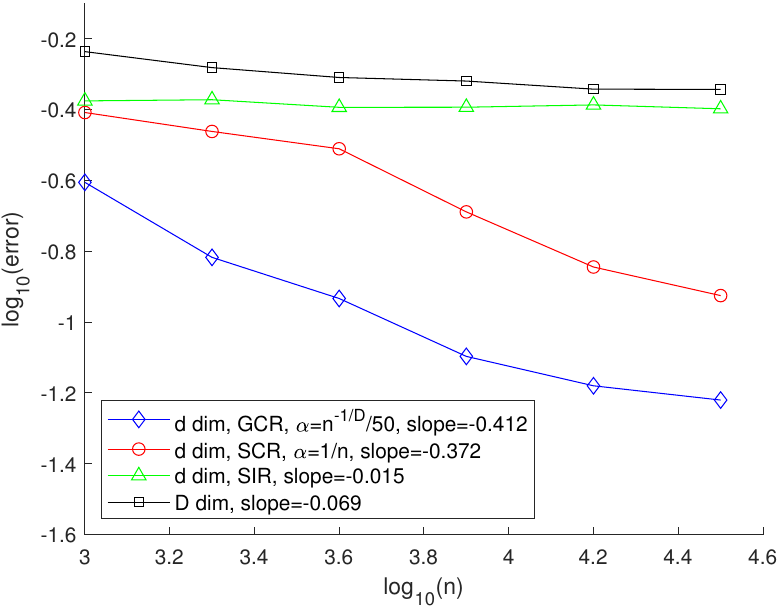}}
  \caption{(Experiment 3 -- Comparison of GCR, SCR and SIR with $5\%$ noise)
  Log-log plot of the central subspace estimation error (a) and the regression error (b) versus $n$ by GCR, SCR and SIR.
  The black curve in (b) represents the regression error of $f$ in $\RR^D$ without an estimation of the central subspace.
  GCR and SCR perform better than SIR, and GCR is the best. Estimating the central subspace greatly improves the rate of convergence in comparison with a direction regression in $\RR^D$.
  }\label{fig.sinNoise}
\end{figure}

Results with $5\%$ noise are shown in Figure \ref{fig.sinNoise}. We set $\alpha=n^{-1/D}/50$ and $\alpha=1/n$ in GCR and SCR. In SIR, each slice contains about 200 samples. 
We observe that GCR perform better than SCR and SIR.

\begin{figure}[t!]
  \centering
  \subfloat[($50\%$ noise) Noiseless $f(\bx)$ in black and noisy $y$ in gray along the $\bv_1$ direction]{\includegraphics[width=0.3\textwidth]{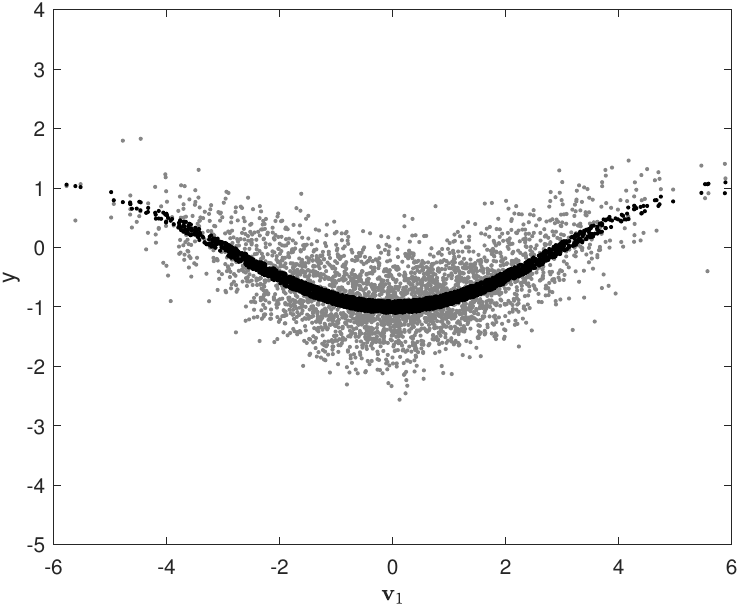}}\hspace{0.2cm}
  \subfloat[($100\%$ noise) Noiseless $f(\bx)$ in black and noisy $y$ in gray along the $\bv_1$ direction]{\includegraphics[width=0.3\textwidth]{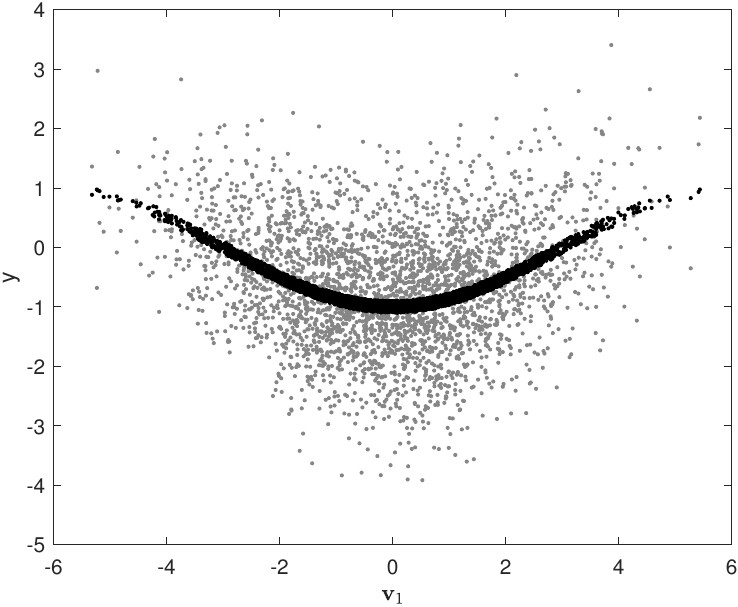}}\hspace{0.2cm}
  \subfloat[($120\%$ noise) Noiseless $f(\bx)$ in black and noisy $y$ in gray along the $\bv_1$ direction]{\includegraphics[width=0.3\textwidth]{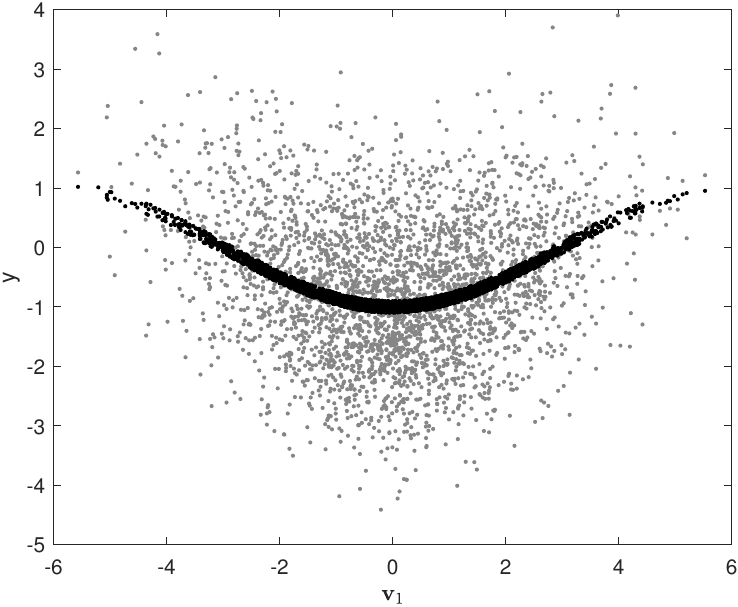}}\\
  \subfloat[($50\%$ noise) Noiseless $f(\bx)$ in black and noisy $y$ in gray along the $\bv_2$ direction]{\includegraphics[width=0.3\textwidth]{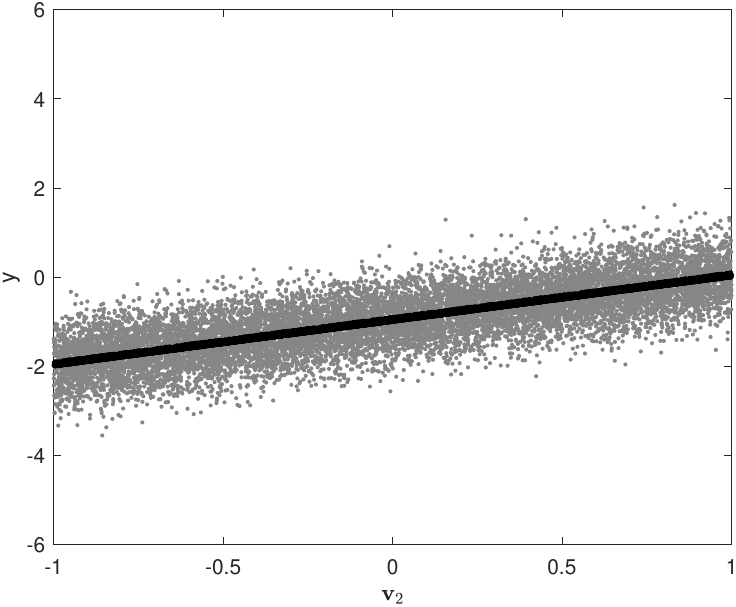}}\hspace{0.2cm}
  \subfloat[($100\%$ noise) Noiseless $f(\bx)$ in black and noisy $y$ in gray along the $\bv_2$ direction]{\includegraphics[width=0.3\textwidth]{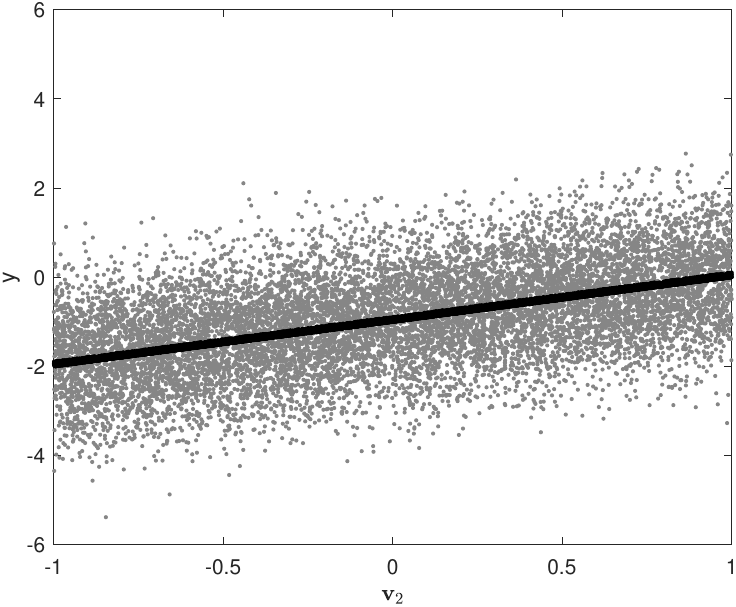}}\hspace{0.2cm}
  \subfloat[($120\%$ noise) Noiseless $f(\bx)$ in black and noisy $y$ in gray along the $\bv_2$ direction]{\includegraphics[width=0.3\textwidth]{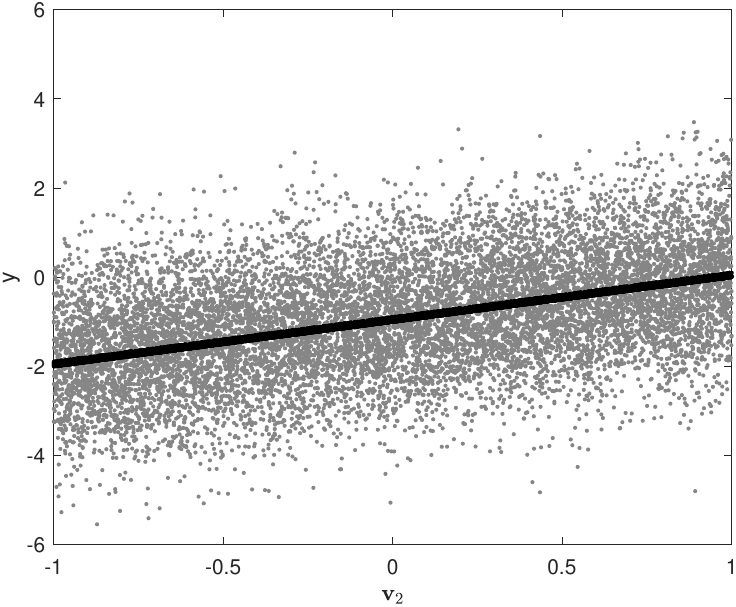}}\\
  \subfloat[($50\%$ noise) Central subspace estimation error]{\includegraphics[width=0.3\textwidth]{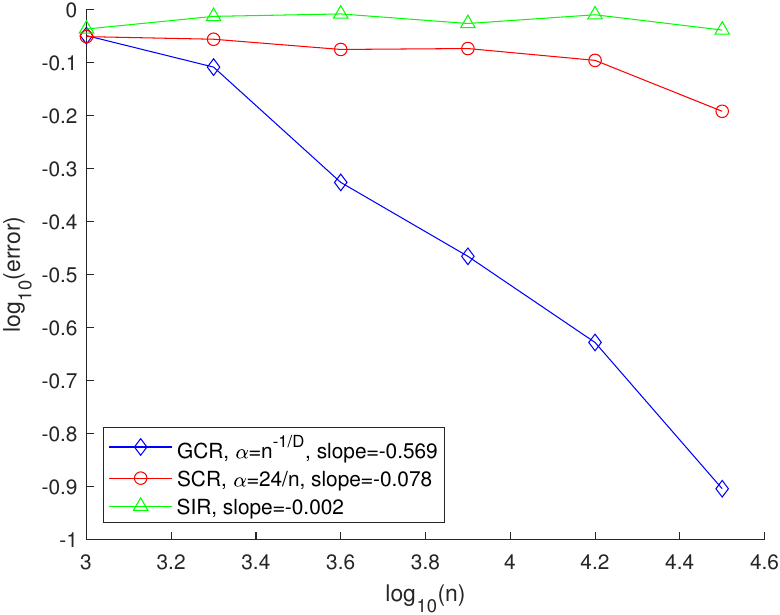}}\hspace{0.2cm}
  \subfloat[($100\%$ noise) Central subspace estimation error]{\includegraphics[width=0.3\textwidth]{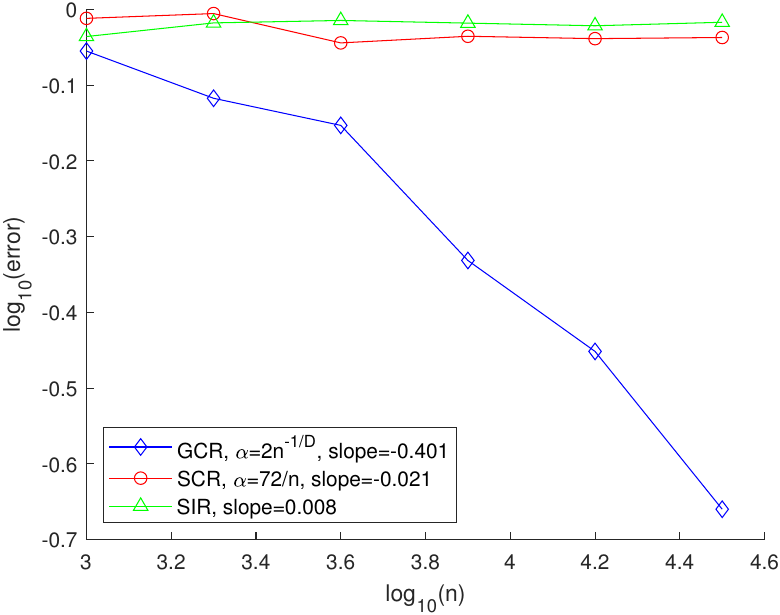}} \hspace{0.2cm}
  \subfloat[($120\%$ noise) Central subspace estimation error]{\includegraphics[width=0.3\textwidth]{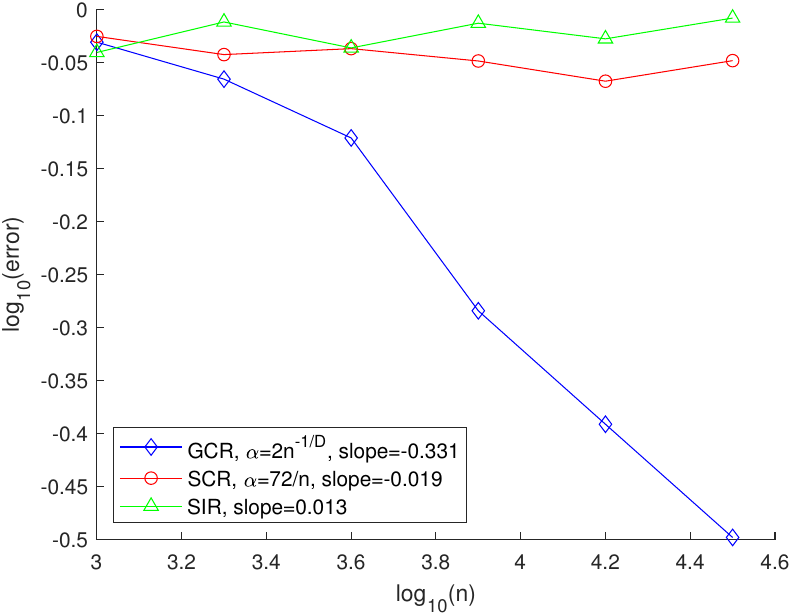}} \hspace{0.2cm}
  \caption{(Experiment 3 -- Comparison of GCR, SCR  and SIR 
  with heavy noise: $50\%$ noise in the left column, $100\%$ in the middle column and $120\%$ in the right column)
  The first row shows the visualization of the noiseless $f(\bx)$ in black and the noisy $y$ in gray along the $\bv_1$ direction while $-0.1<x_{10}<0.1$.
The second row shows the visualization of the noiseless $f(\bx)$ in black and the noisy $y$ in gray along the $\bv_2$ direction while $-0.9<\sum_1^9x_i<0.9$.
The third row shows the log-log plot of the central subspace estimation error by GCR, SCR and SIR versus $n$.
  }\label{fig.sinNoiseH}
\end{figure}

We then compare GCR, SCR and SIR with heavy noise -- $50\%, 100\%$ and $120\%$ noise. Visualization of data 
along the $\bv_1$ and $\bv_2$ directions is shown in the first and second row of Figure \ref{fig.sinNoiseH}.
The third row shows the log-log plot of the central subspace estimation error by GCR, SCR and SIR versus $n$.
Each error is averaged over $10$ experiments. We observe that GCR is very robust against heavy noise -- the central subspace estimation error converges in the order of $n^{-0.569}$ with $50\%$ noise, and the rate slightly degrades in the presence of $100\%$ and $120\%$ noise. In comparison, SCR and SIR tend to fail when noise is heavy.

\subsection{Experiment 4: A non-monotonic function}
The fourth experiment is
\begin{equation}
  f(\bx)=x_1^2+(x_2+x_3)^2, \ \bx \in \RR^{10}
\end{equation}
This function can be expressed as the model in (\ref{eqyi}) with $d=2,D=10$, $g(z_1,z_2)=z_1^2+2z_2^2$, $\bz=\Phi^T\bx,\Phi=[\bv_1,\bv_2]$ and $\bv_1=\be_1,  \bv_2=(\be_2+\be_3)/{\sqrt 2}.$
This experiment is more challenging since $f$ is not monotonic along both $\bv_1$ and $\bv_2$ directions.

\begin{figure}[h!]
  \centering
  \subfloat[Central subspace estimation error]{\includegraphics[width=0.5\textwidth]{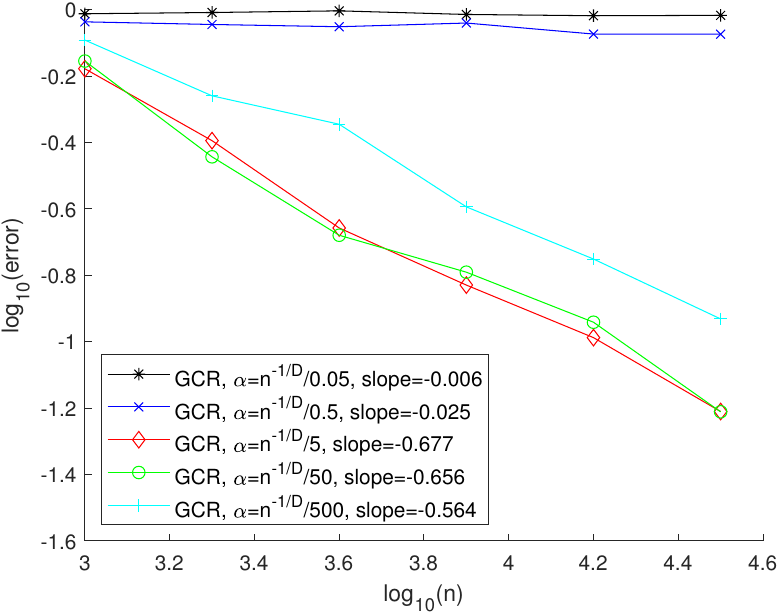}}
  \subfloat[Regression error]{\includegraphics[width=0.5\textwidth]{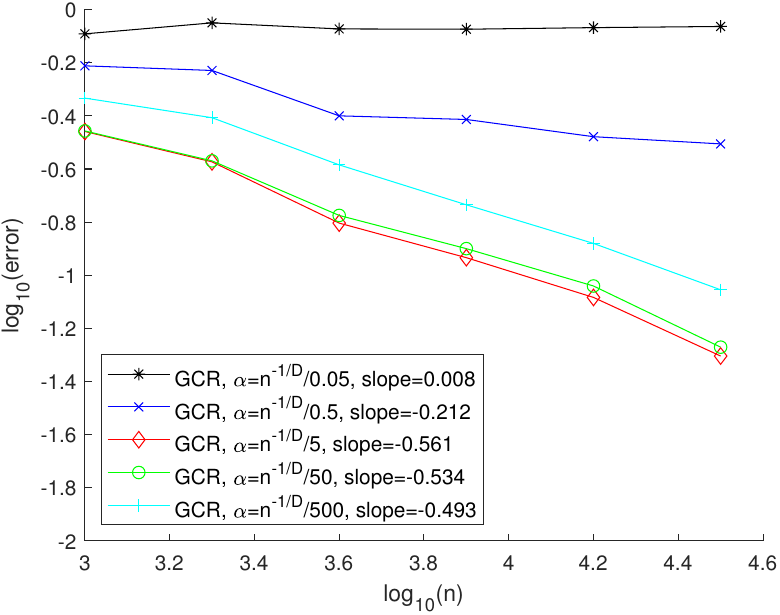}}
  \caption{(Experiment 4 -- performance of GCR without noise) Log-log plot of the central subspace estimation error (a) and the regression error (b) versus $n$, where $\alpha=Cn^{-1/D}$ in GCR with $C=1/0.05,1/0.5,1/5,1/50,1/500$ respectively.}\label{fig.quadCons}
\end{figure}

Performance of GCR with $\alpha=Cn^{-1/D}$ where $C=1/0.05,1/0.5,1/5,1/50,1/500$ on noiseless data is presented in Figure \ref{fig.quadCons}. The central subspace estimation error and the regression error are shown in Figure \ref{fig.quadCons} (a) and (b), respectively. Each error is averaged over 10 experiments. Our observation is similar to Experiment 2 that, in log-log scale, both errors decay linearly as $n$ increases with the same slope, as long as $C$ is not too large.




In Figure \ref{fig.quadNoise}, we compare GCR, SCR and SIR with $5\%$ noise. We set $\alpha=n^{-1/D}/50$ in GCR and $\alpha=1/n$ in SCR. In SIR, each slice contains about 200 samples. We show the central subspace estimation error versus $n$ in Figure \ref{fig.quadNoise} (a) and the regression error versus $n$ in Figure \ref{fig.quadNoise} (b). Each error is averaged over 10 experiments. Among the three methods, GCR yields the smallest error and the fastest rate of convergence.

\begin{figure}[h!]
  \centering
  \subfloat[Central subspace estimation error]{\includegraphics[width=0.5\textwidth]{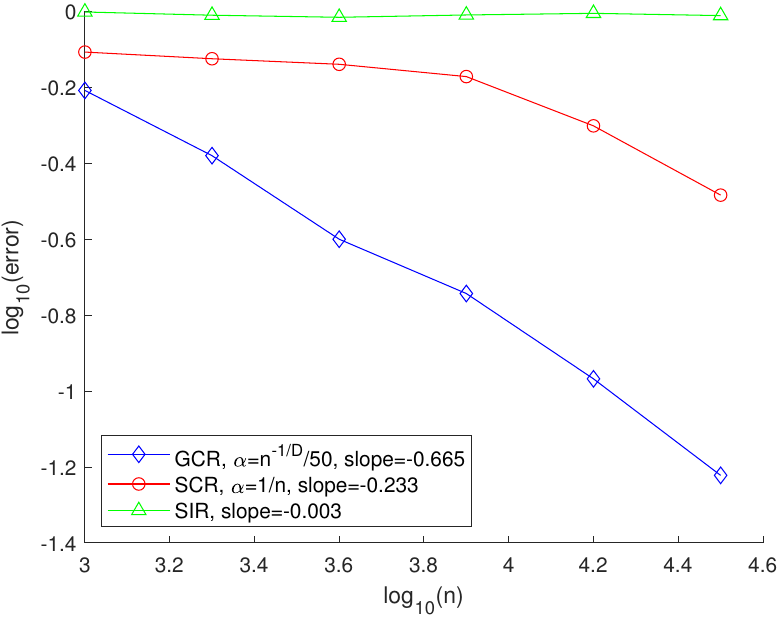}}
  \subfloat[Regression error]{\includegraphics[width=0.5\textwidth]{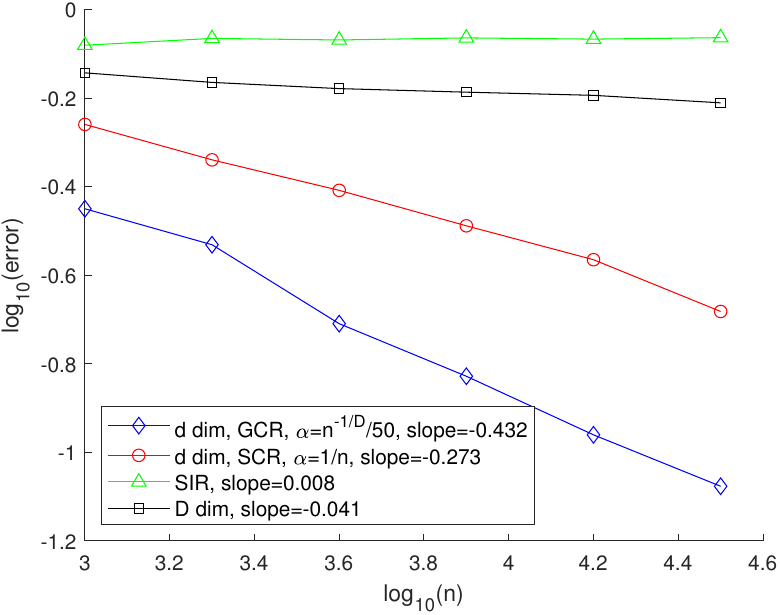}}
  \caption{(Experiment 4 -- Comparison of GCR, SCR and SIR with $5\%$ noise)
  Log-log plot of the central subspace estimation error (a) and the regression error (b) versus $n$ by GCR and SCR.
  The black curve in (b) represents the regression error of $f$ in $\RR^D$ without an estimation of the central subspace.
  GCR performs better than SCR and SIR. Estimation of the central subspace greatly improves the rate of convergence in comparison with direction regression in $\RR^D$.
  }\label{fig.quadNoise}
\end{figure}

Results with $50\%$ noise are shown in Figure \ref{fig.quadNoise50}. We set $\alpha=n^{-1/D}$ in GCR and $\alpha=24/n$ in SCR. In SIR, each slice contains about 200 samples. 
The noisy $y$ (gray) and the noiseless $f(\bx)$ (black) are shown in Figure \ref{fig.quadNoise50} (a) along the $\bv_1$ direction with $-0.2<x_2+x_3<0.2$, and in Figure \ref{fig.quadNoise50} (b) along the $\bv_2$ direction with $-0.1<x_1<0.1$.
The central subspace estimation error by the three methods are displayed in Figure \ref{fig.quadNoise50} (c). Each error is averaged over $10$ experiments. SCR and SIR perform poorly in this test, while GCR is very robust to heavy noise. The central subspace estimation error decays as $O(n^{-0.43})$, which is close to our prediction of $O(n^{-0.5})$.


\begin{figure}[h!]
  \centering
  \subfloat[($50\%$ noise) Noiseless $f(\bx)$ in black and noisy $y$ in gray along the $\bv_1$ direction]{\includegraphics[width=0.3\textwidth]{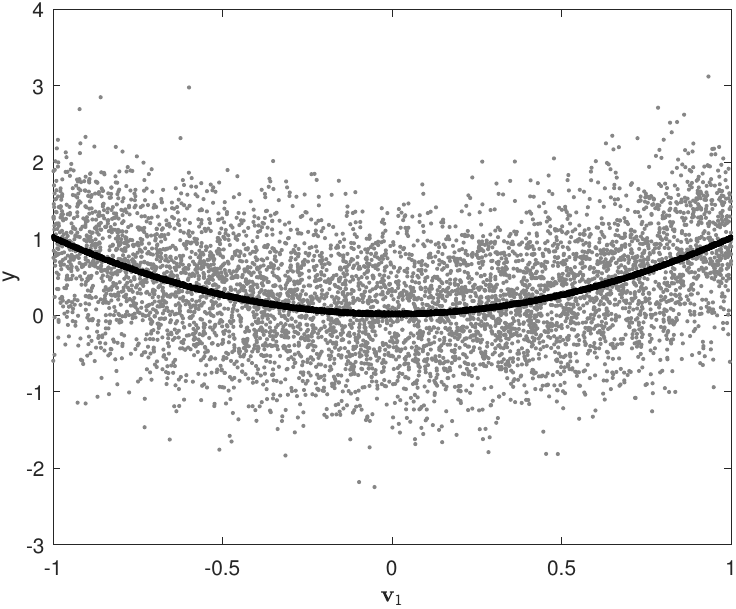}}\hspace{0.2cm}
  \subfloat[($50\%$ noise) Noiseless $f(\bx)$ in black and noisy $y$ in gray along the $\bv_2$ direction]{\includegraphics[width=0.3\textwidth]{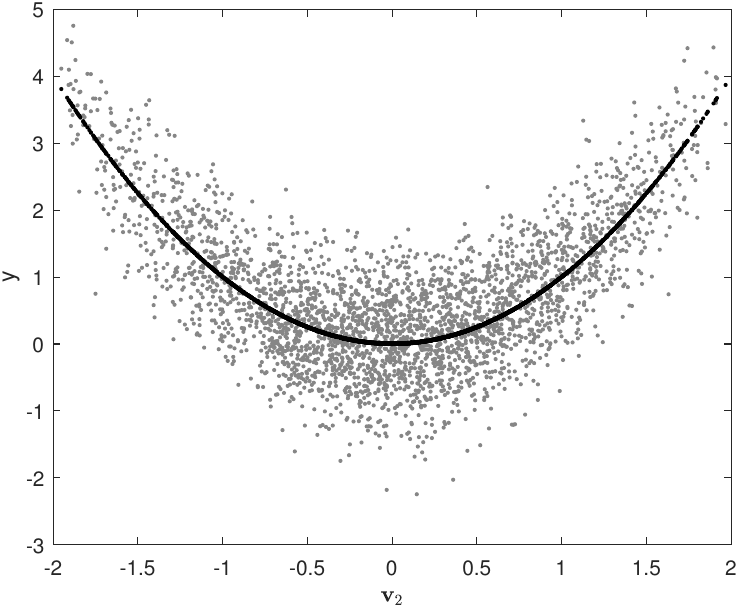}}\hspace{0.2cm}
  \subfloat[Central subspace estimation error]{\includegraphics[width=0.3\textwidth]{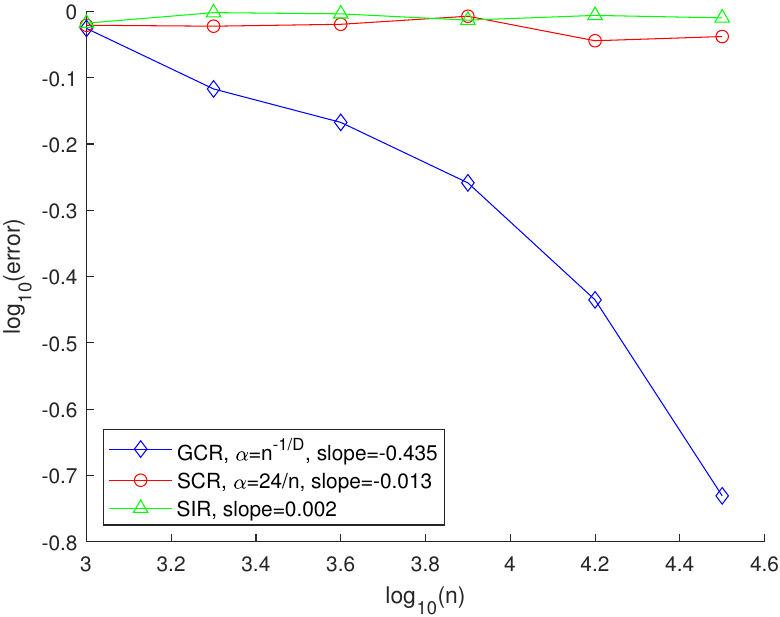}}
  \caption{(Experiment 4 -- Comparison of GCR, SCR and SIR 
  with $50\%$ noise)
  Visualization of noiseless $f(\bx)$ in black and the noisy $y$ in gray along the $\bv_1$ direction while $-0.2<x_2+x_3<0.2$ (a) and along the $\bv_2$ direction while $-0.1<x_1<0.1$ (b). (c) is the log-log plot of the central subspace estimation error versus $n$ by GCR, SCR and SIR.
  GCR performs better than SCR and SIR. GCR is very robust against heavy noise.
  }\label{fig.quadNoise50}
\end{figure}

\subsection{Experiment 5: A non-monotonic function}
In Experiment 3 and 4, the function $f$ can be written as a sum of two single index models. We next compare GCR, SCR and SIR on functions without such structures. Consider
\begin{align*}
  &f(\bx)=10\sin\left(\frac{\pi}{5}\left(x_1+2(x_2+x_3)^2\right)\right), \ \bx \in \RR^{10}
\end{align*}
which can be expressed in the model (\ref{eqyi}) with $D=10, d=2,$
$g(z_1,z_2)=10\sin\left(\frac{\pi}{5}\left(z_1+z_2^2\right)\right)$, $\bz=\Phi^T\bx, \Phi=[\bv_1,\bv_2]$ and
$\bv_1=\be_1,  \bv_2=(\be_2+\be_3)/{\sqrt 2}.$
We test GCR, SCR and SIR with $5\%$ and $50\%$ noise. 
The log-log plot of the central subspace estimation error versus $n$ is displayed in Figure \ref{fig.sinmulti4}. Each error is averaged over $10$ experiments. GCR has a convergence rate of $O(n^{-0.565})$ with $5\%$ noise and $O(n^{-0.447})$ with $50\%$ noise, 
which is robust to heavy noise and performs better than SCR and SIR.

\begin{figure}[h!]
  \subfloat[($5\%$ noise) Central subspace estimation error]{\includegraphics[width=0.46\textwidth]{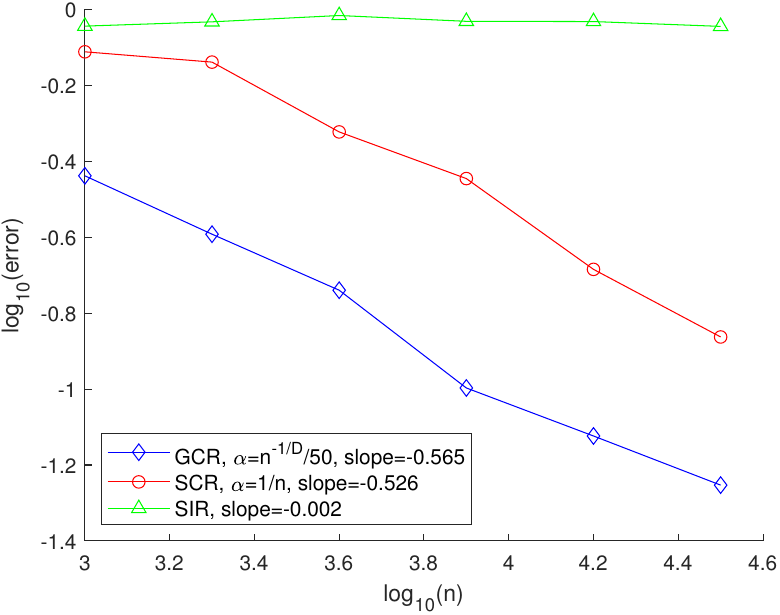}} \hspace{0.2cm}
  \subfloat[($50\%$ noise) Central subspace estimation error]{\includegraphics[width=0.46\textwidth]{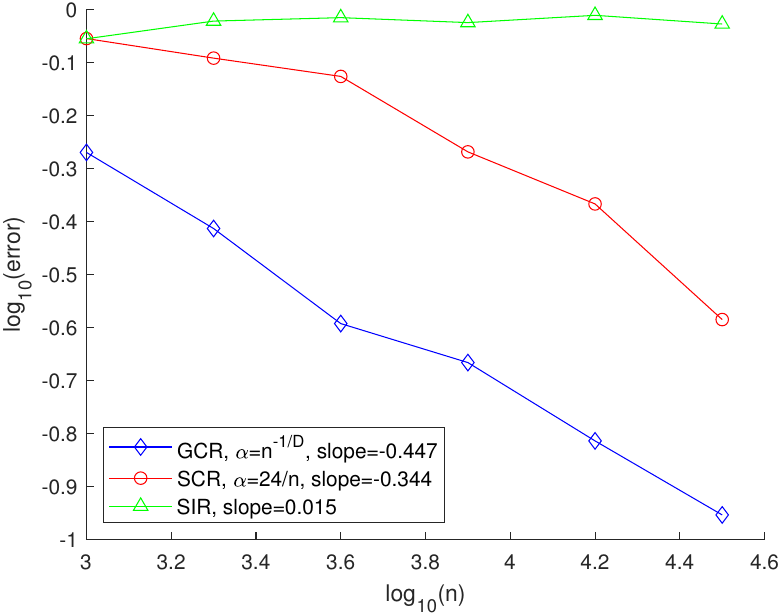}}
  \caption{(Experiment 5 -- Comparison of GCR, SCR and SIR with $5\%$ and $50\%$ noise) Log-log plot of the central subspace estimation error versus $n$ by GCR, SCR and SIR with $5\%$ noise (a) and $50\%$ noise (b).
  }\label{fig.sinmulti4}
\end{figure}

\subsection{Experiment 6: A non-monotonic function}
Our last experiment is
\begin{align*}
  &f(\bx)=2x_1^2(x_2+x_3)^2, \ \bx \in \RR^{10}
\end{align*}
which can be expressed in the model (\ref{eqyi}) with $D=10, d=2,$
$g(z_1,z_2)=4z_1^2z_2^2$, $\bz=\Phi^T\bx, \Phi=[\bv_1,\bv_2]$ and $\bv_1=\be_1, \bv_2=(\be_2+\be_3)/{\sqrt 2}.$

We test GCR, SCR and SIR with $5\%$ and $50\%$ noise.
The log-log plot of the central subspace estimation error versus $n$ is displayed in Figure \ref{fig.sinmulti5}. Each error is averaged over $10$ experiments. GCR has a convergence rate of $O(n^{-0.575})$ with $5\%$ noise and $O(n^{-0.523})$ with $50\%$ noise, which is robust to heavy noise and performs significantly better than SCR and SIR.

\begin{figure}[t!]
  \subfloat[($5\%$ noise) Central subspace estimation error]{\includegraphics[width=0.46\textwidth]{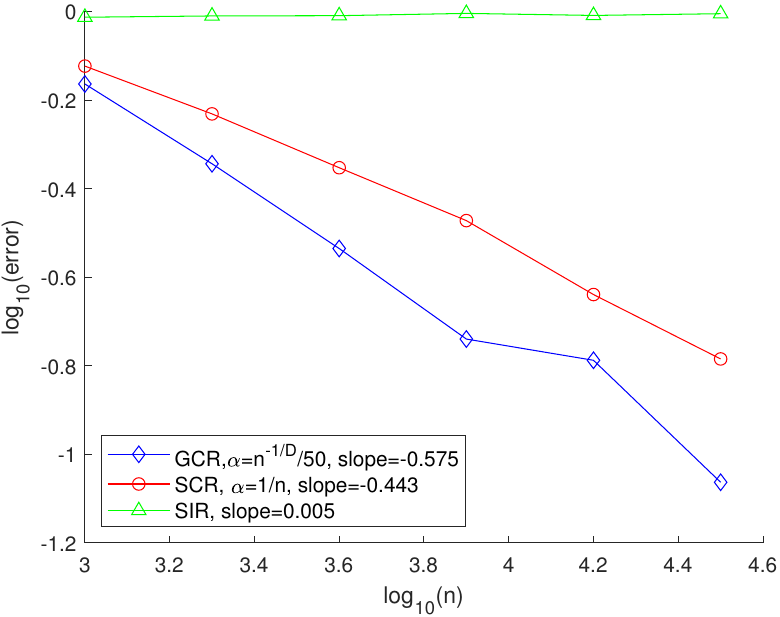}} \hspace{0.2cm}
  \subfloat[($50\%$ noise) Central subspace estimation error]{\includegraphics[width=0.46\textwidth]{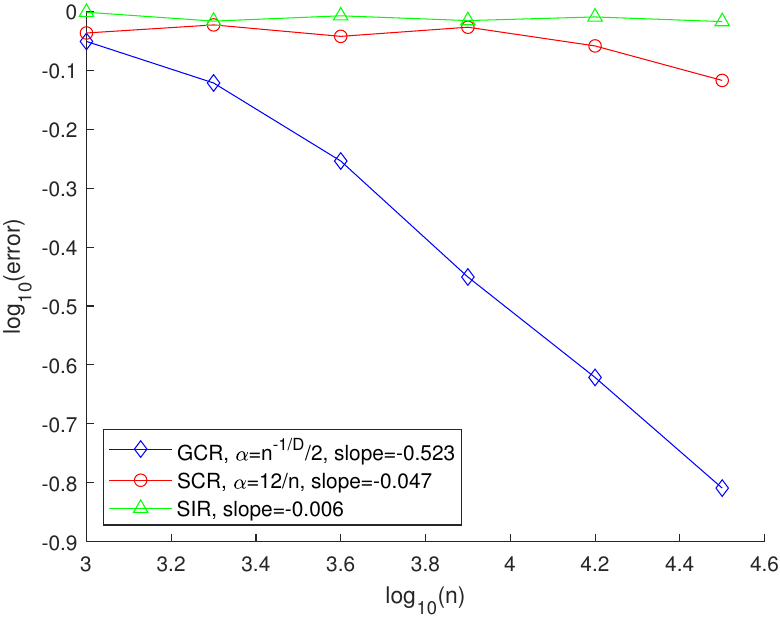}}
  \caption{(Experiment 6 -- Comparison of GCR, SCR and SIR with $5\%$ and $50\%$ noise.) Log-log plot of the central subspace estimation error versus $n$ by GCR, SCR and SIR with $5\%$ noise (a) and $50\%$ noise (b).
  GCR performs better than SCR and SIR.}\label{fig.sinmulti5}
\end{figure} 

\section{Proof of main results}
\label{sec.mainproof}
This section contains the proof of our main results in Section \ref{sec.MainResults}.
Lemmas are proved in the Supplementary materials. 
\subsection{Concentration inequality for matrix-valued U-statistics and proof of Theorem \ref{thm.HDisVf}}
\label{sec.mainproof.matrixU}
U-statistics was introduced by Hoeffding \cite{hoeffding1992class} and the concentration inequalities for scalar-valued U-statistics  have been well studied \cite{hoeffding1992class},\cite[Chapter 5]{serfling2009approximation}. To our knowledge, there are limited results on matrix-valued U-statistics. The concentration inequality for a modified matrix-valued U-statistics can be found in \cite{minsker2020robust}. In this paper, we prove a Bernstein inequality for matrix-valued U-statistics using the tools in \cite{tropp2012user,ahlswede2002strong}. 
We denote $\mathbb{H}^{D \times D}$ as the set of $D \times D$ real-valued Hermitian matrices.

\begin{defi}[Matrix U-statistics]
Let $\bx$ be a random variable with probability distribution $\rho$ in $\RR^D$, and $W: \RR^D\times \cdots \times \RR^D \rightarrow \mathbb{H}^{D \times D}$ be a matrix-valued kernel with $m$ inputs. Suppose we can access $n \ge m$ i.i.d. copies of $\bx$, denoted by $\{\bx_i\}_{i=1}^n$.  The U-statistic of order $m$ and kernel $W$ based on $\{\bx_i\}_{i=1}^n$ is defined as
  \begin{align}
    U_n=\frac{(n-m)!}{n!}\sum_{(i_1,...,i_m)\in I_m^n} W(\bx_{i_1},...,\bx_{i_m})
    \label{ustateq1}
  \end{align}
  where
  $$
  I_m^n=\{(i_1,...,i_m)| 1\leq i_j\leq n, i_j\neq i_k \mbox{ for } j\neq k\}
  $$
  is the set of all $m$-tuple of integers between 1 and $n$.
\end{defi}

We say $W$ is permutation symmetric if for any $(\bx_1,...,\bx_m)$ and any permutation $\pi$, 
$W(\bx_1,...,\bx_m)=W(\bx_{\pi_1},...,\bx_{\pi_m})$.
When $W$ is permutation symmetric, $U_n$ can be written as
\begin{align}
    U_n=\frac{1}{\binom{n}{m}}\sum_{1\leq i_1<\cdots<i_m\leq n} W(\bx_{i_1},...,\bx_{i_m}).
    \label{ustateq2}
\end{align}

U-statistics are unbiased estimators of $\EE W(\bx_1,...,\bx_m)$ where the expectation is taken over the joint distribution of $\bx_1,...,\bx_m$. 
We use tools from \cite{tropp2012user,ahlswede2002strong} to prove the following lemma about the concentration of the matrix-valued U-statistics:

\begin{lem}[Bernstein inequality for matrix-valued U-statistics]
\label{lem.BernsteinU}
Let $\bx$ be a random variable with probability distribution $\rho$ in $\RR^D$, and $W: \RR^D\times \cdots \times \RR^D \rightarrow \mathbb{H}^{D \times D}$ be a permutation symmetric matrix-valued kernel with $m$ inputs. 
Assume
\begin{equation}
\label{eqwcondition}
\|W(\bx_1,...,\bx_m)\|\leq R \text{ for any }(\bx_1,...,\bx_m), \text{ and } \|\EE\left[ (W-\EE W)(W-\EE W)^*\right]\|\leq \sigma_W^2.
\end{equation} 
Suppose we can access $n \ge m$ i.i.d. copies of $\bx$, denoted by $\{\bx_i\}_{i=1}^n$. Let $U_n$ be the U-statistics of $W$ defined in (\ref{ustateq2}). 
Then for any $t>0$,
  \begin{align}
  \PP(\|U_n-\EE U_n\|\geq t)\leq D\exp\left(-\frac{nt^2/4}{m(\sigma_W^2+Rt/3)}\right).
\end{align}
\end{lem}

\begin{proof}[Proof of Lemma \ref{lem.BernsteinU}]
Let $k=\left\lfloor \frac{n}{m}\right\rfloor$. Define
\begin{align}
  V(\bx_1,...,\bx_n)&=\frac{1}{k}\left(\underbrace{W(\bx_1,...,\bx_m)}_{V_1}+\cdots+\underbrace{W(\bx_{(i-1)m+1},...,\bx_{im})}_{V_i}+\cdots +\underbrace{W(\bx_{(k-1)m+1},...,\bx_{km})}_{V_k} \right)
  \label{eq.V}
\end{align}
where $V_i=W(\bx_{(i-1)m+1},\bx_{(i-1)m+2},...,\bx_{im})$. Note that $V$ is the average of $k$ independent empirical realizations of $W$.
Let $\pi$ be a permutation of $(1,2,...,n)$ and $V^{\pi}=V(\bx_{\pi_1},...,\bx_{\pi_n})$. 
The U-statistics in \eqref{ustateq1} can be written as
\begin{align}
  U_n=\frac{1}{n!}\sum_{\pi\in I^n}  V^{\pi},
  \label{eq.U}
\end{align}
where $I^n$ denotes the set of all permutations of $(1,2,...,n)$. We have $\EE U_n=\EE V^{\pi}=\EE V$ for any permutation $\pi$. Then
\begin{align}
U_n-\EE U_n=\frac{1}{n!}\sum_{\pi\in I^n} (V^{\pi}-\EE V).
\end{align}
For any $\mu,t>0$, we have
\begin{align}
  \PP(\|U_n-\EE U_n\|\geq t)&=\PP(e^{\mu\|U_n-\EE U_n\|}\geq e^{\mu t})\nonumber\\
  &\leq e^{-\mu t}\EE e^{\mu\|U_n-\EE U_n\|} \quad {\mbox{by Markov's inequality}} \nonumber\\
  &\leq e^{-\mu t}\EE e^{\mu\frac{1}{n!}\sum_{\pi\in I^n}\|V^{\pi}-\EE V\|}\nonumber \\
   &\leq  e^{-\mu t}\frac{1}{n!}\sum_{\pi\in I^n}\EE e^{\mu\|V^{\pi}-\EE V\|} \quad {\mbox{by the convexity of exponential functions}} \nonumber\\
  &=e^{-\mu t}\EE e^{\mu\|V^{\pi}-\EE V\|}\nonumber\\
  &=e^{-\mu t}\EE e^{\mu\|V-\EE V\|}.
  \label{eq.PU}
\end{align}
Next we bound $\EE e^{\mu\|V-\EE V\|}$.
By the assumption in \eqref{eqwcondition}, for each $V_i$ in \eqref{eq.V}, we have
\begin{align}
  \|V_i-\EE V_i\|\leq R, \text{ and } \EE \|(V_i-\EE V_i)(V_i-\EE V_i)^\top\|\leq \sigma_W^2.
\end{align}
The Bernstein inequality for U-statistics is derived through several important results in \cite{tropp2012user,ahlswede2002strong}. For $0<\mu/k< 3/R$,
\begin{align*}
  \EE e^{\mu\|V-\EE V\|}&\leq \EE \ \tr \left(e^{\mu(V-\EE V)}\right)\quad \mbox{by (3.2.3) of \cite{tropp2012user}}\\
  &=\EE\ \tr \left(e^{\mu\sum_{i=1}^k (V_i-\EE V_i)/k}\right)\\
  &\leq D \exp\left( \sum_{i=1}^k \left\|\log \EE e^{\mu(V_i-\EE V_i)/k}\right\|\right) \quad \mbox{by (3.7.1) of \cite{tropp2012user}}\\
  &\leq D\exp\left(\sum_{i=1}^k  \left\|\frac{\mu^2/2}{k^2(1-R\mu/(3k))}\EE (V_i-\EE V_i)^2\right\|\right) \quad \mbox{by Lemma 6.5.1 of \cite{tropp2012user}} \\
  &= D\exp\left( \sum_{i=1}^k \frac{\mu^2/2}{k^2(1-R\mu/(3k))}\left\|\EE (V_i-\EE V_i)^2\right\|\right)\\
  &\leq D\exp\left( \frac{k\sigma_W^2\mu^2/2}{k^2(1-R\mu/(3k))}\right)\\
  &=D\exp\left( \frac{\sigma_W^2\mu^2/2}{k(1-R\mu/(3k))}\right),
\end{align*}
 where $\tr(\cdot)$ denotes the trace of its argument. 
Combining this bound with (\ref{eq.PU}) gives rise to
\begin{align}
  \PP(\|U_n-\EE U_n\|\geq t)\leq De^{-\mu t}\exp\left( \frac{\sigma_W^2\mu^2/2}{k(1-R\mu/(3k))}\right).
\end{align}
Setting $\mu=\frac{tk}{\sigma_W^2+Rt/3}$ yields
\begin{align*}
  \PP(\|U_n-\EE U_n\|\geq t)&\leq De^{-\frac{kt^2/2}{\sigma_W^2+Rt/3}}\\
  &= D\exp\left(-\left\lfloor \frac{n}{m}\right\rfloor\frac{t^2/2}{\sigma_W^2+Rt/3}\right) \\
  &\leq D\exp\left(-\frac{nt^2/4}{m(\sigma_W^2+Rt/3)}\right).
\end{align*}
\end{proof}
Theorem \ref{thm.HDisVf} is a direct result of Lemma \ref{lem.BernsteinU} by taking $W=\tH(\alpha)$  with $m=2, R=4B^2$ and $ \sigma_W^2=16B^4$.

\subsection{Proof of Theorem \ref{thm.projDisVf}}
\label{sec.mainproof.HDisVf}

To prove Theorem \ref{thm.projDisVf}, we need the following two lemmas on matrix perturbation theory.
\begin{lem}[Davis-Kahan \cite{davis1970rotation,stewart1990matrix}]
	\label{lem.DK}
	Let $W, \widehat W \in \mathbb{C}^{D \times D}$ be two Hermitian matrices with eigenvalues $\lambda_j , \widehat \lambda_j, j=1,\ldots,D$ in non-increasing order. 
	Let the columns of $\Phi$ and $\hPhi$ consist of the eigenvectors associated with $\{\lambda_j\}_{j=D-d+1}^D$ of $W$ and $\{\widehat{\lambda}_j\}_{j=D-d+1}^D$ of $\widehat{W}$, respectively. Suppose $\widehat{\lambda}_{D-d}-\lambda_{D-d+1}>0$.
	Then 
	\begin{equation}
		\|\proj_{\widehat \Phi} - \proj_\Phi \| \le \frac{\|\widehat W - W\|}{\widehat{\lambda}_{D-d}-\lambda_{D-d+1}}.
	\end{equation}

\end{lem}

%

\begin{lem}[Weyl \cite{weyl1912asymptotische}]
	\label{lem.Weyl}
	Let $W, \widehat W \in \mathbb{C}^{D \times D}$ be two Hermitian matrices with eigenvalues $\lambda_j , \widehat \lambda_j, j=1,\ldots,D$ in non-increasing order.
	Then
	\begin{equation}
		|\lambda_j-\widehat \lambda_j|\leq \|\widehat W - W\|, \forall j=1,...,D.
	\end{equation}
\end{lem}

\begin{proof}[Proof of Theorem \ref{thm.projDisVf}]
	

By Proposition \ref{propphi}, when $ \alpha<\alpha_{\rm thresh}$, the eigenvectors associated with the  smallest  $d$ eigenvalues of $H(\alpha)$ span the central subspace. Recall that the column space of $\tPhi$ is the eigenspace associated with the smallest $d$  eigenvalues of $\tH(\alpha)$. To derive the relation between $\proj_{\tPhi}$ and $\proj_{\Phi}$ using Lemma \ref{lem.DK} and Lemma \ref{lem.Weyl}, we need an eigengap of $H(\alpha)$.
Denote the eigenvalues of $H(\alpha)$ and $G(\alpha)$ by $\{\lambda_j^{(H)}\}_{j=1}^D$ and $\{\lambda_j^{(G)}\}_{j=1}^D$ in non-increasing order, respectively. We have 
$\lambda_j^{(H)}=p_{\alpha}\lambda_j^{(G)}.$
When $\alpha<\alpha_{\rm thresh}$, with Assumption \ref{assum-gcr} and \ref{assumrho}(iii), Proposition \ref{propphi} implies $\lambda_{D-d}^{(G)}-\lambda_{D-d+1}^{(G)} \ge c_0 $, and then 
$$\lambda_{D-d}^{(H)}-\lambda_{D-d+1}^{(H)} \ge p_{\alpha}c_0.$$
Combining Lemma \ref{lem.DK} and Lemma \ref{lem.Weyl} and setting $W=H(\alpha),\widehat W=\tH(\alpha)$ give rise to
\begin{align*}
	\left\|\proj_{\tPhi} -\proj_{\Phi}\right\| &\le  \frac{\|\tH(\alpha) - H(\alpha)\|}{\lambda_{D-d}^{(H)}-\lambda_{D-d+1}^{(H)}- \|\tH(\alpha) - H(\alpha)\|}\\
	&\le  \frac{\|\tH(\alpha) - H(\alpha)\|}{p_{\alpha}c_0 - \|\tH(\alpha) - H(\alpha)\|}.
\end{align*}
Notice that $\|\proj_{\tPhi} -\proj_{\Phi}\| \le 2$. For any $0 < t < 2$, a sufficient condition to guarantee $\|\proj_{\tPhi} -\proj_{\Phi}\| \le t$ is
$$\|\tH(\alpha) - H(\alpha)\| \le p_{\alpha}c_0t/(1+t)\ \text{ and }\ p_{\alpha}c_0- \|\tH(\alpha) - H(\alpha)\|>0.$$ 
Since $0 < t < 2$, we can express this sufficient condition as
$\|\tH(\alpha) - H(\alpha)\| \le p_{\alpha}c_0t/3.$
Therefore, for any $t \in (0,2)$,
\begin{align}
	\PP\left\{
	\|\proj_{\tPhi} -\proj_{\Phi}\| \ge t
	\right\}\le
	\PP\left\{
	\|\tH(\alpha) - H(\alpha)\|
	\ge
	p_{\alpha}c_0t/3
	\right\}.
	\label{eq.KVfProb}
\end{align}
Combining (\ref{eq.HDisVf}) and (\ref{eq.KVfProb}) gives rise to (\ref{eq.projDisVf}).


\end{proof}

%


\subsection{Proof of Corollary \ref{coroproj}}
\label{sec.mainproof.G3}
\begin{lem}
\label{lem.sqExpectation}
  If $X \ge 0$ is a random variable such that
  $$\PP(X\geq t)\leq Ae^{-\frac{at^2}{b+ct}}, \ \forall t \ge 0$$
  for some $a,b,c>0$, then
  \begin{align*}
    \EE(X^2)
    \leq\frac{2b}{a}(\log(A)+ A+1)+\frac{8Ac^2}{a^2}.
  \end{align*}

\end{lem}
Lemma \ref{lem.sqExpectation} is proved in Supplementary materials \ref{appen.lem.sqExpectation}. Corollary \ref{coroproj} is proved based on Lemma \ref{lem.sqExpectation}.

\begin{proof}[\bf Proof of Corollary \ref{coroproj}]
  Combining Theorem  \ref{thm.projDisVf} and Lemma \ref{lem.sqExpectation} gives rise to
  \begin{align*}
    \EE\|\proj_{\tPhi} -\proj_{\Phi}\|^2&=\int_0^{+\infty} 2t \PP(\|\proj_{\tPhi} -\proj_{\Phi}\|\geq t)dt\\
    &\leq \int_0^{+\infty}   2t\min\left\{1,D\exp\left(-\frac{p_{\alpha}^2c_0^2nt^2}{32B^2(36B^2+p_{\alpha}c_0t)}\right)\right\}dt \\
     &\leq
     \frac{C_2}{C_1}(\log D+D+1)n^{-1} + \frac{8DC_3^2}{C_1^2}n^{-2} \quad {\mbox{ by Lemma \ref{lem.sqExpectation}}}\\
     &\leq C_4/n,
  \end{align*}
  where $C_1,C_2,C_3$ and $C_4$ are given in (\ref{eq.C1-4}).

\end{proof}

\subsection{Proof of Theorem \ref{thm.fError}}
\label{sec.mainproof.f}
In the regression model (\ref{eqmodelg}), if we condition on the recovered central subspace $\tPhi$, the samples $\{\bz_i\}_{i=n+1}^{2n}$ are independently sampled from the marginal distribution of $\rho$ on the subspace spanned by $\tPhi$. Denote this marginal distribution by $\mu$. For any set $\Omega \subset \RR^d,$ $\mu(\Omega)=\rho (\{\bx\in \mathds{R}^D|\tPhi^\top\bx\in\Omega\})$. According to (\ref{eq.squarf}), the regression error of $f$ crucially depends on the estimation error of $g$, which is established in the following lemma.

\begin{lem}
\label{thm.polyConverge}
Let $\{\bz_i\}_{i=n+1}^{2n}$ be i.i.d. samples from a probability measure $\mu$ such that $\mathrm{supp}(\mu)\subseteq[-B,B]^d$ and $\{y_i\}_{i=n+1}^{2n}$ follows the model in (\ref{eqmodelg}). Let $\tg$ be the estimator in (\ref{eq.gh0}) with polynomial order $k=\lceil s\rceil-1$. Under Assumption \ref{assumrho}(i), \ref{assumg}, \ref{assumxi}, we have
  \begin{align}
    &\EE \int|\tg(\bz)-g(\bz)|^2\mu(d\bz)\nonumber\\
    &\leq 2c\cdot \max(\sigma^2  +\bnoise,2M^2+2\bnoise) \frac{ \binom{d+k}{d}K^d\log n}{n} +\frac{4C_g^2B^{2s}d^{s+k}}{(k!)^2K^{2s}} + 6\bnoise.
  \end{align}
  where $c$ is a universal constant and $\bnoise$ is defined in (\ref{eq.bnoise}).
\end{lem}
Lemma \ref{thm.polyConverge} is proved in Supplementary materials \ref{appen.lem.polyConverge} with classical results in nonparametric regression \cite{gyorfi2006distribution,pollard1984convergence,geer2000empirical,cox1988approximation}. 
The proof of Theorem \ref{thm.fError} is given below.
\begin{proof}[\bf Proof of Theorem \ref{thm.fError}]
From \eqref{phiproj}, (\ref{eq.squarf}) and the fact that $\|\bx\|\leq B$, we have 
\begin{eqnarray}
  |\tf(\bx)-f(\bx) |^2 &\leq& 4C_g^2 B^2\|\proj_{\tPhi} -\proj_{\Phi}\|^2 +2|\tg(\tPhi^\top\bx)-g(\tPhi^\top \bx)|^2.
   \nonumber
\end{eqnarray}
We use $\EE_{\{(\bx_i,y_i)\}_{i=1}^{2n}}(\cdot )$ to denote the expectation with respect to the joint distribution of $\{(\bx_i,y_i)\}_{i=1}^{2n}$. 
The regression error of $f$ can be expressed as
  \begin{align}
  &\underset{\{(\bx_i,y_i)\}_{i=1}^{2n}}{\EE}\left(\|\tf(\bx)-f(\bx) \|_{L^2(\rho)}^2\right) \nonumber
  \\=& \underset{\{(\bx_i,y_i)\}_{i=1}^{2n}}{\EE} \left(\int |\tf(\bx)-f(\bx) |^2 \rho(d\bx)\right) \nonumber\\
  \leq& 4C_g^2 B^2\underset{\{(\bx_i,y_i)\}_{i=1}^{n}}{\EE}(\|\proj_{\tPhi} -\proj_{\Phi}\|^2) \nonumber\\
  &\quad +2\underset{\{(\bx_i,y_i)\}_{i=1}^{n}}{\EE}\left(\underset{\{(\bx_i,y_i)\}_{i=n+1}^{2n}}{\EE} \left(\int|\tg(\tPhi^\top\bx)-g(\tPhi^\top \bx)|^2\rho (d\bx)\right)\Bigg| \{(\bx_i,y_i)\}_{i=1}^n\right).
    \label{eq.squarf.1}
\end{align}
Corollary \ref{coroproj} gives an estimate of the first term:
\begin{align}
  \underset{\{(\bx_i,y_i)\}_{i=1}^{n}}{\EE}(\|\proj_{\tPhi} -\proj_{\Phi}\|^2)&\leq \frac{C_4}{n}.
  \label{eq.regTerm.1}
\end{align}
The second term in (\ref{eq.squarf.1}) can be estimated through Lemma \ref{thm.polyConverge} as follows:
\begin{align}
  &\underset{\{(\bx_i,y_i)\}_{i=1}^{n}}{\EE}\left(\underset{\{(\bx_i,y_i)\}_{i=n+1}^{2n}}{\EE}\left(\int|\tg(\tPhi^\top\bx)-g(\tPhi^\top \bx)|^2\rho (d\bx)\right)\Bigg| \{(\bx_i,y_i)\}_{i=1}^n\right)\nonumber\\
 = &\underset{\{(\bx_i,y_i)\}_{i=1}^{n}}{\EE}\left(\underset{\{(\bz_i,y_i)\}_{i=n+1}^{2n}}{\EE}\left(\int|\tg(\bz)-g(\bz)|^2\mu (d\bz)\right)\Bigg| \{(\bx_i,y_i)\}_{i=1}^n\right)\nonumber\\
  \le& \underset{\{(\bx_i,y_i)\}_{i=1}^{n}}{\EE} \Bigg[2c\cdot \max(\sigma^2  +\bnoise,2M^2+2\bnoise) \frac{ \binom{d+k}{d}K^d\log n}{n} \nonumber\\
  &\hspace{1cm}+\frac{4C_g^2B^{2s}d^{s+k}}{(k!)^2K^{2s}}+ 6\bnoise \Bigg| \{(\bx_i,y_i)\}_{i=1}^n\Bigg]\nonumber\\
  \le& 2c\cdot \left(\sigma^2  +2M^2+3\underset{\{(\bx_i,y_i)\}_{i=1}^{n}}{\EE}\bnoise\right) \frac{ \binom{d+k}{d}K^d\log n}{n}  \nonumber
  \\ &\hspace{1cm}+\frac{4C_g^2B^{2s}d^{s+k}}{(k!)^2K^{2s}}+ 6\underset{\{(\bx_i,y_i)\}_{i=1}^{n}}{\EE} \bnoise  \nonumber\\
  \le& C\left( \frac{(\sigma^2  +2M^2+ 6C_5n^{-1})\log n}{n}\right)^{\frac{2s}{2s+d}}+\frac{12C_5}{n},
  \label{eq.regTerm.2.1}
\end{align}
where $C_5$ is defined in Theorem \ref{thm.fError} and $C=2c\binom{d+k}{d}+\frac{4C_g^2B^{2s}d^{s+k}}{(k!)^2}$ independent of $n$. The last inequality in (\ref{eq.regTerm.2.1}) results from 
$$\underset{\{(\bx_i,y_i)\}_{i=1}^{n}}{\EE} \bnoise\leq 2C_4C^2_gB^2n^{-1}$$ 
by (\ref{phiproj}), Corollary \ref{coroproj} and $K$ chosen according to (\ref{eq.K}). 
Finally, combining (\ref{eq.regTerm.1}) and (\ref{eq.regTerm.2.1}) gives rise to \eqref{thmfeq}.

\commentout{
\begin{align*}
  &\underset{\{(\bx_i,y_i)\}_{i=1}^{2n}}{\EE}\|\hf(\bx)-f(\bx) \|_{L^2(\rho)}^2\leq 16C_8n^{-1} + 2C\left(\frac{\max(\sigma^2  +2C_8n^{-1},2M^2+ 4C_8n^{-1})\log n}{n}\right)^{\frac{2s}{2s+d}}.
\end{align*}}
\end{proof}

\section{Conclusion}
\label{sec.conclusion}
In this paper, we combine GCR and piecewise polynomial approximations to estimate a high-dimensional function which varies along a low-dimensional central subspace.
We prove that, the mean squared estimation error for the central subspace is $O(n^{-1})$  and the mean squared regression error is $O\left(\left(n/\log n\right)^{-\frac{2s}{2s+d}}\right)$. A modified GCR algorithm with improved efficiency is also proposed. Numerical experiments verify our theories and demonstrate that the modified GCR has the same accuracy as that in our theoretical results.


\section*{Acknowledgement}
Wenjing Liao's research is supported by NSF DMS 1818751 and DMS 2012652. Both authors are grateful to Alessandro Lanteri, Mauro Maggioni and Stefano Vigogna for pointing out the noise bias in the regression model \eqref{eqmodelg}, which was not taken care of in the first version of this manuscript. This problem has been fixed in the current version.

\bibliographystyle{abbrv}
\bibliography{reference}

\newpage
\begin{center}
\textbf{\large Supplemental Materials for Learning functions varying along a central subspace}
\end{center}



\numberwithin{equation}{section}
\renewcommand{\theequation}{\thesection.\arabic{equation}}

\appendix

\section{Proof of Proposition \ref{propphi}}
\label{sec.proof.gcr}
\begin{proof}[\bf Proof of Proposition \ref{propphi}]

By expressing $\tbx -\bx = \proj_{\calS_\Phi}(\tbx -\bx) +  \proj_{\calS_\Phi^\perp}(\tbx -\bx)$, we can write the covariance matrix as
\begin{align*}
&\EE\left[(\tbx -\bx)(\tbx -\bx)^\top\right]
\\
 =&\EE \left[\proj_{\calS_\Phi}(\tbx -\bx) \left[ \proj_{\calS_\Phi}(\tbx -\bx) \right]^\top\right]
+ \EE \left[\proj_{\calS_\Phi^\perp}(\tbx -\bx) \left[ \proj_{\calS_\Phi^\perp}(\tbx -\bx) \right]^\top\right]\\
&\hspace{0.5cm} +\EE \left[\proj_{\calS_\Phi}(\tbx -\bx) \left[ \proj_{\calS_\Phi^\perp}(\tbx -\bx) \right]^\top\right]
+ \EE \left[\proj_{\calS_\Phi^\perp}(\tbx -\bx) \left[ \proj_{\calS_\Phi}(\tbx -\bx) \right]^\top\right]\\
 =& \Phi \EE \left[\Phi^\top (\tbx-\bx) (\tbx-\bx)^\top \Phi\right] \Phi^\top +  \Psi\EE\left[ \Psi^\top (\tbx-\bx) (\tbx-\bx)^\top \Psi\right] \Psi^\top\\
&\hspace{0.5cm} +\Phi \EE \left[\Phi^\top (\tbx-\bx) (\tbx-\bx)^\top \Psi\right] \Psi^\top + \Psi \EE \left[\Psi^\top (\tbx-\bx) (\tbx-\bx)^\top \Phi\right] \Phi^\top
\end{align*}
where the columns of $\Psi$ form an orthonormal basis of $\calS_\Phi^\perp$.

We next prove $  \EE  \left[ \Psi^\top (\tbx-\bx) (\tbx-\bx)^\top \Phi | V_f(\tbx,\bx) \le \alpha \right]$ is zero. Note that $\bx=\proj_{\calS_\Phi}\bx+\proj_{\calS_\Phi^{\perp}}\bx$. Define the reflection operator
$$R_{\Phi}(\bx)=\proj_{\calS_\Phi}\bx-\proj_{\calS_\Phi^{\perp}}\bx$$
which reflects $\bx$ with respect to the central subspace $\Phi$. 
The operator $R_{\Phi}$ is invertible and $R_{\Phi}^{-1}=R_{\Phi}$.  We also have 
$$\Psi^\top (R_{\Phi}(\tbx)-R_{\Phi}(\bx)) (R_{\Phi}(\tbx)-R_{\Phi}(\bx))^\top \Phi=-\Psi^\top (\tbx-\bx) (\tbx-\bx)^\top \Phi.$$
For any set $\Omega\subset \supp(\rho)$, we define the reflected set $R_{\Phi}(\Omega)=\{R_{\Phi}(\bx)|\bx\in \Omega\}$.
Since $\bx$ has a spherical distribution, we have, for any set $\Omega\subset \supp(\rho)$,
\begin{align}
  \PP(\bx\in \Omega)=\PP(\bx\in R_{\Phi}(\Omega))=\PP(R_{\Phi}^{-1}(\bx)\in \Omega)=\PP(R_{\Phi}(\bx)\in \Omega).
\end{align}
 Therefore $R_{\Phi}(\bx)$ and $\bx$ have the same distribution, which implies that   $\Psi^\top (R_{\Phi}(\tbx)-R_{\Phi}(\bx)) (R_{\Phi}(\tbx)-R_{\Phi}(\bx))^\top \Phi$ and $\Psi^\top (\tbx-\bx) (\tbx-\bx)^\top \Phi$ have the same distribution. 
 
 We next show $V_f(\tbx,\bx)=V_f(R_{\Phi}(\tbx),R_{\Phi}(\bx))$. It is straightforward that $f(\bx)=g(\Phi^\top\bx)=g(\Phi^\top R_{\Phi}(\bx))=f(R_{\Phi}(\bx))$.
 Since $R_{\Phi}(\bx)$ and $\bx$ have the same distribution, $(\bx,f(\bx))$ and $(R_{\Phi}(\bx),f(\bx))$ have the same distribution. Thus
 \begin{align*}
   V_f(\bx_1,\bx_2)&=\var\{f(\bx)|\bx=t\bx_1+(1-t)\bx_2 \mbox{ for } 0\leq t\leq 1\}\\
   &=\var\{f(\bx)|R_{\Phi}(\bx)=t\bx_1+(1-t)\bx_2 \mbox{ for } 0\leq t\leq 1\}\\
   &=\var\{f(\bx)|\bx=R_{\Phi}^{-1}(t\bx_1+(1-t)\bx_2) \mbox{ for } 0\leq t\leq 1\}\\
   &=\var\{f(\bx)|\bx=tR_{\Phi}(\bx_1)+(1-t)R_{\Phi}(\bx_2) \mbox{ for } 0\leq t\leq 1\}\\
   &=V_f(R_{\Phi}(\bx_1),R_{\Phi}(\bx_2)).
 \end{align*}
Overall, the distributions of $(\bx,f(\bx),\tbx,f(\tbx))$ and $(R_{\Phi}(\bx),f(\bx),R_{\Phi}(\tbx),f(R_{\Phi}(\tbx)))$ are the same.
Therefore, we have
\begin{align*}
&\EE\left[\Psi^\top (\tbx-\bx) (\tbx-\bx)^\top \Phi|V_f(\tbx,\bx)\leq \alpha\right]\\
=&\EE\left[ \Psi^\top (R_{\Phi}(\tbx)-R_{\Phi}(\bx)) (R_{\Phi}(\tbx)-R_{\Phi}(\bx))^\top \Phi|V_f(R_{\Phi}(\tbx),R_{\Phi}(\bx))\leq \alpha\right]\\
=&-\EE\left[\Psi^\top (\tbx-\bx) (\tbx-\bx)^\top \Phi|V_f(\tbx,\bx)\leq \alpha\right],
\end{align*}
which implies that 
$$\EE\left[\Psi^\top (\tbx-\bx) (\tbx-\bx)^\top \Phi|V_f(\tbx,\bx)\leq \alpha\right]=\mathbf{0}.$$
Similarly, we can show that
$$
  \EE [\Phi^\top (\tbx-\bx) (\tbx-\bx)^\top \Psi | V_f(\tbx,\bx) \le \alpha]={\bf 0}.
$$

Let
\begin{align*}
G_1(\alpha) &= \EE\left[ \Phi^\top (\tbx-\bx) (\tbx-\bx)^\top \Phi | V_f(\tbx,\bx) \le \alpha \right] \in \RR^{d \times d}
\\
G_2(\alpha) &= \EE\left[ \Psi^\top (\tbx-\bx) (\tbx-\bx)^\top \Psi | V_f(\tbx,\bx) \le \alpha \right] \in \RR^{(D-d) \times (D-d)}.
\end{align*}
The covariance matrix $G(\alpha)$ can be written as
\begin{align*}
G(\alpha)= \Phi G_1(\alpha) \Phi^\top + \Psi G_2(\alpha) \Psi^\top.
\end{align*}
Let $\hat\lambda_1,\ldots,\hat\lambda_d$ be the eigenvalues of $G_1(\alpha)$ and $\bz_1,\ldots,\bz_d$ be the associated orthonormal eigenvectors, and let $\tilde\lambda_1,\ldots,\tilde\lambda_{D-d}$ be the eigenvalues of $G_2(\alpha)$ and $\tilde\bz_1,\ldots,\tilde\bz_{D-d}$ be the associated orthonormal eigenvectors. Then the eigenvalues of $G(\alpha)$ are $\hat\lambda_1,\ldots,\hat\lambda_d,\tilde\lambda_1,\ldots,\tilde\lambda_{D-d}$ and the associated eigenvectors are  $\Phi \bz_1,\ldots, \Phi \bz_d,\Psi\tilde\bz_1,\ldots,\Psi\tilde\bz_{D-d}$ since
\begin{align*}
G(\alpha) \Phi \bz_j &= \Phi G_1(\alpha) \Phi^\top  \Phi \bz_j = \hat\lambda_j\Phi \bz_j, j=1,\ldots,d,
\\
G(\alpha) \Psi \tilde\bz_k &= \Psi G_1(\alpha) \Psi^\top  \Psi \tilde \bz_k = \tilde\lambda_k \Psi\tilde\bz_k, k=1,\ldots,D-d.
\end{align*}
The eigenvalues satisfy
\begin{align*}
\hat\lambda_j &= (\Phi \bz_j)^\top  G_1(\alpha) \Phi \bz_j \\
&= \bz_j^\top \Phi^\top \Phi \EE\left[ \Phi^\top (\tbx-\bx) (\tbx-\bx)^\top \Phi | V_f(\tbx,\bx) \le \alpha \right] \Phi^\top \Phi \bz_j
\\
& =  \EE\left[\bz_j^\top \Phi^\top (\tbx-\bx) (\tbx-\bx)^\top \Phi \bz_j | V_f(\tbx,\bx) \le \alpha \right]
\\
& = {\var\left[(\Phi \bz_j)^\top (\tbx-\bx) | V_f(\tbx,\bx) \le \alpha \right]}, \ j=1,\ldots,d;
\\
\tilde\lambda_k
& = \var\left[(\Psi \tilde\bz_k)^\top (\tbx-\bx) | V_f(\tbx,\bx) \le \alpha \right], \ k=1,\ldots,D-d.
\end{align*}
By Assumption \ref{assum-gcr}, 
$$\left(\min_k\tilde\lambda_k\right)-\left(\max_j\lambda_j\right)\geq c_0.$$
\end{proof}

\section{Proof of Example \ref{lem.sphere}}\label{appen.lem.sphere}
\begin{proof}[Proof of Example \ref{lem.sphere}]
	Denote $\Omega=\supp(\rho)$ and $\calB_{D,r}(\bx)$ the $D$-dimensional ball with radius $r$ centered at $\bx$. Since $\rho$ has a spherical distribution, we have $\Omega=\calB_{D,B}(\mathbf{0})$. Denote $\Omega_D$ as the largest $D$-dimensional hypercube inside $\Omega$. There are infinitely many such hypercubes all of which have side length $2B/\sqrt{D}$. Any such hypercube can be written as
	$$\Omega_D=\{\bx:\bx=a_1\bu_1+\cdots +a_D\bu_D \mbox{ with } \max_{1\leq j\leq D} |a_j|\leq B/\sqrt{D}\}$$
	for a set of orthornormal basis $\{\bu_1,\bu_2,...,\bu_D\}$ in $\RR^D$.
	In this proof, we choose the one satisfying $\bu_j=\bphi_j$ for $1\leq j\leq d$, where $\bphi_j$ denotes the $j$-th column of $\Phi$. We then define
	$$
	\Omega_d:=\{\ba=[a_1,...,a_d]^{\top}\in \RR^d: a_1\bu_1+\cdots +a_d\bu_d\in \Omega_D\cap \calS_{\Phi}\}=\{\ba\in \RR^d: \|\ba\|_{\infty}\leq B/\sqrt{D}\}.
	$$
	Here $\Omega_d$ is a $d$-dimensional hypercube with side length $2B/\sqrt{D}$. For any $\bx\in\Omega_D$, we have $\Phi^\top\bx\in\Omega_d$.
	
	
	Then $p_{\alpha}$ can be lower bounded by
	\begin{align}
		p_{\alpha}&=\PP(V_f(\tbx,\bx)\leq \alpha) \nonumber\\
		&= \int_{\Omega} \int_{\Omega} \mone\left\{V_f(\tbx,\bx)\leq \alpha\right\} d\rho(\tbx) d\rho(\bx) \nonumber \\
		&\geq \int_{\Omega_{D}} \int_{\Omega} \mone\left\{V_f(\tbx,\bx)\leq \alpha\right\} d\rho(\tbx) d\rho(\bx).
		\label{eq.PSP.p}
	\end{align}
	In the rest of the proof, we derive a lower bound for the right hand side of (\ref{eq.PSP.p}).


	According to Lemma \ref{lemphix}, for any $\bx_1,\bx_2\in\Omega$ with $\|\Phi^\top(\bx_1-\bx_2)\|\leq \sqrt{\alpha}/C_g$, we have $V_f(\bx_1,\bx_2)\leq \alpha$, where $C_g$ is a Lipschitz constant of $g$. 
	
	Define
	\begin{align*}
		&\widetilde{\Omega}_d=\{\ba\in \RR^d: \|\ba\|_{\infty}\leq B/\sqrt{D}-\sqrt{\alpha}/C_g\},\\
		&\widetilde{\Omega}_D=\{\bx\in \Omega_D: \Phi^\top\bx\in\widetilde{\Omega}_d\}.
	\end{align*}
	Here $\widetilde{\Omega}_d\subset \Omega_d$ is a $d$-dimensional hypercube with side length $2B/\sqrt{D}-2\sqrt{\alpha}/C_g>0$ such that, for any $\ba\in \widetilde{\Omega}_d$, we have $\calB_{d,\frac{\sqrt{\alpha}}{C_g}}(\ba)\subset \Omega_d$. Such a choice of $\widetilde{\Omega}_d$ guarantees that for any $\bx\in\widetilde{\Omega}_D$ with $\Phi^\top\bx\in \widetilde{\Omega}_d$, $\calB_{d,\frac{\sqrt{\alpha}}{C_g}}(\Phi^\top\bx)$ is entirely inside $\Omega_d$ (see Figure \ref{fig.spherical} for a demonstration). Denote the volume of a subset $S$ in $\RR^d$ by $\Vol_d(S)$. We have
	\begin{align*}
		\PP(V_f(\tbx,\bx)\leq \alpha|\bx\in \widetilde{\Omega}_D)&\geq \left(\frac{2B}{\sqrt{D}}\right)^{D-d} \Vol_d\left(\calB_{d,\frac{\sqrt{\alpha}}{C_g}}(\Phi^\top\bx)\right)h_{\min}\\
		&=C_1\alpha^{d/2}h_{\min}
	\end{align*}
	\begin{minipage}{0.55\textwidth}
		with $C_1=\frac{1}{C_g^d}\left(\frac{2B}{\sqrt{D}}\right)^{D-d}\frac{\pi^{d/2}}{\Gamma(d/2+1)}$ and $\Gamma$ being the gamma function. Therefore
		\begin{align*}
			p_{\alpha}&\geq \int_{\Omega_{D}} \int_{\Omega} \mone\left\{V_f(\tbx,\bx)\leq \alpha\right\} d\rho(\tbx) d\rho(\bx)\\
			&\geq \int_{\widetilde{\Omega}_{D}}  \PP(V_f(\tbx,\bx)\leq \alpha|\bx\in \widetilde{\Omega}_D) d\rho(\bx)\\
			&\geq \int_{\widetilde{\Omega}_{D}}  C_1\alpha^{d/2}h_{\min} d\rho(\bx)\\
			&=C_2\alpha^{d/2}h_{\min}^2
		\end{align*}
		with 
		$C_2=\frac{1}{C_g^d}\left(\frac{2B}{\sqrt{D}}\right)^{2(D-d)}\left(\frac{2B}{\sqrt{D}}- \frac{2\sqrt{\alpha}}{C_g}\right)^d\frac{\pi^{d/2}}{\Gamma(d/2+1)}.$
	\end{minipage}
	\begin{minipage}{0.4\textwidth}
		\begin{figure}[H]
			\centering		\includegraphics[width=0.6\textwidth]{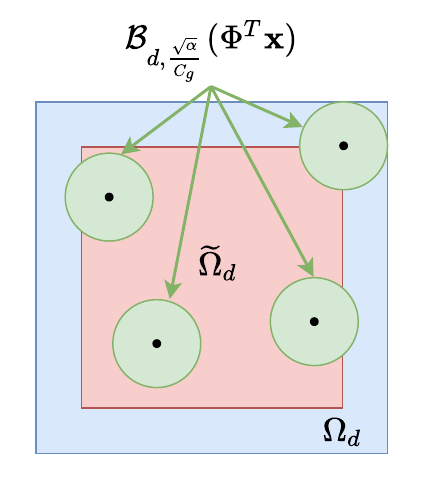}
			\caption{\label{fig.spherical} For any $\bx\in\Omega_D$ with $\Phi^\top\bx\in \widetilde{\Omega}_d$, $\calB_{d,\frac{\sqrt{\alpha}}{C_g}}(\Phi^\top\bx)$ entirely locates inside $\Omega_d$}
		\end{figure}
	\end{minipage}
\end{proof}

\section{Proof of Example \ref{lemmavfgamma}} \label{appen.lem.lemmavfgamma}
\begin{proof}[\bf Proof of Example \ref{lemmavfgamma}]
	Let $\bar{f}_\ell$ be the mean of $f(\bx)$ on $\ell(\bx_i,\bx_j)$ and $\bar{f}_{ij}$ be the mean of $f(\bx)$ on $T_{ij}(r)$.
	We first estimate the difference between $\bar{f}_\ell$ and $\bar{f}_{ij}$. For simplicity, we denote $\ell(\bx_i,\bx_j)$ by $\ell$. We parameterize $\ell$ by $t$ for $0\leq t\leq 1$ and use $\br(t)$ to denote a point on $\ell$ such that $\br(0) =\bx_i$ and $\br(1)=\bx_j$. Let $\mathbb{S}(t)$ be the disk centered at $\br(t)$ with radius $r$ on the hyperplane of dimension $D-1$ which is perpendicular to $\ell$. 
	Let $\rho^\top$ be a measure on $[0,1]$ such that, for any $\br(t) \in \ell,$ $ \rho^\top (dt) = \lim_{r\rightarrow 0} \frac{\rho(T(\br(t),\br(t+dt),r))}{\rho(T(\br(0),\br(1),r))}$, where $T(\br(t),\br(t+dt),r)$ is the tube enclosing $\ell(\br(t),\br(t+dt))$ with radius $r$.
	Since $\rho$ is uniform and $T_{ij}(r)\subset \supp(\rho)$, we can express
	$$
	\bar{f}_\ell=\int_0^1 f(\br(t)) \rho^\top (dt)= \frac{1}{\rho(T_{ij}(r))}\int_{0}^1 \int_{\mathbb{S}(t)} f(\br(t))\rho(d\bx )dt$$
	and
	$$
	\bar{f}_{ij}=\frac{1}{\rho(T_{ij}(r))}\int_{T_{ij}(r)} f(\bx)\rho(d\bx)
	= \frac{1}{\rho(T_{ij}(r))}\int_{0}^1 \int_{\mathbb{S}(t)} f(\bx)\rho(d\bx )dt.
	$$
	Therefore
	\begin{eqnarray*}
		|\bar{f}_\ell-\bar{f}_{ij}|&=&\frac{1}{\rho(T_{ij}(r))}\left| \int_0^1 \int_{\mathbb{S}(t)}[ f(\br(t))-f(\bx)]\rho(d\bx) dt\right|
		\leq r C_g.
	\end{eqnarray*}
	We next derive the error bound between $V_f(\bx_i ,\bx_j)$ and $V_f(\bx_i ,\bx_j,r)$:
	\begin{align*}
		&|V_f(\bx_i ,\bx_j)-V_f(\bx_i ,\bx_j,r)|\\
		=&\left| \frac{1}{\rho(T_{ij}(r))}\int_0^1 \int_{\mathbb{S}(t)} (f(\br(t))-\bar{f}_\ell)^2\rho(d\bx )dt-  \frac{1}{\rho(T_{ij}(r))}\int_0^1 \int_{\mathbb{S}(t)} (f(\bx))-\bar{f}_{ij})^2\rho(d\bx) dt\right|\\
		\leq&\frac{1}{\rho(T_{ij}(r))}\int_0^1 \int_{\mathbb{S}(t)} \left|  (f(\br(t))-\bar{f}_\ell)^2 - (f(\bx))-\bar{f}_{ij})^2\right| \rho(d\bx )dt\\
		=&\frac{1}{\rho(T_{ij}(r))}\int_0^1 \int_{\mathbb{S}(t)} \left|  (f(\br(t))-\bar{f}_\ell+f(\bx)-\bar{f}_{ij}) (f(\br(t))-f(\bx)+\bar{f}_{ij}-\bar{f}_\ell) \right| \rho(d\bx )dt\\
		\le&\frac{1}{\rho(T_{ij}(r))}\int_0^1 \int_{\mathbb{S}(t)} \left|  (f(\br(t))-\bar{f}_\ell+f(\bx)-\bar{f}_{ij}) \right| \cdot 2rC_g  \rho(d\bx )dt.
	\end{align*}
	On the other hand, for any $\br(t) \in \ell$,
	\begin{align*}
		|f(\br(t))-\bar{f}_\ell| \le  C_g \|\bx_i -\bx_j\| \le 2C_g B.
	\end{align*}
	Similarly, for any $\bx \in T_{ij}(r)$,
	$$ |f(\bx)-\bar{f}_{ij}|  \le C_g \sup_{\bz, \bw \in T_{ij}(r)} \|\bz-\bw\| = C_g \sqrt{\|\bx_i-\bx_j\|^2 + 4 r^2} \le 3C_g B $$
	when $4r^2 \le 5B^2.$ In summary,
	\begin{align*}
		&|V_f(\bx_i ,\bx_j)-V_f(\bx_i ,\bx_j,r)|\le 10  C_g^2 B r.
	\end{align*}
\end{proof}

\section{Proof of Theorem \ref{prop.alg1}}
\label{appen.prop.alg2}
The proof of Theorem \ref{prop.alg1} relies on several lemmas, which are presented and proved in Section \ref{sec.thm.alg1.lemma}. We prove Theorem \ref{prop.alg1} in Section \ref{sec.thm.alg1.proof}.
\subsection{Some key lemmas}\label{sec.thm.alg1.lemma}
The following lemma gives an estimate on the difference between the population variance of a bounded random variable and its empirical counterpart.
\begin{lem}
	\label{lemconv}
	Let $s$ be a random variable in $[m-A,m+A]$ for some $m \in \RR$ and $A>0$. Suppose $\{s_i\}_{i=1}^{n}$ are independent copies of $s$. Denote the empirical variance by $\hV(s) = \frac{1}{n-1} \sum_{i=1}^n (s_i -\bar s )^2$
	where $\bar s = (s_1+\ldots+s_n)/n.$
	Then
	\begin{eqnarray*}
		\PP\left( |\hV(s)-\EE \hV(s)|\geq t\right)\leq 2e^{-\frac{nt^2}{16A^4}}\ \text{ for any } t>0.
	\end{eqnarray*}
\end{lem}
Lemma \ref{lemconv} is a consequence of U-statistics \cite{hoeffding1994probability} applied on $\hV(s)$.

Based on Lemma \ref{lemconv}, we prove a high probability bound between $\Vyijgamma$ and $\hVyijgamma$:
\begin{lem}\label{lem.III}
	Let $\{\bx_i\}_{i=1}^n$ be i.i.d. samples from the probability measure $\rho$ , and $\{y_i\}_{i=1}^n$ be sampled according to the model in \eqref{eqyi}, under Assumption \ref{assumrho}(i), \ref{assumg}, \ref{assumxi} and \ref{assumtube}. Let $\nu >0$ and set $\alpha_0$ and $r$ according to (\ref{eqalpha00}).
	Then $$\PP \left\{|\Vyijgamma- \hVyijgamma|\leq\frac{\alpha_0}{2}+2\sigma^2 \right\} \ge 1-4n^{-\nu}.$$
\end{lem}
\begin{proof}[\bf Proof of Lemma \ref{lem.III}]
	We set $\eta =  \alpha_0/(2C_g^2) $, and consider the tube $T_{ij}(r)$ with $\|\bx_i-\bx_j\| > \eta$ and $\|\bx_i-\bx_j\| \le \eta$ as separate cases.
	
	\begin{description}
		
		\item[Case I when $\|\bx_i-\bx_j\| \le \eta$:]
		
		When $\|\bx_i-\bx_j\| \le \eta$, we expect $\Vyijgamma$ to be small since $f(\bx)$ has a small variation within $T_{ij}(r)$. Let $\bz = (\bx_i+\bx_j)/2$. For any $\bx \in T_{ij}(r)$, $|f(\bx)-f(\bz)| \le C_g  \sqrt{(\eta/2)^2+r^2}$.  Hence,
		\begin{align*}
			|\Vyijgamma- \hVyijgamma|
			&\le
			\max(\Vyijgamma,\hVyijgamma)\\
			&\leq 2\left(C_g^2[(\eta/2)^2+r^2]+\sigma^2\right) \le  \alpha_0/2  + 2\sigma^2
		\end{align*}
		where the last inequality holds as long as $ \eta^2 \le \alpha_0 / (2C_g^2) $ and $r^2  \leq \alpha_0 /(8 C_g^2)$. By \eqref{eqalpha00}, $\alpha_0$ is small when $n$ is sufficiently large. These conditions are guaranteed as $ \eta^2 < \eta = \alpha_0 / (2C_g^2) $ and $r^2 <r \le \alpha_0 /(8 C_g^2)$ when $n$ is sufficiently large.

		\item[Case II when $\|\bx_i-\bx_j\| > \eta$:]
		When $\|\bx_i-\bx_j\| > \eta$, there are sufficient points in $T_{ij}(r)$ so that $\hVyijgamma$ is concentrated on $\Vyijgamma$.
		Let $\rho(T_{ij}(r))$ be the measure of the tube $T_{ij}(r)$, which satisfies $\rho(T_{ij}(r)) \ge c_1 r^{D-1}\eta$ by \eqref{eqmeasuretube}. Let $n_{ij}(r)$ be the number of points in $T_{ij}(r)$ and then  $\widehat\rho(T_{ij}(r)) = n_{ij}(r)/n$ is the empirical measure of $T_{ij}(r)$. By \cite[Lemma 29]{liao2016adaptive}, we have the following concentration of measure:
		\begin{align}
			&\PP\left\{|\widehat\rho(T_{ij}(r)) - \rho(T_{ij}(r))| \ge \frac 1 2 \rho(T_{ij}(r)) \right\} \le 2 e^{- \frac 3 {28} n \rho(T_{ij}(r))}
			\le 2 e^{- \frac 3 {28}c_1 n r^{D-1}\eta}
			\label{lemmavpeq1}
		\end{align}
		On the condition of the event $|\widehat\rho(T_{ij}(r)) - \rho(T_{ij}(r))| \le \frac 1 2 \rho(T_{ij}(r))$, we have $\widehat\rho(T_{ij}(r)) \ge \frac 1 2 \rho(T_{ij}(r)) $, which implies $n_{ij}(r) \ge \frac 1 2 c_1 n r^{D-1}\eta$. By Lemma \ref{lemconv}, $\hVyijgamma$ is concentrated on $\Vyijgamma$ with high probability:
		\begin{align*}
			&\PP\left\{|\Vyijgamma - \hVyijgamma|  \ge t \right\}
			\le 2e^{- \frac{n_{ij}(r)t^2}{16(M+\sigma)^4}}  ,\forall t>0.
		\end{align*}
		Setting $t = \frac 1 2 \alpha_0$ gives rise to
		\begin{align}
			&\PP\left\{|\Vyijgamma - \hVyijgamma|  \ge \frac 1 2 \alpha_0 \bigg  | |\widehat\rho(T_{ij}(r)) - \rho(T_{ij}(r))| \le \frac 1 2 \rho(T_{ij}(r)) \right\} \nonumber
			\\
			\le & 2e^{- \frac{n_{ij}(r) \alpha_0^2}{64(M+\sigma)^4 }}
			\le
			e^{- \frac{ c_1 n r^{D-1}\eta \alpha_0^2}{128(M+\sigma)^4 }}
			\label{lemmavpeq2}
		\end{align}
		Combining \eqref{lemmavpeq1} and \eqref{lemmavpeq2} yields
		\begin{align*}
			&\PP\left\{|\Vyijgamma - \hVyijgamma|  \ge \frac 1 2 \alpha_0  \right\}
			\le
			\PP\left\{|\widehat\rho(T_{ij}(r)) - \rho(T_{ij}(r))| \ge \frac 1 2 \rho(T_{ij}(r)) \right\}
			\\
			&+  \PP\left\{|\Vyijgamma - \hVyijgamma|  \ge \frac 1 2 \alpha_0 \bigg  | |\widehat\rho(T_{ij}(r)) - \rho(T_{ij}(r))| \le \frac 1 2 \rho(T_{ij}(r)) \right\}
			\\
			& \le  2 e^{- \frac {3} {28}c_1 n r^{D-1}\eta}
			+ 2e^{- \frac{ c_1 n r^{D-1}\eta \alpha_0^2}{128(M+\sigma)^4 }}
			\leq  2 e^{-\frac{3c_1 C_1^{D-1}n \alpha_0^D}{56C_g^2}}
			+ 2e^{-\frac{c_1 C_1^{D-1}n \alpha_0^{D+2}}{256(M+\sigma)^4 C_g^2}}.
		\end{align*}
		When $\alpha_0$ is chosen according to \eqref{eqalpha00}, one obtains
		$$\PP\left\{|\Vyijgamma - \hVyijgamma|  \ge \frac 1 2 \alpha_0  \right\} \le 4 n^{-\nu}.$$
		
	\end{description}
\end{proof}

A key step in the proof of Theorem \ref{prop.alg1} is to estimate the difference between $\hV_y(\bx_i,\bx_j,r)$ and $V_f(\bx_i,\bx_j)$, which is given in the next lemma.
\begin{lem}
	\label{lem.hV}
	Let $\{\bx_i\}_{i=1}^n$ be i.i.d. samples from the probability measure $\rho$, and $\{y_i\}_{i=1}^n$ be sampled according to the model in \eqref{eqyi}, under Assumption \ref{assumrho}(i), \ref{assumg}-\ref{assumVr}.
	Let $\nu >0$ and set $\alpha_0$ and $r$ as in (\ref{eqalpha00}).
	Then
	\begin{equation}
		\label{eqvhatv}
		\PP\left( |V_f(\bx_i,\bx_j)-\hV_y(\bx_i,\bx_j,r)|\leq \alpha_0+3\sigma^2\right) \geq 1-4n^{-\nu}.
	\end{equation}
\end{lem}
\begin{proof}[\bf Proof of Lemma \ref{lem.hV}]
	We decompose the difference by
\begin{align}
&|\Vfij - \hVyijgamma| \nonumber\\
=&|\Vfij-\Vfijgamma
+ \Vfijgamma- \Vyijgamma
+ \Vyijgamma- \hVyijgamma|\nonumber\\
\leq& \underbrace{|\Vfij-\Vfijgamma|}_{\rm I}
+\underbrace{ |\Vfijgamma- \Vyijgamma| }_{\rm II}
+\underbrace{  |\Vyijgamma- \hVyijgamma|}_{\rm III}.
\label{eq.hVdecom}
\end{align}

The term ${\rm I}$ in (\ref{eq.hVdecom}) represents the difference between the variance of $f$ over the segment between $\bx_i$ and $\bx_j$, and the variance of $f$ in the tube with radius $r$ enclosing this segment. According to Assumption \ref{assumVr}, ${\rm I}
\leq \frac{\alpha_0}{2} \text{ if } r \le \frac{\alpha_0}{2c_2}$.

The term ${\rm II}$ satisfies ${\rm II}\le \sigma^2$ since it captures the variance of the bounded noise in $[-\sigma,\sigma]$.
%


The term ${\rm III}$ in (\ref{eq.hVdecom}) captures the difference between the population variance of $y$ in the tube $T_{ij}(r)$ and its empirical counterpart. A high probability bound is given by Lemma \ref{lem.III}.

Putting the above ingredients together, we have
\begin{align*}
{\rm I}
&\leq \frac{\alpha_0}{2} \text{ if } r \le \frac{\alpha_0}{2c_2} \text{ from Assumption \ref{assumVr}},
\\
{\rm II} & \le  \sigma^2,
\\
{\rm III} & \leq\frac{\alpha_0}{2}+2\sigma^2 \text{ with probability no less than }1-4n^{-\nu},
\end{align*}
which implies \eqref{eqvhatv}.

\end{proof}

The following lemma gives an upper bound of $V_f(\bx_i,\bx_j)$.
\begin{lem}
	\label{lemphix}
	Suppose $f$ is defined as \eqref{eqyi}, with the function $g$ satisfying Assumption \ref{assumg}.
	For any two points $\bx_i, \bx_j \in \RR^D$, we have $V_f(\bx_i,\bx_j) \le C_g^2 \|\Phi^\top\bx_i-\Phi^\top \bx_j\|^2$.
\end{lem}
\begin{proof}[\bf Proof of Lemma \ref{lemphix}]
	For any $\bx = (1-t) \bx_i + t \bx_j, t \in [0,1]$, we have
	\begin{align*}
		|f(\bx) - f(\bx_i)| & = |g(\Phi^\top \bx) - g(\Phi^\top \bx_i)|
		\\
		& \le C_g  \|\Phi^\top [(1-t) \bx_i + t \bx_j] -\Phi^\top \bx_i\|
		\\
		& = C_g  \| t (\Phi^\top \bx_j - \Phi^\top \bx_i)
		\|
		\le C_g \|\Phi^\top\bx_i-\Phi^\top \bx_j\|.
	\end{align*}
	Therefore,  $V_f(\bx_i,\bx_j) = \var\left(f(\bx)| \bx= (1-t) \bx_i + t \bx_j, t \in [0,1]\right) \le C_g^2 \|\Phi^\top\bx_i-\Phi^\top \bx_j\|^2$.
\end{proof}

\subsection{Proof of Theorem \ref{prop.alg1}}\label{sec.thm.alg1.proof}
\begin{proof}[\bf Proof of Theorem \ref{prop.alg1}]
We prove (\ref{eq.hV}) and (\ref{eq.hn}) in sequel.

\noindent\textbf{$\bullet$ Proof of (\ref{eq.hV}).}
Lemma \ref{lem.hV} gives an estimate on $|V_f(\bx_i,\bx_j)-\hV_y(\bx_i,\bx_j,r)|$ for an $(i,j)$ pair. We prove (\ref{eq.hV}) by deriving a union bound to show that if $n$ is large enough, then with high probability, every $V_f(\bx_{i},\bx_{j})$ is small for all the $(i,j)$ pairs satisfying $\calA(i,j) = 1$.
  Lemma \ref{lem.hV} implies that, for any $\alpha>0$, $$
  \PP\left(V_f(\bx_{i_k},\bx_{j_k}) > \alpha+\alpha_0+3\sigma^2\right) \leq 4n^{-\nu}.
  $$
  Then we have
  \begin{align*}
  & \PP\left(\bigcap_{k=1}^{\hn_{\alpha}}\left[V_f(\bx_{i_k},\bx_{j_k})\leq \alpha+ \alpha_0+3\sigma^2\right]\right)\\
    = &1-\PP\left(\bigcup_{k=1}^{\hn_{\alpha}}\left[V_f(\bx_{i_k},\bx_{j_k})> \alpha+ \alpha_0+3\sigma^2\right]\right)\\
    \geq& 1-\sum_{k=1}^{\hn_{\alpha}}\PP\left(V_f(\bx_{i_k},\bx_{j_k})> \alpha+ \alpha_0+3\sigma^2\right)\\
   \geq& 1-4\hn_{\alpha}n^{-\nu}\\
   \geq & 1-2n^{-(\nu-1)}.
  \end{align*}

\noindent\textbf{$\bullet$ Proof of (\ref{eq.hn}).} 
%
%
%
%
%
In Algorithm \ref{alg1}, we say two points $\bx_i,\bx_j$ are connected if Algorithm \ref{alg1} outputs $\calA(\bx_i,\bx_j) = 1$.
Algorithm \ref{alg1} guarantees that if $\bx_i$ and $\bx_j$ are connected, they can not be connected with any other points, so that the connected pairs are independent from each other. As a consequence, $\hn_{\alpha}$ is upper bounded by $n/2$.

To prove the lower bound of $\hn_{\alpha}$, we first discretize the domain $[-B,B]^d$ into small cubes along each direction with grid spacing $h = 2B (\log n/n)^{\frac {1}{2D}}$. Denote the $d$-dimensional cube with side length $h$ as $\widetilde\Omega_k,  k=1,\ldots,N$, and {then $N \leq (2B/h)^d = (n/\log n)^{\frac {d}{2D}}$.} Define $N$ prisms in $\RR^D$ as
$$\Omega_k = \{\bx \in \RR^D: \Phi^\top \bx \in \widetilde\Omega_k\},\ k =1,\ldots,N. $$
For any $\Omega_k$ and any $\bx_i,\bx_j \in \Omega_k$, we next show $\hV_y(\bx_i,\bx_j,r) \le \alpha$ with high probability. If $\bx_i,\bx_j \in \Omega_k$, $\|\Phi^\top \bx_i - \Phi^\top \bx_j\| \le \sqrt{d} h$ and by Lemma \ref{lemphix}, $V_f(\bx_i,\bx_j) \le d C_g^2 h^2 =  4d C_g^2 B^2 (\log n /n)^{\frac 1 D}$. According to Lemma \ref{lem.hV},
$
\PP\left( |V_f(\bx_i,\bx_j)-\hV_y(\bx_i,\bx_j,r)|\leq \alpha_0+3\sigma^2\right) \geq 1-4n^{-\nu}
$. Therefore, for any $\bx_i,\bx_j \in \Omega_k$, we have
$$\PP\left\{\hV_y(\bx_i,\bx_j,r) > 4d C_g^2 B^2 \left(\frac{\log n}{n}\right)^{\frac 1 D}+\alpha_0+3\sigma^2 \right\} \le 4n^{-\nu}.$$
Denote $\#\Omega_k = \#\{\bx_i: \bx_i \in \Omega_k\}.$ Applying a union bound gives
\begin{align*}
&\PP\left\{\exists \Omega_k \text{ such that }\exists \bx_i ,\bx_j \in \Omega_k:\  \hV_y(\bx_i,\bx_j,r) > 4d C_g^2 B^2 \left(\frac{\log n}{n}\right)^{\frac 1 D}+\alpha_0+3\sigma^2 \right\}
\\
\le
&  \sum_{k}\PP\left\{ \exists \bx_i ,\bx_j \in \Omega_k:\  \hV_y(\bx_i,\bx_j,r) > 4d C_g^2 B^2 \left(\frac{\log n}{n}\right)^{\frac 1 D}+\alpha_0+3\sigma^2 \right\}
\\
\le &\sum_k  {\#\Omega_k \choose 2} 4n^{-\nu}
\le \sum_k 2 (\#\Omega_k)^2 n^{-\nu}
\le 2 \left(\sum_k \#\Omega_k\right)^2 n^{-\nu} \le 2n^{-(\nu-2)}.
\end{align*}
The equation above shows that, in all sets $\Omega_k, k=1,\ldots,N$, all pairs of points $\bx_i,\bx_j$ in each $\Omega_k$ satisfy $\hV_y(\bx_i,\bx_j,r) \le \alpha$ with probability no less than $1-2n^{-(\nu-2)}$.

 In Algorithm \ref{alg1}, two points $\bx_i,\bx_j$ are likely to be connected if $\hV_y(\bx_i,\bx_j,r) \le \alpha$.  Under the condition that in all the sets $\Omega_k, k=1,\ldots,N$, all pairs of points $\bx_i,\bx_j$ in each $\Omega_k$ satisfy $\hV_y(\bx_i,\bx_j,r) \le \alpha$, there is at most one point in each $\Omega_k$ that is not connected with other points in the output of Algorithm \ref{alg1}. Therefore, the number of connected pairs satisfies
$$\hn_\alpha \ge \frac 1 2 (n-N) = \frac 1 2 \left[n - \left(\frac{n}{\log n}\right)^{\frac{d}{2D}} \right] \ge \frac n 4$$
if $n$ is sufficiently large such that $2 (n/\log n)^{\frac{d}{2D}} \le n$.

\end{proof}

\section{Proof of Lemma \ref{lem.sqExpectation}}
\label{appen.lem.sqExpectation}
\begin{proof}[\bf Proof of Lemma \ref{lem.sqExpectation}]
$\EE(X^2)$ can be computed by the following intergral
  \begin{align*}
    \EE(X^2)
    &= {\int_{0}^{+\infty}2t\PP(X\geq t)dt }\\
    &=\int_0^{+\infty} 2t\min(1,Ae^{-\frac{at^2}{b+ct}})dt \\
    &\leq \int_0^{t_0} 2t dt+ \int_{t_0}^{t_1}2 tAe^{-\frac{at^2}{2b}} dt + \int_{t_1}^{+\infty} 2tAe^{-\frac{at}{2c}}dt\\
    &\leq t_0^2+ \frac{2Ab}{a}\left( e^{-\frac{at_0^2}{2b}}-e^{-\frac{at_1^2}{2b}}\right) + \frac{4c}{a}Ae^{-\frac{at_1}{2c}} \left(t_1 + \frac{2c}{a}\right)
  \end{align*}
  where $t_0,t_1$ are chosen such that $Ae^{-\frac{at_0^2}{2b}}=1, b=ct_1$, which gives $t_0^2=({2b}/{a})\log A$ and $t_1={b}/{c}$. The second equality is due to the fact that $Ae^{-\frac{at^2}{b+ct}}$ is a probability which is no larger than 1. Plugging $t_0$ and $t_1$ to the equation above gives rise to
  \begin{align*}
    \EE(X^2)&\leq \frac{2b\log A}{a} + \frac{2b}{a} -\frac{2Ab}{a}e^{-\frac{ab}{2c^2}} + \frac{4Ab}{a}e^{-\frac{ab}{2c^2}} + \frac{8Ac^2}{a^2} e^{-\frac{ab}{2c^2}}\\
    &= \frac{2b}{a}(\log A+Ae^{-\frac{ab}{2c^2}}+1)+2e^{-\frac{ab}{2c^2}}\frac{4Ac^2}{a^2}\\
    &\leq \frac{2b}{a}(\log A+A+1)+\frac{8Ac^2}{a^2}.
  \end{align*}
\end{proof}
\section{Proof of Lemma \ref{thm.polyConverge}}
\label{appen.lem.polyConverge}
Lemma \ref{thm.polyConverge} is based on the following Lemma \ref{lem.taylor} (a generalization of \cite[Lemma 11.1]{gyorfi2006distribution} in one dimension to the multidimensional case) and Lemma \ref{prop.nonparametric} (\cite[Theorem 11.3]{gyorfi2006distribution}) below, which are standard results in non-parametric statistics.
\begin{lem}{\cite[Lemma 11.1]{gyorfi2006distribution}}
\label{lem.taylor}
Suppose $g: [-B,B]^d\rightarrow\RR$ is $(s,C_g)$-smooth with $s=k+\beta$ for some $k \in \NN_0$ and $\beta\in(0,1]$.
  Let $g_k$ be the Taylor polynomial of $g$ at $\ba\in [-B,B]^d$ of degree $k$. Then
  $$
  |g(\bz)-g_k(\bz)|\leq \frac{C_gd^{k/2}\|\bz-\ba\|^{s}}{k!},
  $$
  for all $\bz$ near $\ba$.
\end{lem}
\begin{proof}[\bf Proof of Lemma \ref{lem.taylor}]
We denote the multi-index $\boldsymbol\alpha = (\alpha_1,\ldots,\alpha_d)$ and let $\boldsymbol\alpha ! = \alpha_1 ! \cdots \alpha_d !$, $\bz^{\boldsymbol\alpha}=z_1^{\alpha_1}\cdots z_d^{\alpha_d}$ for $\bz=(z_1,...,z_d)$. Denote the partial derivative $D^{\boldsymbol\alpha}:=\frac{\partial^{\alpha_1+\ldots+\alpha_d} }{\partial x_1^{\alpha_1}\cdots \partial x_d^{\alpha_d}}$.
The Taylor expansion of $g$ at $\ba$ gives
  \begin{align*}
    &g(\bz)- g_k(\bz)\\
     =& g(\bz)-g_{k-1}(\bz)-\sum_{|\boldsymbol\alpha|=k} \frac{D^{\boldsymbol\alpha}g (\ba)(\bz-\ba)^{\boldsymbol\alpha}}{\boldsymbol\alpha!}\\
     =&\sum_{|\boldsymbol\alpha|=k} \frac{k}{\boldsymbol\alpha!} \left[ \int_{0}^1 (1-t)^{k-1} D^{\boldsymbol\alpha}g(\ba+t(\bz-\ba))dt\right] (\bz-\ba)^{\boldsymbol\alpha} - \sum_{|\boldsymbol\alpha|=k} \frac{(D^{\boldsymbol\alpha}g)(\ba)(\bz-\ba)^{\boldsymbol\alpha}}{\boldsymbol\alpha!}\\
     =&\sum_{|\boldsymbol\alpha|=k} \frac{(\bz-\ba)^{\boldsymbol\alpha}}{\boldsymbol\alpha!}k  \int_0^1 \left[ (1-t)^{k-1}D^{\boldsymbol\alpha}g(\ba+t(\bz-\ba)) - (1-t)^{k-1} D^{\boldsymbol\alpha}g (\ba) \right] dt \\
    =& \sum_{|\boldsymbol\alpha|=k} \frac{(\bz-\ba)^{\boldsymbol\alpha}}{\boldsymbol\alpha!}k  \int_0^1 (1-t)^{k-1}\left[D^{\boldsymbol\alpha}g(\ba+t(\bz-\ba)) - (D^{\boldsymbol\alpha}g)(\ba)\right]dt.
      \end{align*}
From this, and the assumption that $g$ is $(s,C_g)$ smooth, one obtains
   \begin{align*}
    &\left| g(\bz)- g_k(\bz) \right| \\
     \leq &\sum_{|\boldsymbol\alpha|=k} \frac{|(\bz-\ba)^{\boldsymbol\alpha}|}{\boldsymbol\alpha!}k  \int_0^1 (1-t)^{k-1}C_g\|\bz-\ba\|^{\beta}dt \\
    =& C_g\|\bz-\ba\|^{\beta} \sum_{|\boldsymbol\alpha|=k}\frac{|(\bz-\ba)^{\boldsymbol\alpha}|}{\boldsymbol\alpha!}\\
    =&  \frac{C_g\|\bz-\ba\|^{\beta}}{k!} \sum_{|\boldsymbol\alpha|=k}\frac{k!}{\boldsymbol\alpha!}|(\bz-\ba)^{\boldsymbol\alpha}| \\
     =&\frac{C_g\|\bz-\ba\|^{\beta}}{k!}\left(\sum_{i=1}^d |z_i-a_i|\right)^k \quad \mbox{by the multinomial theorem} \\
     \leq& \frac{C_g\|\bz-\ba\|^{\beta}}{k!}\|\bz-\ba\|_1^k\\
     \leq& \frac{C_g\|\bz-\ba\|^{\beta}}{k!}\left(\sqrt{d}\|\bz-\ba\|\right)^k\\
     =&\frac{C_gd^{k/2}\|\bz-\ba\|^{s}}{k!}.
  \end{align*}
\end{proof}

We next introduce a standard result (\cite[Theorem 11.3]{gyorfi2006distribution}) in nonparametric statistics.

\begin{lem}{\cite[Theorem 11.3]{gyorfi2006distribution}}
\label{prop.nonparametric}
  Let $\{\bz_i\}_{i=n+1}^{2n}$ be i.i.d. sampled from a probability measure $\mu$ such that $\mathrm{supp}(\mu)\subseteq[-B,B]^d$ and let $\{y_i\}_{i=n+1}^{2n}$ be sampled from the regression model
  $$y = h(\bz)+\zeta$$ where $\EE[ \zeta|\bz] = 0$.
Suppose $\|h\|_{\infty} < +\infty$ and $\sup_{\bz}\var(\zeta|\bz) \le\bar{\sigma}^2<\infty$.
  Let $\mathcal{F}$ be a linear space of functions from $\RR^d$ to $\RR$, and $\widetilde{h}$ be the estimator given by
  $$\widetilde h = T_{\|h\|_{\infty}} \bar h \quad \text{ where }\quad \bar h = \argmin_{p \in \calF} \frac 1 n\sum_{i=n+1}^{2n} |p(\bz_i) - y_i|^2.$$
   Then
  \begin{align}
     &\EE\int|\widetilde{h}(\bz)-h(\bz)|^2\mu(d\bz)\nonumber\\
   \leq  &c\cdot \max(\bar{\sigma}^2,\|h\|_{\infty}^2) \frac{ {\rm dim}(\mathcal{F})\log n}{n} +8 \inf_{p \in \calF} \int |p(\bz)- h(\bz)|^2 \mu(d\bz).
    \label{eq.nonparametric}
  \end{align}
  for some universal constant $c$.
\end{lem}
In (\ref{eq.nonparametric}), the mean squared error is decomposed into two terms: the first term captures the variance and the second term estimates the bias.
We consider piecewise polynomial approximation with order no more than $k$ such that $\mathcal{F}=\mathcal{F}_k$.
\begin{proof}[\bf Proof of Lemma \ref{thm.polyConverge}]
In order to apply Lemma \ref{prop.nonparametric}, we express the regression model in \eqref{eqmodelg} as
$$y_i = \underbrace{g(\bz_i)+\eta(\bz_i)}_{h(\bz_i)} +\underbrace{ \widetilde \xi_i - \eta(\bz_i)}_{\zeta_i}$$
with $\eta$ defined in \eqref{eqeta}. The function $h$ is bounded: $\|h\|_\infty \le \|g\|_\infty + \|\eta\|_\infty \le M + C_gB\|\tPhi-\Phi\|$. The noise satisfies $\EE[\zeta|\bz] = 0$ and
 \begin{align*}
 \sup_{\bz = \tPhi^\top \bx}\var( \zeta |\bz) &= \sup_{\bz = \tPhi^\top \bx} \var\left(\xi +g(\Phi^\top\bx)-g(\tPhi^\top\bx) -  \EE[g(\Phi^\top\bx)-g(\tPhi^\top\bx) | \tPhi^\top \bx] \right)
 \\
 &= \sup_{\bz = \tPhi^\top \bx} \var(\xi) +\var[g(\Phi^\top\bx)-g(\tPhi^\top\bx)  | \tPhi^\top \bx  =\bz ]
 \\
 & \le \sigma^2  +C_g^2 B^2 \|\tPhi -\Phi\|^2= \sigma^2  +\bnoise,
 \end{align*}
 where $\bnoise$ is defined in (\ref{eq.bnoise}).
We apply Lemma \ref{prop.nonparametric} with
 $\mathcal{F}=\mathcal{F}_k$, and then
  \begin{align*}
    &\EE \int|\widetilde{g}(\bz)-g(\bz)-\eta(\bz)|^2\mu(d\bz)\\
    \leq& c\cdot \max(\sigma^2  +\bnoise,2M^2+2\bnoise) \frac{ {\rm dim}(\mathcal{F}_k)\log n}{n}
    +8 \inf_{p \in \calF_k} \int |p(\bz)- g(\bz)-\eta(\bz)|^2 \mu(d\bz).
  \end{align*}
  for some universal $c$. Here ${\rm dim}(\mathcal{F}_k)$ is the dimension of $\calF_k$ and hence ${\rm dim}(\mathcal{F}_k) = \binom{d+k}{d}K^d$. Let $g_k$ be the piecewise Taylor polynomial of $g$ with degree no more than $k = \lceil s\rceil-1$ on the partition of $[-B,B]^d$ into $K^d$ cubes, where the Taylor expansion is at the center of each cube. The second term can be estimated from Lemma \ref{lem.taylor} as
  \begin{align*}
  &\inf_{p\in \mathcal{F}_k} \int |p(\bz)-g(\bz)-\eta(\bz)|^2\mu(d\bz)
  \le  \sup_{\bz} |g_k(\bz)-g(\bz)-\eta(\bz)|^2\\
  \leq& 2\left(\frac{C_gd^{k/2}}{k!}\left(\frac{Bd^{1/2}}{K}\right)^s\right)^2 + 2 \|\eta\|_\infty^2
   \le \frac{2C_g^2B^{2s}d^{s+k}}{(k!)^2K^{2s}} + 2\bnoise.
  \end{align*}
 Therefore
  \begin{align*}
    &\EE \int|\widetilde{g}(\bz)-g(\bz)-\eta(\bz)|^2\mu(d\bz)\\
    \leq& c\cdot \max(\sigma^2  +\bnoise,2M^2+2\bnoise) \frac{ \binom{d+k}{d}K^d\log n}{n} +\frac{2C_g^2B^{2s}d^{s+k}}{(k!)^2K^{2s}} + 2\bnoise.
  \end{align*}
  We have
  \begin{align*}
    &\EE \int|\widetilde{g}(\bz)-g(\bz)|^2\mu(d\bz)=\EE \int|\widetilde{g}(\bz)-g(\bz)-\eta(\bz)+\eta(\bz)|^2\mu(d\bz)\\
    \leq& \EE \left(2\int|\widetilde{g}(\bz)-g(\bz)-\eta(\bz)|^2\mu(d\bz)+2\int|\eta(\bz)|^2\mu(d\bz)\right)\\
    \leq& 2c\cdot \max(\sigma^2  +\bnoise,2M^2+2\bnoise) \frac{ \binom{d+k}{d}K^d\log n}{n} +\frac{4C_g^2B^{2s}d^{s+k}}{(k!)^2K^{2s}} + 6\bnoise.
  \end{align*}
\end{proof}

\commentout{
\section{Proof of \eqref{phiproj}}
\label{appen.lem.canonicalVector}
\begin{proof}[\bf Proof of \eqref{phiproj}]
We have
\begin{align*}
  \|\hPhi-\Phi\|^2&\leq \sum_{i=1}^d \|\hbphi_i-\bphi_i\|^2\\
  &=\sum_{i=1}^d \hbphi_i^\top\hbphi_i+\bphi_i^\top\bphi_i-\hbphi_i^\top\bphi_i-\bphi_i^\top\hbphi_i\\
  &=\sum_{i=1}^d 2-2\cos(\theta_i) \leq \sum_{i=1}^d 2(1-\cos^2(\theta_i))\\
  &=\sum_{i=1}^d 2\sin^2(\theta_i) \leq 2d\sin^2(\theta_1)\\
  &= 2d \|\proj_{\hPhi}-\proj_{\Phi}\|^2.
\end{align*}
The last equality comes from \cite[Theorem 5.5]{stewart1990matrix}.
\end{proof} 
}

\end{document}